\documentclass[11pt]{article}
\usepackage{amssymb,amsmath,amsthm,amsfonts, commath, mathrsfs, mathtools, secdot}
\usepackage[T1]{fontenc}		
\usepackage[a4paper, margin=1in]{geometry}

\usepackage{bbm}
\usepackage{extarrows}
\usepackage{enumerate}
\usepackage{multicol}	
\usepackage{fge}
\usepackage{oldgerm}
\usepackage{hyperref} 
\usepackage{paralist} 
\usepackage{accents}

\usepackage{tikz}
\usetikzlibrary{datavisualization}
\usetikzlibrary{datavisualization.formats.functions}
\usepackage{pgfplots}
\usepackage{tkz-fct}
\pgfplotsset{every axis/.append style={
    axis x line=middle,    
    axis y line=middle,    
    axis line style={<->}, 
    xlabel={$x$},          
    ylabel={$y$},          
    },
    cmhplot/.style={color=blue,mark=none,line width=1pt,<->},
    soldot/.style={color=blue,only marks,mark=*},
    holdot/.style={color=blue,fill=white,only marks,mark=*},
}
    
\tikzset{>=stealth}

\usepackage[natbib=true, backend=biber, style=numeric]{biblatex}
\usepackage{csquotes}
\sectiondot{subsection}
\usepackage{xcolor}

\usepackage{color}

\newtheoremstyle{Assump}%
  {3pt}
  {3pt}
  {\itshape}
  {}
  {\bfseries}
  {.}
  {.5em}
  {\thmname{#1} \thmnumber{#2} \thmnote{\normalfont#3}}

\newtheorem{theorem}{Theorem}[section]
\newtheorem{lem}[theorem]{Lemma}
\newtheorem{prop}[theorem]{Proposition}
\newtheorem{cor}[theorem]{Corollary}
\theoremstyle{definition}
\theoremstyle{definition}
\theoremstyle{definition}\newtheorem{remark}[theorem]{Remark}
\theoremstyle{definition}\newtheorem{defi}[theorem]{Definition}

\numberwithin{equation}{section}

\newcommand{\N}{\mathbb{N}}
\newcommand{\R}{\mathbb{R}}

\newcommand{\Q}{\mathbb{Q}}
\newcommand{\D}{{\bf D}}											
\newcommand{\C}{{\bf C}}											
\newcommand{\B}{\mathscr{B}}
\newcommand{\F}{\mathbb{F}}											
\newcommand*\diff{\mathop{}\!\mathrm{d}}								
\newcommand{\E}{\mathbb{E}}											
\newcommand{\Pro}{\mathbb{P}}										
\newcommand{\indep}{\perp \!\!\! \perp}									
\newcommand{\oland}{\mbox{\scriptsize $\mathcal{O}$}}						
\newcommand{\Oland}{\mathcal{O}}										
\newcommand{\ind}{\operatorname{\mathbf{1}}}							
\renewcommand{\max}[1]{\underset{#1}{\operatorname{max}}\;}				
\renewcommand{\min}[1]{\underset{#1}{\operatorname{min}}\;}					
\renewcommand{\sup}[1]{\underset{#1}{\operatorname{sup}}\;}					
\renewcommand{\setminus}{\mathbin{\fgebackslash}}						
\newcommand{\intersection}[2]{\overset{#2}{\bigcap\limits_{#1}}\;}				
\newcommand{\conv}[2]{\; \xrightarrow[ {#1}]{#2} \; }						
    
\newcommand{\dJ}{\diff_{\operatorname{J_1}}}								
\newcommand{\dM}{\diff_{\operatorname{M_1}}}	                

\title{Functional weak convergence of stochastic integrals for \\ moving averages and continuous-time random walks}
\author{}
\date{}

\author{Andreas S\o jmark\thanks{London School of Economics, Department of Statistics, London, WC2A 2AE, UK. \texttt{a.sojmark@lse.ac.uk}} \,\; and \, Fabrice Wunderlich\thanks{University of Oxford, Mathematical Institute, Oxford, OX2 6GG, UK. \texttt{wunderlich@maths.ox.ac.uk}}  }

\begin{document}
\maketitle

\vspace{-10pt}

\begin{abstract}
There is an extensive theory of weak convergence for moving averages and continuous-time random walks (CTRWs) with respect to Skorokhod's $M_1\hspace{-0.1pt}$ and $J_1\hspace{-0.1pt}$ topologies. Here we address the fundamental question of how this translates into functional limit theorems in the $M_1\hspace{-0.1pt}$ or $J_1\hspace{-0.1pt}$ topology for stochastic integrals driven by these processes. As an important application, we provide weak approximation results for general SDEs driven by time-changed L\'evy processes. Such SDEs and their associated fractional Fokker--Planck--Kolmogorov equations are central to models of anomalous diffusion in statistical physics. Our results yield a rigorous functional characterisation of these as continuum limits of the underlying models driven by CTRWs. With regard to strictly $M_1\hspace{-0.1pt}$ convergent moving averages and correlated CTRWs, it turns out that the convergence of stochastic integrals can fail decidedly and fundamental new challenges arise compared to the $J_1\hspace{-0.1pt}$ setting. Nevertheless, we identify natural classes of integrand processes for which there is $M_1\hspace{-0.1pt}$ convergence of the stochastic integrals. We also show that these results are flexible enough to yield functional limit theorems in the $M_1\hspace{-0.1pt}$ topology for certain stochastic delay differential equations driven by moving averages.
\end{abstract}


\section{Introduction} 
In view of the classical central limit theorem and its functional extensions, the concept of diffusion is a remarkably robust modelling paradigm. With the appropriate square root scaling, it gives the correct macroscopic description of any random walk whose i.i.d.~jumps are of finite variance. Yet, many phenomena across the natural and social sciences involve heavy-tailed power laws that lead to different notions of anomalous---or fractional---diffusion for the macroscopic movement of particles, be they actual particles or some quantities identified as such. This deviation from classical diffusion lies at the heart of a growing field of research known as fractional calculus which sits at the interface of mathematical physics, probability, and partial differential equations.

Befittingly, anomalous diffusion, too, is underpinned by a robust class of limit theorems for random walks, now with jumps of infinite variance. In fact, not only jumps of infinite variance are relevant, but also infinite mean waiting times between the jumps, leading to the concept of continuous-time random walks (CTRWs). For details on this, see Section \ref{sect:CTRW_and_GD}. As long as the jumps and waiting times are, in a certain sense, attracted to given stable laws, the appropriate scaling limit will be described by a fractional equivalent of the heat equation for classical diffusion. Here a fractional derivative in time captures an element of sub-diffusion due to infinite-mean power laws of the waiting times, while fractional derivatives in space capture elements of super-diffusion due to infinite-variance power laws of the jumps.

For an excellent introduction to fractional calculus and anomalous diffusion from a probabilistic viewpoint, we refer to the monograph of Meerschaert \& Sikorskii \cite{book_Meerschaert}. In particular, \cite[Ch.~4]{book_Meerschaert} introduces the underlying functional limit theory for CTRWs that we start from in this work. The main definitions and results in the literature are recalled in Section \ref{sect:CTRW_and_GD}.

As one might expect, stochastic analysis plays a central role in the study of anomalous diffusion, see e.g.~the treatments of stochastic calculus for CTRWs and subordinated L\'evy processes in \cite{Germano_etal, janicki_weron, Kallsen_shiryaev, Kobayashi,Nane_Ni}. The latter arise from CTRW models in the continuum limit (Section \ref{subsect:CTRW_limit}). Therefore, it becomes a fundamental question to understand the conditions under which we can relate stochastic integrals driven by CTRWs to those driven by their scaling limits. Within this work, we explore this question in the sense of functional weak convergence on Skorokhod space with respect to Skorokhod's $J_1$ and $M_1$ topologies.

\subsection{Anomalous diffusion and convergence of stochastic integrals}\label{subsect:anomalous_diffusion_stoch_int}

In the physics literature, models of anomalous diffusion are often formulated in terms of fractional Fokker--Planck--Kolmogorov equations. Akin to the fractional versions of the heat equation briefly discussed above, they take the general form
\begin{equation}\label{eq:frac_FKP}
	D^{\beta}_t p(t,x) = \mathcal{A}_xp(t,x),\quad p(0,x)=p_0,
\end{equation}
where $D^\beta_t$ is a fractional Caputo derivative of order $\beta \in(0,1)$ and $ \mathcal{A}_x=\mathcal{A}(x,D_x)$ is a pseudo-differential operator with a given symbol $\Psi(x,\xi)$.

For many applications, the Cauchy problem \eqref{eq:frac_FKP} is intimately linked to an underlying CTRW model and may be derived in that way, as in the pioneering work of Metzler, Barkai \& Klafter \cite{metzler99}. For a general overview, we refer to Metzler \& Klafter's influential survey papers \cite{metzler_klafter1, metzler_klafter2}. At the interface of PDEs and probability theory, there is moreover a significant interest in establishing and utilizing stochastic representations of fractional Cauchy problems such as \eqref{eq:frac_FKP}, see e.g.~\cite{beaumer_nane_meerschart, hahn_fkp1, hahn_earlier_article, Magdiarz_SDE2, Meerschaert_SDE} and the general treatment in \cite[Ch.~4]{ascione}. Beyond their mathematical interest, these results provide a rigorous connection between fractional Fokker--Planck--Kolmogorov formulations of anomalous diffusion and single-particle tracking experiments, as e.g.~surveyed in \cite{Metzler_particle_track}. Furthermore, they open up for Monte Carlo methods in the numerical approximation of solutions to fractional PDEs.

Stochastic representations of \eqref{eq:frac_FKP}, and variants thereof, generally involve SDEs driven by time-changed L\'evy processes, where the symbol $\Psi(x,\xi)$ of $\mathcal{A}_x$ corresponds to the characteristic exponent of the parent L\'evy process, while $\beta$ is the index of stability for a stable subordinator whose generalised inverse yields the time-change. The study of such SDEs is of independent interest. Several works are concerned with their qualitative properties \cite{ Nane_Ni, Nane_Ni2} and their numerical analysis \cite{gajda_magdiarz, jin_kobayashi, jum_kobayashi}. Likewise, there is a significant interest in understanding subordinated L\'evy processes on their own, for example motivated by the study of CTRWs \cite{meerschaert_schilling} or related to applications in actuarial mathematics \cite{constantinescu} and finance \cite{carr_wu}.

In their monograph \cite{hahn}, Hahn, Kobayashi \& Umarov describe a unifying paradigm for the study of anomalous diffusion, bridging analytical and probabilistic approaches. They highlight the driving process $X$ as the central object---in the case of \eqref{eq:frac_FKP}, a subordinated L\'evy process---and identify the following three pillars: (i) limit theorems for the underlying CTRW constructions, (ii) analysis of SDEs driven by $X$, and (iii) analysis of the corresponding fractional Fokker--Planck--Kolmogorov equations. With regard to (i) and (ii), Section 6.5 of \cite[Ch.~6]{hahn} starts by stressing that an important but less studied question is that of functional weak convergence for stochastic integrals driven by CTRWs. Furthermore, Remark 6.4 of \cite[Ch.~6]{hahn} discusses the interest in applying this machinery to establish functional limit theorems for SDEs driven by CTRWs, conceivably leading to SDEs driven by time-changed L\'evy-processes in line with the study of such SDEs in \cite[Chs.~6-7]{hahn}. These points will all be addressed by the results that we develop herein.

\subsection{Main contributions}

This work provides a systematic treatment of functional weak convergence for stochastic integrals driven by moving averages and CTRWs, when these integrators are weakly convergent for Skorokhod's $J_1$ or $M_1$ topologies. In the $J_1$ case, we obtain rather universal statements that generalize earlier works, while the $M_1$ case leads to tailored results which, as far as the authors are aware, are completely new in the literature.

Based on the general theory developed in \cite{jakubowskimeminpages, kurtzprotter}, several works have studied the functional weak convergence of stochastic integrals driven by particular CTRWs. This is for example the case in \cite{davis, paulaskas_rachev} with a view towards problems in statistical inference, and in \cite{Burr, scalas} as well as \cite[Sect.~6.5]{hahn} motivated by the analysis of scaling limits for anomalous diffusion models as discussed above. The existing results, however, are confined to the $J_1$ topology and involve quite restrictive assumptions on the integrands and integrators, relying on arguments that are specific to the particular settings. Our treatment of stochastic integral convergence in the $J_1$ case goes much further, allowing for weakly $J_1$ convergent CTRW integrators at the full level of generality in which they are typically considered in the literature.

As outlined in Section \ref{sect:CTRW_and_GD}, important classes of moving averages \cite{avram} and so-called correlated CTRWs \cite{meerschaert2} converge weakly in the $M_1$ but not the $J_1$ topology. Already in  the fundamental work \cite{jakubowski}, which extended the general limit theory for stochastic integrals beyond the $J_1$ topology, the $M_1$ convergence of moving averages from \cite{avram} was singled out by Jakubowski to underline that \emph{`for some naturally arising integrators the requirement of convergence in the usual Skorohod topology may be too strong'} \cite[p.~2142]{jakubowski}. 
Nevertheless, the question of if and when there is convergence of stochastic integrals driven by moving averages has remained unaddressed. Indeed, the conditions in \cite{jakubowski} turn out \emph{not} to apply to this class of integrators in general, and it is far from trivial what can be said. We are able to devise a generalised framework that rectifies this situation and thus confirms that one can have weak $M_1$ convergence of the corresponding stochastic integrals under some additional but natural requirements on the interplay between the integrands and integrators.

Finally, we apply our results to stochastic differential equations. Firstly, we fill a gap in the theory of anomalous diffusion identified in \cite[Rem.~6.4]{hahn}, by deriving a functional limit theorem for SDEs driven by $J_1$ convergent CTRWs. Secondly, we show that there is weak $M_1$ convergence for certain stochastic delay differential equations driven by moving averages. Related applications could be functional limit theorems for stochastic partial differential equations subject to anomalous diffusion, e.g., stochastic perturbations of \eqref{eq:frac_FKP} as in \cite{kim_2019, kim_park_2023}, provided they are driven by CTRWs. For a general discussion of stable Lévy processes as the driving noise in physical models motivated by stable limit theorems, see~\cite{klafter2}.

\subsection{Other applications}

We wish to briefly highlight three different areas of application where the results of this paper are relevant. For brevity, we defer the discussion of concrete problems to a separate note \cite{andreasfabrice_applicationpaper}. Here, we only outline the general ideas.

An option pricing framework for log-prices given by subordinated stable L\'evy processes (with tempering) was recently developed in \cite{jack}, motivated by tick-by-tick CTRW models originating in \cite{scalas_finance}. Regarding the performance of trading strategies and hedging of risk, this raises the question of when there is convergence of the financial gains, given by left continuous trading strategies integrated against the price processes. The classical case of random walks tending to geometric Brownian motion was addressed in \cite{duffie_protter}, while our results allow for a treatment of $J_1$ or $M_1$ convergence in the setting of \cite{jack}. In this regard, we also mention that the recent work \cite{ait-sahalia_jacod} defines a given continuum asset price model to be compatible with a tick-by-tick model if the latter is weakly $M_1$ convergent to the former, but there has been no discussion of the weak convergence of the corresponding financial gains.

In actuarial mathematics, ruin theory with risky investments leads to the study of certain SDEs driven by L\'evy processes which give the value of the reserves \cite{Paulsen_AAP}. The case where these SDEs are driven by Brownian motion has been shown to arise as a scaling limit of models based on compound Poisson processes \cite{Paulsen_Gjessing_Advances}. The latter are simple examples of CTRWs. However, applications naturally call for general CTRW models, both for the arrival of claims, and, as per the previous paragraph, for the prices of the risky investments. In particular, the impact of correlated claim sizes or price innovations is important to understand. Our framework may serve as the starting point for an analysis of the weak $J_1$ or $M_1$ convergence of general CTRW based models to the type of SDEs studied in \cite{Paulsen_AAP}.

Finally, functional limit theorems for stochastic integrals driven by random walks and moving averages, converging to stable processes, play an important role in the theory of statistical inference for cointegrated processes with heavy tails \cite{paulaskas_rachev, fabozzi_comment}. The former work \cite{paulaskas_rachev} provided the first rigorous treatment for $J_1$ convergent random walks, while \cite{fabozzi_comment} later pinpointed some imprecisions in the literature and emphasised the more challenging nature of moving averages. Our results relax the assumptions on the i.i.d.~innovations in \cite{paulaskas_rachev} and  could allow for an analysis of linear correlation structures with $M_1$ convergence.

\subsection{Overview of the paper} 

Section \ref{sect:CTRW_and_GD} gives the precise definitions of moving averages and CTRWs, and also covers the associated limit theory in the $M_1$ and $J_1$ topologies. Next, Section \ref{subsect:stoch_int_conv} recalls some central concepts from \cite{andreasfabrice_theorypaper} concerning the functional weak convergence of stochastic integrals; in particular, the notion of good decompositions for the integrators (Definition \ref{defi:GD}) which will play a prominent role throughout.

In Section \ref{subsec:convergence_uncorrelated_CTRWs}, we show that uncorrelated, possibly coupled, CTRWs have good decompositions which yields a universal result on the weak $J_1$ convergence of stochastic integrals driven by such CTRWs (Theorem \ref{thm:result_int_conv_uncorrelated_CTRWs}). This provides a positive answer to the open problem described by Hahn, Kobayashi \& Umarov in the first part of \cite[Rem.~6.4]{hahn}. Proceeding from there, Section \ref{subsec:Correlated_CTRWs_do_not_blend_in} derives an analogous result on weak $M_1$ convergence in the case of moving averages and correlated CTRWs with a finite variation scaling limit (Theorem \ref{thm:result_int_conv_correlated_CTRWs_alpha<1}).

When the scaling limits display infinite variation, there can be a fundamental failure of weak convergence for stochastic integrals driven by strictly $M_1$ convergent moving averages or CTRWs. This comes down to a lack of good decompositions, unlike the uncorrelated case discussed above. Thus, while the work in Section \ref{sect:GD_CTRW} takes place within the framework of \cite{andreasfabrice_theorypaper}, entirely new ideas are needed to generalise this framework in the next part of the paper. Specifically, we must resort to a much more tailored analysis that exploits the structure of the classes of integrators that we are interested in. We conclude Section \ref{sect:GD_CTRW} with Section \ref{sec:quantities_and_approach_needed_for_integral_convergence_proof} which provides a brief overview of how we approach this problem in relation to \cite{andreasfabrice_theorypaper}.

The overarching principle in Section \ref{sec:generalised_framework} is to adjust the notion of good decompositions by including a suitable remainder term, which has to be well-behaved with respect to the integrands. Section \ref{subsec:Direct_control_of_variation} approaches this based on direct control of the variation of the remainder (Theorem \ref{thm:Lipschitz_GDmodCA}), while Section \ref{subsec:Control_through_independence} explores a procedure based on a certain independence condition between the integrands and the `future' of the remainder (Theorem \ref{thm:Integral_converg_under_GDmodCI}). We suspect that both approaches may be more widely applicable, but our focus is on moving averages and correlated CTRWs. In turn, Section \ref{subsec:Final_results_for_CTRWs} presents two general results on the weak $M_1$ convergence of stochastic integrals driven by those processes (Theorems \ref{thm:result_int_conv_CTRWs_lipschitz_integrands} and \ref{thm:result_int_conv_CTRWs_with_independence_cond}).

We end the paper by applying the above analysis to address convergence questions for SDEs and stochastic delay differential equations (SDDEs). Section \ref{sec:SDE} confirms that SDEs driven by uncorrelated CTRWs converge weakly in $J_1$ to the corresponding SDEs driven by subordinated L\'evy processes (Theorem \ref{thm:SDE_approx_result}), while Section \ref{sect:SDDE} shows how the independence framework of Section \ref{subsec:Control_through_independence} can be used to obtain $M_1$ convergence for certain SDDEs driven by strictly $M_1$ convergent moving averages (Theorem \ref{thm:5.18}).

All proofs are postponed to Section \ref{sec:proofs}, where they appear in the same order in which the corresponding results are presented in Sections \ref{sect:GD_CTRW}--\ref{sect:SDEs_SDDEs}.

\section{Moving averages, CTRWs, and their scaling limits}\label{sect:CTRW_and_GD}

In this section, we recall the precise definitions of moving averages and CTRWs with appropriate rescaling, and we then discuss the functional CLTs that have been established for these processes.

Throughout the paper, we shall write $\D_{\R^d}[0,\infty)$ for the Skorokhod space consisting of all c\`adl\`ag paths $x:[0,\infty)\rightarrow  \mathbb{R}^d$, for a given dimension $d\geq 1$. Moreover, we shall use $\dJ$ and $\dM$ to refer to a fixed choice of metrics that induce, respectively, the $J_1$ and $M_1$ topologies on this space. For details on these topologies, we refer to \cite{whitt}.

\subsection{Moving averages}\label{subsect:Moving_averages} 
We begin by recalling that a \emph{moving average} (suitably re-scaled) is a continuous-time stochastic process of the form
\begin{equation} X^n_t \; :=  \frac{1}{n^{\frac 1 \alpha}}\; \sum_{k=1}^{\lfloor nt \rfloor} \, \zeta_k,\qquad t\geq0,\qquad n\geq1, \label{eq:MovAv}
\end{equation}
where the innovations are given by
\begin{equation}\label{eq:linear_process_innovations} \zeta_i \; := \; \sum_{j=0}^\infty \; c_j \theta_{i-j}, \qquad i \ge 1,
\end{equation}
for an i.i.d.~sequence $\{ \theta_k : -\infty < k < \infty \}$ of $\R^d$-valued random variables (on a common probability space $(\Omega, \mathcal{F}, \Pro)$). Here, the $\theta_k$ are assumed to be in the \emph{normal domain of attraction} of a non-degenerate strictly $\alpha$-stable random variable $\tilde{\theta}$ with $0< \alpha< 2$, i.e. that there is weak convergence
\begin{align} n^{-\frac{1}{\alpha}} \,\left(\theta_1 \; + \; \hdots\; + \; \theta_n \right) \; \Rightarrow \; \tilde \theta \label{eq:norm_dom_of_attraction}
\end{align}
in $\mathbb{R}^d$ (see \cite[p.~114/115]{whitt} or \cite[p.~312/313]{feller}). Recall that $\tilde{\theta}$ is said to have a strictly $\alpha$-stable law on $\mathbb{R}^d$ if there exist i.i.d.~copies $\tilde{\theta}_1,\tilde{\theta}_2,... $ of $\tilde{\theta}$ such that 
$\tilde{\theta}_1  +  \hdots +  \tilde{\theta}_n \sim  n^{\frac1\alpha} \, \tilde{\theta}$ for all $ n\geq 1.$
For $\alpha=2$, \eqref{eq:norm_dom_of_attraction} holds instead with $\tilde\theta$ being a non-degenerate Gaussian random variable. Throughout, we assume $c_j \geq0$ for all $j$, and we require that $\sum_{j=0}^\infty c_j^\rho < \infty$, for some $0<\rho<\alpha$. The latter ensures that the series \eqref{eq:linear_process_innovations} converges in $L^\rho(\Omega, \mathcal{F}, \Pro)$---and, in fact, almost surely (see \cite{avram, kawata}). 

Throughout, we shall assume that, if $1<\alpha\le 2$, then either $c_j=0$ for all but finitely many $j\ge 0$ or the sequence $(c_j)_{j\ge 0}$ is monotone and $\sum_{j=0}^\infty c_j^\rho <\infty$ for some $\rho<1$. 

For a \emph{zero-order moving average} (i.e., $c_0>0$ and $c_j=0$ for all $j\ge 1$), it is a classical result of Skorokhod \cite{skorokhod} that $X^n \Rightarrow c_0 Z$ in $(\D_{\R^d}[0,\infty), \, \dJ)$, where $Z$ is a Brownian motion if $\alpha=2$ or an $\alpha$-stable Lévy process (with $Z_1 \sim \tilde \theta$) if $0<\alpha<2$. For zero-order moving averages and uncorrelated CTRWs, we shall absorb $c_0$ into the innovations and hence take $c_0=1$. Avram \& Taqqu \cite{avram} studied functional convergence in the general case, showing that $ X^n \Rightarrow (\sum_{j=0}^\infty c_j) Z$ on $(\D_{\R^d}[0,\infty), \, \dM)$ if $\E[\theta_1]=0$ when $1<\alpha <2$, or if the law of $\theta_1$ is symmetric when $\alpha=1$. Throughout, we assume this condition to be satisfied whenever we refer to moving averages, except for zero-order moving averages. Unless it is stated otherwise, this condition does \emph{not} apply to the continuous-time random walks introduced below. Finally, \cite{avram} also showed that, for $0<\alpha<2$, if not just $c_0>0$ but also $c_j>0$ for at least one $j\ge 1$, then the convergence \emph{cannot} be strengthened to hold in the $J_1$ topology.

\subsection{Continuous-time random walks}\label{subsect:CTRW_limit} CTRWs are generalisations of moving averages \eqref{eq:MovAv}, allowing for random waiting times $J_i$ (with possibly infinite mean) in between jumps. More precisely, let $J_1, J_2,...$ be strictly positive i.i.d.~random variables in the normal domain of attraction of a $\beta$-stable random variable with $\beta \in (0,1)$, defined on the same probability space as the $\theta_k$ above. A CTRW (suitably re-scaled), then takes the form
\begin{align} X^n_t \; := \;  \frac{1}{n^{\frac \beta \alpha}} \sum_{k=1}^{N_{nt}} \, \zeta_k,\qquad N_{nt} \; := \; \max{} \{ m \ge 0 \; : \; L_m \le nt \}, \label{defi:CTRW}
\end{align}
with $L_m:= J_1+\hdots+J_m$ and $L_0\equiv 0$. Note that $N_{nt}$ gives the number of jumps up until time $nt$. In the literature two basic types are distinguished: a CTRW is said to be \emph{uncorrelated} if $c_j=0$ for all $j\ge 1$ and \emph{correlated} otherwise. Further, we will also use the term \emph{finitely correlated} whenever there exists $\mathcal J\ge 1$ such that $c_j=0$ for all $j> \mathcal J$.

If the sequences $(J_i)_{i\geq 1}$ and $(\zeta_k)_{k\geq 1}$ are independent, the CTRW \eqref{defi:CTRW} is typically called \emph{uncoupled}. For this setting, Becker-Kern, Meerschaert \& Scheffler \cite{Becker-Kern} as well as Meerschaert, Nane \& Xiao \cite{meerschaert2} extended the results of \cite{avram} for moving averages, showing that CTRWs exhibit a similar scaling-limit behaviour. That is, for $0<\alpha\le 2$, we have
\begin{align} n^{-\beta} N_{n\bullet} \; \Rightarrow \; D^{-1} \label{eq:time_convergence_to_subordinator}
\end{align}
on $(\D_{\R}[0,\infty), \dJ)$ as well as
\begin{equation}\label{eq:CTRW_$M_1$_conv}
	\; X^n \; \; \Rightarrow \; \; \Bigl(\,\sum_{j=0}^\infty c_j\Bigr)Z_{D^{-1}}
\end{equation}
on the $M_1$ Skorokhod space $(\D_{\R^d}[0,\infty), \dM)$, where $Z$ is either a Brownian motion ($\alpha=2$) or an $\alpha$-stable Lévy process ($0<\alpha<2$), and the process 
$$ D^{-1}_t \; := \; \operatorname{inf} \, \left\{ \, s\ge 0 \, : \, D_s \, > \, t\, \right\}$$
is the generalised inverse of a $\beta$-stable subordinator $D$ with $\beta \in (0,1)$. Since $D$ is strictly increasing, $D^{-1}$ is a continuous process. Based on the arguments of \cite{avram} for moving averages with $0<\alpha<2$, it was shown in \cite{meerschaert2} that the convergence \emph{cannot} be strengthened to $J_1$ if there is $j\neq i$ such that $c_i,c_j>0$. If $c_0>0$ is the only non-zero constant, then the $M_1$ convergence \eqref{eq:CTRW_$M_1$_conv} was first shown in \cite{Becker-Kern} and later improved to hold in the $J_1$ topology \cite{henry_straka}. 

One can also consider \emph{coupled} CTRWs. Let \eqref{defi:CTRW} be uncorrelated with $c_0=1$, then the CTRWs are said to be \emph{coupled} if the pairs of associated innovations and jump times $(\zeta_k,J_k)_{k\geq 1}$ constitute an i.i.d. sequence, while however, we allow for dependence of the components $\zeta_k$ and $J_k$ of each pair. In this framework, if $\sum_{k=1}^{\lfloor nt \rfloor}(n^{-1/\alpha} \zeta_k, n^{-1/\beta} J_k) \Rightarrow (Z,D)$ with $Z$ and $D$ as above, the results of \cite{henry_straka,jurlewicz} yield
\begin{equation}\label{eq:coupled_CTRW_$J_1$_conv}
	\; X^n \; \; \Rightarrow  \; \; \bigl((Z^-)_{D^{-1}}\bigr)^+,
\end{equation}
on $(\D_{\R^d}[0,\infty), \dJ)$ where we have used the notation $x^-(t):=x(t-)$ and $x^+(t):=x(t+)$. Should the CTRWs be uncoupled, then $Z$ and $D$ arise as independent Lévy processes, so they almost surely have no common discontinuities, and the limit simplifies to $Z_{D^{-1}}$ in agreement with \eqref{eq:CTRW_$M_1$_conv} for uncorrelated CTRWs (for further details, see \cite[Lem.~3.9]{henry_straka}).

\begin{remark}[Domain of attraction]The above assumptions---on the jumps and the waiting times being in the normal domain of attraction of the respective stable laws---serve to ease notation and presentation. Analogous convergence results to the ones discussed above also hold for random variables belonging only to the \emph{strict domain of attraction} (as opposed to the stricter normal domain of attraction). Moreover, in the case of zero-order moving averages, Skorokhod \cite{skorokhod} proved a similar convergence result for innovations that are but in the domain of attraction of an $\alpha$-stable law, where the limit then is an $\alpha$-stable Lévy process or a Brownian motion with drift. All of our results in the subsequent sections generalise to these settings, with only minor adaptations to the proofs.
\end{remark}

\begin{remark}[Omitting the distant absolute past leaves results unchanged] \label{rem:altered_CTRW} For moving averages and CTRWs, the dependence structure can be altered in such a way that there is but a dependence on finitely many innovations from the absolute past without affecting the convergence results. More precisely, if $\mathcal J\ge 1$ and we redefine 
	$$ \zeta_i \; := \; \sum_{j=0}^{i+\mathcal J} c_j \theta_{i-j}, $$
	then, if all other assumptions remain in place, we still have $(\sum_{j=0}^\infty c_j)^{-1} X^n \Rightarrow Z$ for moving averages \eqref{eq:MovAv}, and $(\sum_{j=0}^\infty c_j)^{-1} X^n \Rightarrow Z_{D^{-1}}$ for correlated CTRWs \eqref{defi:CTRW}. This can be shown in full analogy to the proofs in \cite{avram,meerschaert2}, as the remainder becomes asymptotically negligible.
\end{remark}

Throughout, we work on a given family of filtered probability spaces $(\Omega, \mathcal{F}^n, \F^n, \Pro)$, for $n\geq 1$. A natural choice of filtration is
\begin{align*} \mathcal{F}^n_t \; := \;  \sigma(X^n_s \, : \, 0\le s \le t) \, \vee \, \mathcal{G}^n_t,\qquad t\geq0, \, n\ge 1, \end{align*}
where the $\mathcal{G}^n_s$, $s\ge 0$, can be any filtration independent of $\sigma(\theta_k, J_m : -\infty < k < \infty, m\ge 1)$. When the innovations $\zeta_i$ are non-zero a.s., one readily confirms that this is the same as
\begin{align} \mathcal{F}^n_t \; = \;  \sigma(\zeta^n_{N_{ns}}, \, N_{ns} \, : \, 0\le s \le t) \, \vee \, \mathcal{G}^n_t,\label{eq:filtration_CTRWs} \end{align}
where each $\zeta^n_{N_{ns}}$ should be understood as $\sum_{k=1}^\infty \zeta^n_k \ind_{\{N_{ns}=k\}}$. Unless it is explicitly stated otherwise, the filtration is given by \eqref{eq:filtration_CTRWs}.

\section{Convergence of stochastic integrals and the case of uncorrelated CTRWs}\label{sect:GD_CTRW}

In earlier works on the weak convergence of stochastic integrals with uncorrelated CTRW integrators, only the uncoupled case has been considered and the approaches have necessitated  restrictive assumptions; see \cite{Burr, paulaskas_rachev, scalas} and \cite[Ch.~6]{hahn}. These works are all based on the general $J_1$ theory developed in \cite{jakubowskimeminpages, kurtzprotter} and hence rely on verifying the P-UT condition of \cite{jakubowskimeminpages} or the UCV condition of \cite{kurtzprotter}. In \cite{scalas}, this is done for zero-order moving averages with $\alpha \in (1,2]$ and innovations whose laws are symmetric. Via \cite{meerschaert} and a time-change argument as in \cite{Kobayashi}, this yields a result for certain uncorrelated, uncoupled CTRW integrators and a continuous, deterministic integrand. We note that \cite[Sect.~4.3]{scalas} also considers zero-order moving averages with $\alpha\in (0,2]$ and without the symmetry assumption, but there is an error in the proof. In \cite{Burr}, the P-UT and UCV conditions are verified for uncorrelated, uncoupled CTRWs with $\alpha \in (1,2]$ and centered innovations. Similarly, the analysis of zero-order moving averages in \cite{paulaskas_rachev} relies on $\alpha \in (1,2]$ and the innovations being centered. Beyond the limitations of particular assumptions, we stress that none of the above arguments generalise to coupled CTRWs or the interesting critical case $\alpha=1$, even for zero-order moving averages.

Within this section, our approach will rely on the general framework for weak convergence of stochastic integrals presented in \cite{andreasfabrice_theorypaper}, akin to \cite{jakubowskimeminpages, kurtzprotter}. Similarly to the role of the P-UT and UCV conditions in the above discussion, the key point will be to verify a notion of good decompositions for the integrators which we state in Definition \ref{defi:GD} below.

\subsection{General conditions for the convergence of stochastic integrals}
\label{subsect:stoch_int_conv}
This section recalls two important definitions from \cite{andreasfabrice_theorypaper} that will play a central role throughout. They are concerned with the type of behavior of the integrators and integrands that allows for general results on the functional weak convergence of the corresponding stochastic integrals.

We begin by fixing some notation. We will write $|Z|^*_t:= \operatorname{sup}_{0\le s\le t} |Z_s|$ for the running supremum and $\operatorname{TV}_{[0,t]}(Z)
:= |Z_0|+\operatorname{sup}_{\pi}  \,\sum_{i=1}^{m} |Z_{t_{i}}-Z_{t_{i-1}}|$ for the total variation, where $\pi$ ranges over all $0=t_0<\cdots<t_m=t$. Finally $\Delta Z_t= Z_t-Z_{t-}$ for $t>0$ and $\Delta Z_{0}:=0$.
\begin{defi}[Good decompositions, {\cite[Def.~2.3] {andreasfabrice_theorypaper}}]\label{defi:GD} Let $(X^n)_{n\geq 1}$ be a sequence of semimartingales on probability spaces $(\Omega^n, \mathcal{F}^n, \F^n, \Pro^n)$. The sequence is said to have \emph{good decompositions} \eqref{eq:Mn_An_condition} for the given filtrations $\F^n$, if there exist decompositions
	\[
	X^n = M^n + A^n,\quad M^n  \text{ local martingales},\quad  A^n \text{ finite variation processes},
	\]
	such that, for every $t>0$, we have 
	\begin{align}\label{eq:Mn_An_condition}\tag{GD}
		\lim_{R\rightarrow \infty}  \limsup_{n\rightarrow\infty } \, \mathbb{P}^n\bigl(\text{TV}_{[0,t]}(A^n)>R\bigr)=0 \quad \; \; \text{and}  \; \; \quad
		\limsup_{n\rightarrow \infty} \,  \mathbb{E}^n\bigl[\, |\Delta M^n_{t \land \tau^n_c}|\, \bigr] <\infty,
	\end{align}
	for all $c>0$, where $\tau^n_c:= \operatorname{inf}\{s >0: |M^n|^*_s\ge c \}$. 
\end{defi}

We note that \eqref{eq:Mn_An_condition} plays a role analogous to that of the P-UT and UCV conditions in \cite{jakubowski,jakubowskimeminpages,kurtzprotter} (for details on how they compare, see \cite{andreasfabrice_theorypaper}). To address the interplay between integrands and integrators, we define a function $\hat{w}^T_\delta: \D_{\R^d}[0,\infty) \times \D_{\R^d}[0,\infty) \to \R_+$ of the largest consecutive increment within a $\delta$-period of time on $[0,T]$, namely
\begin{align*}
	\hat{w}^T_{\delta}(x,y)  := \operatorname{sup}\!\left\{  |x^{(i)}(s)- x^{(i)}(t)| \wedge |y^{(i)}(t)-y^{(i)}(u)|  :   s < t  < u \le (s+\delta), \, 1\le i\le d    \right\},
\end{align*}
where the supremum is restricted to $0\leq s, u\leq T$. Here we have used the usual notation $a\wedge b:= \operatorname{min}\{a,b\}$, and we have denoted by $x^{(i)}$ the $i$-th coordinate of $x$. In addition to the pivotal \eqref{eq:Mn_An_condition} property, the second essential ingredient is the following condition.

\begin{defi}[Asymptotically vanishing consecutive increments, {\cite[Def.~2.2]{andreasfabrice_theorypaper}}]\label{def:avco}
	Let $(X^n)_{n\ge 1}$ and $(H^n)_{n\ge 1}$ be sequences of $d$-dimensional càdlàg processes on given probability spaces $(\Omega^n, \mathbbm{F}^n, \Pro^n)$. The sequence $(H^n, X^n)_{n\ge 1}$ is said to satisfy the \emph{asymptotically vanishing consecutive increments} condition if, for every $\gamma>0$ and $T> 0$, it holds that
	\begin{align} 
		\lim\limits_{\delta \downarrow 0}  \limsup\limits_{n\rightarrow \infty} \,  \Pro^n\bigl( \, \hat w_{\delta}^T(H^n, \, X^n) \; > \; \gamma \, \bigr) \; = \; 0. \tag{AVCI} \label{eq:oscillcond}
	\end{align}
\end{defi}

Whilst the above formulation is taken from \cite[Sect.~2.1]{andreasfabrice_theorypaper}, we stress that the idea of enforcing a condition of this form comes from \cite{jakubowski}, as discussed in more detail in \cite{andreasfabrice_theorypaper}. 

Now consider given integrators $X^n$ and integrands $H^n$ which are weakly convergent to $X$, $H$ respectively, for the $J_1$ or $M_1$ topologies. If the $X^n$ have good decompositions \eqref{eq:Mn_An_condition} and the integrands are so that \eqref{eq:oscillcond} is satisfied, then one can make general statements about the weak convergence, in the $J_1$ or $M_1$ topology, of the It{\^o} integrals $\int_0^\bullet H^n_{s-}\diff X^n_s$ to the corresponding It{\^o} integral $\int_0^\bullet H_{s-}\diff X_s$ for the limiting processes. We recall such a result in Theorem \ref{thm:3.19} in the \hyperref[appn]{Appendix} which we shall make use of in the work that follows.

\subsection{Convergence of stochastic integrals driven by uncorrelated CTRWs}
\label{subsec:convergence_uncorrelated_CTRWs}

We will start by showing that uncorrelated CTRWs enjoy good decompositions \eqref{eq:Mn_An_condition}. In order to establish this, we fix $a\ge 1$ and consider the decompositions
\begin{align} X^n_t \; = \; \sum_{k=1}^{N_{nt}} \zeta^n_k \;
	&= \;M^n_t \; + \;\sum_{k=1}^{N_{nt}} \,  \zeta^n_k \ind_{\{ |\zeta^n_k|>a\}} \; + \; \sum_{k=1}^{N_{nt}} \, \E[ \zeta^n_k\ind_{\{ |\zeta^n_k |\le a\}} \, | \, J_k ] \label{eq:uncorrelated_CTRW_decomp}
\end{align}
where, for simplicity, we have introduced $\zeta^n_k:= n^{-\frac\beta\alpha} \zeta_k=n^{-\frac\beta\alpha} \theta_k$ so that $X^n_t= \sum_{k=1}^{N_{nt}} \zeta^n_k$, and where we have defined
\begin{equation}\label{eq:CTRW_martingale}
	M^n_t:= \sum_{k=1}^{N_{nt}} \zeta^n_k \ind_{\{|\zeta^n_k|\le a \}} \, - \,  \sum_{k=1}^{N_{nt}} \, \E[ \zeta^n_k\ind_{\{ |\zeta^n_k |\le a\}} \, | \, J_k ].
\end{equation}
Note that, for uncorrelated uncoupled CTRWs, the last summand on the right-hand sides of \eqref{eq:uncorrelated_CTRW_decomp} and \eqref{eq:CTRW_martingale} simplifies to $\pm N_{nt} \E[\zeta^n_1\ind_{\{ |\zeta^n_1 |\le a\}}]$. For (uncorrelated) coupled CTRWs, we will require the following conditional moment condition:
\begin{align}
	\begin{aligned}
		&\E[\zeta_1 \,|\, J_1]=0 \;\text{ a.s.},
		\quad \text{if } 1<\alpha\le2,\\[-1pt]
		&\operatorname{sup}_{R>0}
		\E\bigl[\bigl|\E[\zeta_1\ind_{\{|\zeta_1|\le R\}} \,|\, J_1]\bigr|\bigr]
		<\infty,
		\quad \text{if } \alpha=1.
	\end{aligned}
	\tag{CC}\label{eq:condition_centering_condition_coupled_case}
\end{align}

Before we proceed, we briefly recount the key steps in \cite{paulaskas_rachev}, as these arguments are the closest to the approach we implement here. The setting of \cite{paulaskas_rachev} is $\alpha>1$ with deterministic waiting times (i.e., essentially $\beta=1$) and the innovations $\theta_k$ are centered. Leading up to \cite[Prop.~3]{paulaskas_rachev}, a decomposition similar to \eqref{eq:uncorrelated_CTRW_decomp} is introduced. However, the proof of \cite[Prop.~3]{paulaskas_rachev} then exploits the existence of first moments of the $\theta_k$ to pass over to (in this case equivalently) proving $\operatorname{sup}_{n\ge 1}  n \E[ \zeta^n_1\ind_{\{ |\zeta^n_1 |> a\}}] < \infty$. In general, if the innovations are not centered or if e.g.~$\alpha\le 1$, such a procedure fails to apply even for zero-order moving averages. Instead, we are able to handle the general case by exploiting the weak convergence and the tail regularity of the $\theta_k$. The first step is the following simple observation.

\begin{prop} \label{prop:uncorr_CTRW_martingales}
	Let $(X^n)_{n\ge 1}$ in \eqref{eq:uncorrelated_CTRW_decomp} be a sequence of uncorrelated, possibly coupled, CTRWs with $0< \alpha \le 2$. For each $n\ge 1$, the process $M^n$ defined by \eqref{eq:CTRW_martingale} above is a martingale with respect to the filtration $(\mathcal{F}^n_t)_{t\ge 0}$ given in \eqref{eq:filtration_CTRWs}. Furthermore, we have $|\Delta M^n|\le 2a$, for all $n\geq 1$.
\end{prop}

Restricted to uncorrelated uncoupled CTRWs with centered innovations, the martingality of the CTRWs has been observed by \cite[p.~6]{Germano_etal}, and, under the additional assumption that $\E[N_{nt}]<\infty$ for all $n\ge 1, t\ge 0$, by \cite[Lem.~1]{Burr}. Similar observations were also made in \cite{paulaskas_rachev} as discussed above, but again in a less general setting than the one we consider here.

Given that not only the number and size of large jumps of the uncorrelated uncoupled CTRW are tight on compact time intervals but also the sequence $n^{-\beta} N_{n\bullet}$ (due to their tightness on Skorokhod space, namely \eqref{eq:time_convergence_to_subordinator} and \eqref{eq:CTRW_$M_1$_conv}), whether or not the $X^n$ admit good decompositions ultimately depends on whether or not the supremum $\operatorname{sup}_{n\ge 1}  n^\beta |\E[ \zeta^n_1\ind_{\{ |\zeta^n_1 |\le a\}}]|$ is finite.

Define the truncation function $h(x)= x \ind_{\{ |x|\le a\}} + \, \operatorname{sgn}(x) \, a \ind_{\{ |x|> a\}}$ for $x\in \mathbb{R}$ and let $h(z)=(h(z_1),\ldots,h(z_d))$ for $z\in\mathbb{R}^d$. We have $\Pro(|\theta_1|>x)=\Oland(x^{-\alpha})$ (see e.g.~\cite[Thm.~4.5.2]{whitt} for $0<\alpha<2$, and note $\E[|\theta_1|^2]<\infty$ for $\alpha=2$) and $n^\beta \Pro(|\zeta^n_1| > a) \to c_\alpha a^{-\alpha}$ as $n\to \infty$, for some $c_\alpha\geq 0$. Thus, the condition $\operatorname{sup}_{n\ge 1}  n^\beta |\E[ \zeta^n_1\ind_{\{ |\zeta^n_1 |\le a\}}] |< \infty$ is equivalent to $\operatorname{sup}_{n\ge 1}  n^\beta |\E[ h(\zeta^n_1)]| < \infty$. The latter is guaranteed by the following proposition.

\begin{prop}\label{prop:aux_result_CTRW_has_GD}
	In the above setting, $\lim_{n\rightarrow \infty} n^\beta \E[h(\zeta^n_1)] = b$, for some  $b \in \R^d$.
\end{prop}

As suggested by the above, the decomposition \eqref{eq:uncorrelated_CTRW_decomp}, and Propositions \ref{prop:uncorr_CTRW_martingales} and \ref{prop:aux_result_CTRW_has_GD} allow us to conclude that uncorrelated uncoupled CTRWs have good decompositions. An analogous result naturally holds for zero-order moving averages. The details of this argument (together with the proof of good decompositions for coupled CTRWs enjoying \eqref{eq:condition_centering_condition_coupled_case}) can be found in Section \ref{sec:proofs}.

\begin{theorem}[Good decompositions in the uncorrelated case]\label{thm:CTRW_has_GD}
	As defined in Section \ref{sect:CTRW_and_GD}, let $(X^n)_{n\ge 1}$ be a sequence of zero-order moving averages, uncorrelated uncoupled CTRWs, or (uncorrelated) coupled CTRWs with \eqref{eq:condition_centering_condition_coupled_case}. Then, $(X^n)_{n\ge 1}$ has good decompositions \eqref{eq:Mn_An_condition} for the filtrations \eqref{eq:filtration_CTRWs}.
\end{theorem}

With good decompositions for zero-order moving averages, and uncorrelated CTRWs, we can now state the following result on the functional weak convergence of stochastic integrals driven by these very processes.

\begin{theorem}[Weak integral convergence in the uncorrelated case] \label{thm:result_int_conv_uncorrelated_CTRWs}
	Let $0<\alpha\le 2$ and $(X^n)_{n\ge 1}$ be a sequence of zero-order moving averages, uncorrelated uncoupled CTRWs, or (uncorrelated) coupled CTRW with \eqref{eq:condition_centering_condition_coupled_case}, defined on filtered probability spaces \eqref{eq:filtration_CTRWs}. Further, let $H^n$ be càdlàg adapted processes for the same filtered probability spaces. If $(H^n,X^n) \Rightarrow (H, X)$ on $(\D_{\R^d}[0,\infty), \dM) \times (\D_{\R^d}[0,\infty), \dJ)$ with $X$ given by $Z$, $Z_{D^{-1}}$ or $((Z^{-})_{D^{-1}})^+$, and if the $(H^n,X^n)$ satisfy \eqref{eq:oscillcond}, then $X$ is a semimartingale in the natural filtration generated by $(H,X)$ and it holds that
	\begin{align} \left( X^n, \, \int_0^\bullet  \, H^n_{s-} \; \diff X^n_s\right) \; \; \Longrightarrow \; \;  \left( X, \, \int_0^\bullet \, H_{s-} \; \diff X_s\right) \quad \text{ on } \quad \; (\D_{\R^{2d}}[0,\infty), \dJ). \label{eq:integral_convergence}
	\end{align}
\end{theorem}
Since Theorem \ref{thm:CTRW_has_GD} gives \eqref{eq:Mn_An_condition}, this result follows from Theorem \ref{thm:3.19} in the \hyperref[appn]{Appendix}. Regarding the \eqref{eq:oscillcond} condition, we note that it is automatically satisfied if the limiting integrand $H$ and integrator $X$ almost surely have no common discontinuities, see Proposition \ref{prop:3.3} in the \hyperref[appn]{Appendix}. This would for instance be the case if $H$ is predictable, since the jumps of a Lévy process are totally inaccessible. Beyond such simple sufficient criteria, one often has to resort to more specific properties of the processes. A relevant example, where one can verify \eqref{eq:oscillcond} directly, is to consider integrands of the form $H^n_t= \sum_{i=1}^\infty g_{n,i}(t^n_i,X^n_{t^n_i}) \ind_{(t^n_i,t^n_{i+1}]}(t)$, with equicontinuous $g_{n,i}$. Then, one can make use of the alternative $J_1$ tightness criteria given in \cite[Thm.~12.4]{billingsley} and the $J_1$ tightness of the uncorrelated CTRWs in order to show that \eqref{eq:oscillcond} holds. In general, however, it may not always be clear how to verify \eqref{eq:oscillcond} directly. For such cases, the next remark might turn out useful.

\begin{remark}[An alternative to AVCI] \label{rem:alternative_conditions_AVCI}
	We wish to point out that the \eqref{eq:oscillcond} condition in Theorem \ref{thm:result_int_conv_uncorrelated_CTRWs} may be replaced by an alternative set of conditions, namely conditions (a) and (b) from \cite[Thm.~4.1]{andreasfabrice_theorypaper}. These two conditions require, firstly, suitable independence of integrands and increments of the integrators on the intervals of random partitions (which can vary with $n$), and, secondly, asymptotically vanishing increments of the integrators in probability on these intervals. For brevity of this exposition, we refer the interested reader to \cite[Thm.~4.1]{andreasfabrice_theorypaper} for a precise description. Nevertheless, we shall stress that, for uncorrelated uncoupled CTRWs $X^n$, those conditions are readily seen to be verified when the integrands $H^n$ are adapted to the natural filtration of the $X^n$ (or more generally \eqref{eq:filtration_CTRWs}). Indeed, part (a) can be shown by \cite[Lem.~4.3]{andreasfabrice_theorypaper}, and (b) follows as explained in the paragraph following \cite[Ex.~4.2]{andreasfabrice_theorypaper}. While we prefer to rely only on \eqref{eq:oscillcond} for the statement of our general results, we note that the alternative conditions (a) and (b) will be used in Section \ref{sec:proofs} for the proofs of Theorems \ref{thm:result_int_conv_CTRWs_with_independence_cond} and \ref{thm:SDE_approx_result}. In fact, whenever \eqref{eq:oscillcond} is invoked, these alternative conditions may be substituted without affecting the results (by replacing Theorem \ref{thm:3.19} with \cite[Thm.~4.1]{andreasfabrice_theorypaper} at any point of the proofs in Section \ref{sec:proofs} where the former is referred to).
\end{remark}

\subsection{Moving averages and correlated CTRWs do not blend in}\label{subsec:Correlated_CTRWs_do_not_blend_in}
As we shall see below, one unfortunately cannot expect moving averages and correlated CTRWs to have good decompositions \eqref{eq:Mn_An_condition} when $\alpha \in [1,2)$, so a new framework will need to be developed. These issues, however, are \emph{not} present for $\alpha\in(0,1)$, as the following result confirms.

\begin{prop}[Good decompositions for $0<\alpha<1$] \label{prop:uncoupled_correlated_CTRW_tight_total_variation} If\phantom{s}$\alpha\in(0,1)$, then moving averages as in \eqref{eq:MovAv} and correlated CTRWs as in \eqref{defi:CTRW} are processes of tight total variation on compact time intervals and therefore admit good decompositions. 
\end{prop}

Due to Proposition \ref{prop:uncoupled_correlated_CTRW_tight_total_variation}, we can extend Theorem \ref{thm:result_int_conv_uncorrelated_CTRWs} to cover integrators which are correlated CTRWs or moving averages with $0<\alpha <1$. 

\begin{theorem}[$M_1$ integral convergence for $0<\alpha<1$]\label{thm:result_int_conv_correlated_CTRWs_alpha<1}
	Theorem \ref{thm:result_int_conv_uncorrelated_CTRWs} also holds if the $X^n$ are moving averages or  correlated CTRWs with $0<\alpha<1$, where either $X=(\sum_{j=0}^\infty c_j) Z$ or $X=(\sum_{j=0}^\infty c_j) Z_{D^{-1}}$, respectively, provided we replace $\dJ$ with $\dM$.
\end{theorem}

In the case $1\le \alpha\le 2$, the situation is very different from that of uncorrelated CTRWs or zero-order moving averages: even a single-delay correlation structure of the $X^n$ in general does \emph{not} admit good decompositions \eqref{eq:Mn_An_condition}. This and the resulting failure of convergence, even for integrands $H^n$ which converge to the zero process almost surely in the uniform norm, is demonstrated in \cite[Prop.~5.3]{andreasfabrice_theorypaper}. We note that, in view of \cite[Prop.~C.1]{andreasfabrice_theorypaper}, this also entails that the P-UT and UCV conditions discussed at the start of Section \ref{sect:GD_CTRW} fail to be satisfied.

Intuitively, what happens in these cases is that, based on the auto-regressive structure of the CTRW or moving averages, the sign of the last jump provides a type of predictive advantage for the direction of the subsequent one. This can be utilised by constructing adapted integrands $H^n$ that exploit this advantage to `pick up' just enough of the variation through integration to make the integrals explode. 

In these circumstances, to preserve integral convergence we can seek to compensate the lack of \eqref{eq:Mn_An_condition} by imposing additional restrictions on the interplay between the integrands and integrators. In the later Sections \ref{sec:generalised_framework} and \ref{subsec:Final_results_for_CTRWs}, we will pursue such an extended framework for weak integral convergence on Skorokhod space, covering integrators that are in a certain sense close to admitting good decompositions.

\subsection{Going beyond the good decompositions}  \label{sec:quantities_and_approach_needed_for_integral_convergence_proof}

A general approach to weak convergence of Itô integrals on the $J_1$ or $M_1$ Skorokhod space, that is $\int_0^\bullet H^n_{s-}\diff X^n_s \Rightarrow \int_0^\bullet H_{s-}\diff X_s$, is described in \cite[pp.~26--30]{andreasfabrice_theorypaper}, provided that the pairs $(H^n,X^n)$ are adapted to the same filtration and that $(H^n,X^n) \Rightarrow (H,X)$ on the product space $(\D_{\R^d}[0,\infty),\tilde\rho) \times (\D_{\R^d}[0,\infty),\rho)$, with $\tilde\rho,\rho \in \{\dJ,\dM\}$. Analogously to the ideas set forth in \cite{jakubowski}, by discretising the integrands $H^n$ both with respect to certain shrinking deterministic partitions and with respect to suitable random partitions, the proof can be divided into four subproblems, each of which can be addressed individually. All but one of these subproblems are essentially managed by resorting to \eqref{eq:oscillcond} together with the weak $J_1$ or $M_1$ convergence of the pairs $(H^n,X^n) $. The remaining subproblem, however, relies crucially on the existence of good decompositions \eqref{eq:Mn_An_condition} for the integrators $X^n$. More precisely, the terms that are sought to be controlled for this subproblem are given by 
\begin{align} \Upsilon_{n,m,\varepsilon} \; := \;  \E^n \biggl[ \; \biggl| \, \int_0^\bullet \, (H^n_{s-} - H^{n\, \vert \, m,\varepsilon}_{s-}) \; \diff X^n_s \, \biggr|^*_T \; \wedge \; 1 \; \biggr], \qquad n\ge 1, \label{eq:quantity1_integral_convergence_term_where_GD_needed}
\end{align}
for all $T>0$, where $H^{n\, \vert \, m,\varepsilon}_s:= \sum_{i=0}^{\infty} \, H^n_{\tau^{n}_i} \ind_{[\tau^{n}_i,\, \tau^{n}_{i+1})}(s)$ is a discretisation of $H^n$ on a random partition $\pi^{n,m,\varepsilon}=\{0=\tau^{n}_0<\tau^{n}_1<...\}$, which is almost surely finite on bounded time intervals. The partition can be written as the union of a random (stopping time) partition, depending on $n\ge 1$ and $\varepsilon>0$, and ensuring that $|H^n-H^{n\, \vert \, m,\varepsilon}|\le \varepsilon$ almost surely for all $n\ge 1$, and a deterministic partition $\rho^m$, solely depending on $m\ge 1$, of almost sure continuity points of the limits $(H,X)$ such that its mesh size $|\rho^m|\to 0$ as $m\to \infty$. Effectively, the mode of control required for these terms is 
\begin{align}
	\lim\limits_{\varepsilon \to 0} \; \limsup\limits_{m\to \infty} \; \limsup\limits_{n\to \infty}\; \Upsilon_{n,m,\varepsilon} \; = \; 0. \label{eq:quantity2_integral_convergence_term_where_GD_needed}
\end{align}
We would like to stress that this control is equivalent to 
\begin{align}   \lim\limits_{\varepsilon \to 0} \;\limsup\limits_{m\to \infty} \; \limsup\limits_{n\to \infty}\; \Pro^n \biggl( \; \biggl| \, \int_0^\bullet \, (H^n_{s-} - H^{n\, \vert \, m,\varepsilon}_{s-}) \; \diff X^n_s \, \biggr|^*_T \; \ge \; \eta \; \biggr) \; = \; 0 \label{eq:quantity3_integral_convergence_term_where_GD_needed}
\end{align}
for all $\eta>0$. If the integrators $X^n$ admit good decompositions \eqref{eq:Mn_An_condition}, then we have the required control \eqref{eq:quantity2_integral_convergence_term_where_GD_needed}.

When the integrators $X^n$ do \emph{not} have good decompositions \eqref{eq:Mn_An_condition}, the weak convergence (or even the tightness) of the stochastic integrals $\int_0^\bullet H^n_{s-}\diff X^n_s$ on Skorokhod space can fail markedly, as already discussed in Section \ref{subsec:Correlated_CTRWs_do_not_blend_in}. It is far from obvious which set of meaningful conditions one could impose in order to restore the weak convergence of stochastic integrals, and we note that, to the best of our knowledge, this type of analysis has not been actively attempted in any earlier works. Motivated by the general structure of correlated CTRWs and moving averages, the next section introduces a generalised framework that will rely on some additional conditions with respect to the interplay of the integrands and integrators in order to compensate for the lack of good decompositions.


\section{A generalised framework beyond good decompositions}
\label{sec:generalised_framework}
Consider the correlated CTRWs $X^n$ defined in \eqref{defi:CTRW} with $1\le \alpha\le 2$, on a given probability space equipped with the filtration \eqref{eq:filtration_CTRWs}. One can then decompose $\psi^{-1} X^n=U^n+V^n$ with
\begin{align} U^n_t \; := \; U^{n,1}_t \, + \, U^{n,2}_t \; := \;  \frac{1}{n^{\frac \beta \alpha}} \, \sum_{k=1}^{N_{nt}} \; \theta_{k}\; + \; \frac{\psi^{-1}}{n^{\frac \beta \alpha}} \,\sum_{k=0}^{\infty}\, \biggl( \, \sum_{j=k+1}^{ N_{nt}+k } c_{j} \, \biggr) \, \theta_{-k}, \label{eq:defi_U^n_in_CTRW_decomp}
\end{align}
\begin{align} V^n_t \; := \; - \, \frac{\psi^{-1}}{n^{\frac \beta \alpha}} \, \sum_{k=1}^{N_{nt}} \, \biggl(\, \sum_{j=k}^{\infty } c_{j}\, \biggr) \, \theta_{N_{nt}-k+1}, \label{eq:defi_V^n_in_CTRW_decomp}
\end{align}
where $\psi:= \sum_{j=0}^{\infty}c_j$. We observe that the first summand $U^{n,1}$ of $U^n$ is nothing else than an uncorrelated uncoupled CTRW and so, by Theorem \ref{thm:CTRW_has_GD}, possesses good decompositions \eqref{eq:Mn_An_condition}. If, by suitable assumptions, we ensure that the second summand $U^{n,2}$ is of tight total variation on compact time intervals, then it will come as no surprise that weak integral convergence can be achieved by simply controlling the interplay of the integrands $H^n$ and the processes $V^n$. In order for the former to be true, we impose a mild technical condition on the tail summability of the $c_j$, more precisely 
\begin{align} 
	\sum_{i=1}^\infty  \Big( \sum_{j=i}^\infty c_j \Big)^\rho \, < \, \infty \;\;  \text{ for some } 0<\rho<1, \quad \text{if } 1 \le \alpha\, \le \, 2.  \tag{TC} \label{eq:technical_condition} \end{align}
This condition is also needed as a technical tool in the proof of the later Proposition \ref{prop:correlated_CTRW_are_GDmodCI}.

\begin{remark} By Fubini's theorem, $\sum_{k=1}^\infty k \, c_k^\rho <\infty$ is a sufficient criterion for \eqref{eq:technical_condition}.
\end{remark}

\begin{lem} \label{lem:decomp_U^n_part_has_GD}
	Assume $(J_k)_{k\geq 1} \indep (\theta_k)_{k\in \mathbb{Z}}$. Under \eqref{eq:technical_condition}, the processes $U^n$ in \eqref{eq:defi_U^n_in_CTRW_decomp} have good decompositions \eqref{eq:Mn_An_condition} for the filtrations $\bar{\mathcal{F}}^n_t:=\sigma(N_{ns},\,\theta_i:0\le s\le t,\,i\le N_{nt})\vee \mathcal{G}^n_t$ with $\mathcal{G}^n$ as in \eqref{eq:filtration_CTRWs}.
\end{lem}

\subsection{Direct control of variation}
\label{subsec:Direct_control_of_variation}

The most direct approach to controlling integrals against the $V^n$ amounts to having the $H^n$ `tame' the total variation of the $V^n$. If the $H^n$ are pure jump processes, this becomes particularly simple, leading to Proposition \ref{prop:direct_control_of _TV} below.

\begin{defi}[Random countable partition] We call $\pi:=\{s_k : k=0,1,2,...\}\cup \{T\}$ a \emph{(countable) partition of $[0,T]$} if $\bigcup_{k=0}^\infty [s_{k}, s_{k+1})=[0,T)$ and $[s_{k}, s_{k+1}) \cap [s_{\ell}, s_{\ell+1})=\emptyset$ for all $k\neq \ell$. If $s_m=T$, we set $s_k=T$ for all $k>m$. If the $s_k$ are random variables, the partition is said to be \emph{random}. Further, we denote the mesh by $|\pi|:= \operatorname{sup}_{k\ge 0}|s_{k+1}-s_{k}|$.
\end{defi}

Towards the next proposition, let $(H^n)_{n\ge 1}$ be a sequence of adapted pure jump càdlàg integrands such that, for every $T>0$, the set $\operatorname{Disc}_{(0,T)}(H^n)\cup\{0,T\}=\{s_k:k\ge0\}\cup\{T\}$ is a countable partition of $[0,T]$. Further assume $X^n=U^n+V^n$, $V^n$ are semimartingales adapted to the same filtration as the $H^n$, where the $U^n$ have \eqref{eq:Mn_An_condition}, $X^n \Rightarrow X$ on $(\D_{\R^d}[0,\infty), \rho)$ with $\rho \in \{\dJ, \dM\}$ and $U^n \Rightarrow X$ on $(\D_{\R^d}[0,\infty), \dM)$ for some $X$, and
\begin{align} \operatorname{sup}_{n\ge 1} \, \Pro^n \Bigl(  \sum_{k=0}^\infty \; |V^n_{s_{k+1}}-V^n_{s_{k}}| \; > \; \lambda \Bigr)\; \conv{\lambda \to \infty} \; \; 0,\quad \text{for each}\;\;T>0. \label{eq:direct_control_TV}
\end{align}

\begin{prop}[Pure jump integrands and control of variation] \label{prop:direct_control_of _TV} 
	Suppose $(H^n, X^n)$ satisfy \eqref{eq:direct_control_TV} and \eqref{eq:oscillcond}. If $(H^n,X^n) \Rightarrow (H,X)$ on $(\D_{\R^d}[0,\infty), \dM) \times (\D_{\R^d}[0,\infty), \dM)$, with $(H^n,U^n)\Rightarrow_{f.d.d.} (H,X)$ on a dense set of times, then it holds
	$$ \Big(X^n, \, \int_0^\bullet \, H^n_{s-} \; \diff X^n_s\Big)  \; \; \; \Longrightarrow \; \; \;  \Big(X, \, \int_0^\bullet \, H_{s-} \; \diff X_s\Big) \qquad \text{ on } \quad \; (\D_{\R^{2d}}[0,\infty), \dM).$$
\end{prop}

The approach of controlling directly the activity of the $V^n$ is not restricted to pure jump integrands. Indeed, we can ask for the integrands to be Lipschitz continuous in such a way that they exhibit enough inertia to not be able to react in a critical way to changes of the $V^n$ and therefore do not `pick up' too much of the latter's variation through integration.

\begin{defi}[GD modulo controllable activity] \label{def:GD_mod_CA}
	Let $(X^n)_{n\ge 1}$ be a sequence of $d$-dimensional semimartingales on filtered probability spaces $(\Omega^n, \mathcal{F}^n, \mathbbm{F}^n, \Pro^n)$ and let $X$ be defined on $(\Omega,  \mathcal{F}, \mathbb{F}, \Pro)$. We say that the sequence $(X^n)_{n\ge 1}$ has \emph{good decompositions modulo a weakly asymptotically negligible process of $n^{\gamma}$-controllable activity on an $n^{-\lambda}$-fine partition}---abbreviated as $\operatorname{GD\,mod\,CA}(\gamma, \lambda)$---with $\gamma,\lambda>0$, if there exist processes $(U^n)_{n\ge 1}$, $(V^n)_{n\ge 1}$ on the same filtered probability spaces as the $X^n$ such that 
	\begin{enumerate}[(i)]
		\item the $U^n$ are semimartingales having \eqref{eq:Mn_An_condition} and $U^n \Rightarrow X$ on $(\D_{\R^d}[0,\infty),\dM)$;
		\item the $V^n$ are adapted, càdlàg pure jump processes of finite variation such that, for every $T>0$ there exist random partitions $\pi^n:=\pi^n(\omega)$ of $[0,T]$ with $|\pi^n|\le n^{-\lambda}$ as well as $\operatorname{Disc}_{[0,T]}(V^n) \subseteq \pi^n$ almost surely, and it holds
		\begin{align*} n^{-\gamma} \; \sum_{s \; \in \; \pi^n} \; |V^n_{s}|  \; \; \xrightarrow[{n\to \infty}]{\Pro^n} \;\;  0\, \end{align*}
		\item $X^n=U^n+V^n \Rightarrow X$ on $(\D_{\R^d}[0,\infty), \dM)$.
	\end{enumerate}
\end{defi}

\begin{theorem}[Lipschitz integrands and $\operatorname{GD\, mod \, CA}$] \label{thm:Lipschitz_GDmodCA} 
	Let $(X^n)_{n\ge 1}$ be a sequence of\,\,\,$d$-dimensional semimartingales on filtered probability spaces $(\Omega^n, \mathcal{F}^n, \mathbbm{F}^n, \Pro^n)$ which are $\operatorname{GD \, mod \, CA}(\gamma, \tilde{\gamma})$ as in Definition \ref{def:GD_mod_CA}. Further let $(H^n)_{n\ge 1}$ be a sequence of adapted càdlàg processes on the same filtered probability spaces and suppose that, for every $n\ge 1$ and $T>0$, $H^n$ is almost surely Lipschitz continuous on $[0,T]$ with Lipschitz constant $C_T\,n^{\tilde\gamma - \gamma}$, that is
	$$ |H^n_t - H^n_s| \; \le \; C_T \;  n^{\tilde \gamma-\gamma} \; |t-s|$$ 
	for all $0\le s\le t\le T$, where $C_T>0$ only depends on $T$. When $\tilde{\gamma}>\gamma$, suppose further that both $(H^n, U^n)$ and $(H^n, X^n)$ satisfy \eqref{eq:oscillcond}.	If $(H^n, X^n) \Rightarrow (H,X)$ on $(\D_{\R^d}[0,\infty), \dM) \times (\D_{\R^d}[0,\infty), \rho)$, where $\rho \in \{\dJ, \dM\}$, then $X$ is a semimartingale in the natural filtration generated by $(H,X)$ when $\tilde{\gamma}\geq \gamma$, or its own filtration when $\tilde{\gamma}<\gamma$, and it holds
	$$ \Big(X^n, \, \int_0^\bullet \, H^n_{s-} \; \diff X^n_s\Big)  \; \; \; \Longrightarrow \; \; \;  \Big(X, \, \int_0^\bullet \, H_{s-} \; \diff X_s\Big) \qquad \text{ on } \quad \; (\D_{\R^{2d}}[0,\infty), \rho),$$
	where the integral is simply understood as $t\mapsto H_0(X_t-X_0)$ when $\tilde{\gamma}<\gamma$.
\end{theorem} 

\begin{remark} By definition of the consecutive increment function $\hat{w}$ in \eqref{eq:oscillcond}, we note that, since $X^n=U^n+V^n$, showing $(H^n, X^n)$ and $(H^n, U^n)$ satisfy \eqref{eq:oscillcond} is equivalent to establishing \eqref{eq:oscillcond} for $(H^n, X^n)$ and $(H^n, V^n)$, or $(H^n, U^n)$ and $(H^n, V^n)$.
\end{remark}

\subsection{Control through independence}
\label{subsec:Control_through_independence}

Instead of restricting the class of admissible integrands to such processes which act as a direct control to the activity of the $V^n$ (as pursued in Section \ref{subsec:Direct_control_of_variation}), there is another more probabilistic approach that suggests itself: if we impose that the integrands must not anticipate the `future' behaviour of the integrator remainders $V^n$ (i.e., adequate independence) and the (conditional) expectation of the latter is suitably centered around zero, then this should offer enough control for a weak continuity result of stochastic integrals. In the sequel, we will provide a precise framework for the implementation of this idea. 

\begin{defi}[GD modulo processes controllable by independence] \ \label{defi:GD_mod_CI} 
	Let $(X^n)_{n\ge 1}$ be $d$-dimensional semimartingales on filtered probability spaces $(\Omega^n, \mathcal{F}^n, \mathbbm{F}^n, \Pro^n)$ and let $X$ be a process on some filtered probability space $(\Omega,  \mathcal{F}, \mathbb{F}, \Pro)$. We say that the sequence $(X^n)_{n\ge 1}$ has \emph{good decompositions modulo weakly asymptotically negligible processes controllable through independence of the integrands}---abbreviated as $\operatorname{GD\,mod\,CI}$---if there exist processes $(U^n)_{n\ge 1}$, $(\tilde U ^n)_{n\ge 1}$, $(V^{n,i})_{n\ge 1 }$, $i\ge 1$, which are defined on the same filtered probability spaces as the $X^n$ as well as $f:\N \to (0,\infty)$ and $1<\lambda,\mu \leq 2$ such that 
	\begin{enumerate}[(i)]
		\item the $U^n$, $\tilde U^n$ are semimartingales having \eqref{eq:Mn_An_condition} and $U^n \Rightarrow X$ on $(\D_{\R^d}[0,\infty),\dM)$;
		\item the $V^{n,i}$ are pure jump semimartingales, with $V_0^{n,i}=0$ and finitely many jumps on compact time intervals such that $\sum_{i=1}^\infty V^{n,i}$ exists almost surely for each $n$. Let $\sigma^{n,i}_1 \le \sigma^{n,i}_2 \le ...$ be stopping times such that $\operatorname{Disc}(V^{n,i}) \subseteq \{ \sigma^{n,i}_k : k\ge 1\}$ and denote  $\Lambda^{n,i}(t):= \operatorname{sup}\{k \ge 1 :  \sigma^{n,i}_k\le t\}$. Further, it holds that:
		\begin{enumerate}
			\item[(ii.i)] for every $T>0$, the sequence $\operatorname{sup}_{i\ge 1} |\Lambda^{n,i}(\bullet)/f(n)|^*_T$ is tight in $\R$;
			\item[(ii.ii)] for every $n,k,i\ge 1$ we have
			$$ \E^n \Big[ V^{n,i}_{\sigma^{n,i}_k} \; | \; \mathcal{V}^i_{n,k-1}\Big] \; = \; 0,$$
			where $\mathcal{V}^i_{n,k}:= \sigma\bigl(V^{n,i}_{\sigma^{n,i}_j}: j\le k \bigr)$;
			\item[(ii.iii)] it holds that
			\begin{align*}
				&\limsup_{n \to \infty} \;  \; \sum_{i=1}^{\infty} \, \Big(\sum_{k=1}^{\lfloor K \, f(n)\rfloor}\,  \E^n \Big[ \,  |V^{n,i}_{\sigma^{n,i}_k}|^{\, \lambda} \ind_{\{|V^{n,i}_{\sigma^{n,i}_k}|\, > \, 1\}} \, \Big]\Big)^\frac{1}{\lambda} \; < \; \infty \qquad \text{ and}\\
				&\limsup_{n \to \infty} \;  \; \sum_{i=1}^{\infty} \, \Big(\sum_{k=1}^{\lfloor K \, f(n)\rfloor} \, \E^n \Big[ \,  |V^{n,i}_{\sigma^{n,i}_k}|^{\, \mu} \ind_{\{|V^{n,i}_{\sigma^{n,i}_k}|\, \le \, 1\}} \, \Big]\Big)^\frac{1}{\mu} \; < \; \infty ,\quad \text{for each $K>0$};
			\end{align*}
		\end{enumerate}
		\item $X^n=U^n+\tilde U^n+\sum_{i=1}^\infty V^{n,i} \Rightarrow X$ on $(\D_{\R^d}[0,\infty), \dM)$.
	\end{enumerate}
\end{defi}

\begin{theorem}[Independent integrands and $\operatorname{GD\,mod\,CI}$] \ \label{thm:Integral_converg_under_GDmodCI} 
	Let $(X^n)_{n\ge 1}$ be a sequence of $d$-dimensional semimartingales on filtered probability spaces $(\Omega^n, \mathcal{F}^n, \mathbbm{F}^n, \Pro^n)$ which are $\operatorname{GD\,mod\,CI}$ and let $(H^n)_{n\ge 1}$ be a sequence of adapted càdlàg processes on the same filtered probability spaces. Suppose that for every $n,k,i\ge 1$ it holds 
	\begin{gather}
		\mathcal{H}^i_{n,k} \vee \mathcal{V}^i_{n,k-2}
		\; \; \indep \; \;
		(V^{n,i}_{\sigma^{n,i}_k},\, V^{n,i}_{\sigma^{n,i}_{k-1}}),
		\label{eq:indep_cond_int_conv_GDmodCI}\\
		\mathcal{V}^i_{n,k}=\sigma\big(\,V^{n,i}_{\sigma^{n,i}_j}:j\le k\big),
		\quad
		\mathcal{H}^i_{n,k}:=
		\sigma\big(\,H^n_{t\,\wedge\,\sigma^{n,i}_k}:t\ge0\big)
		\vee \sigma\big(\sigma^{n,i}_\ell:\ell\le k\big),
		\notag
	\end{gather}
	and that the pairs $(H^n, X^n)$ satisfy \eqref{eq:oscillcond}. If $(H^n, X^n) \Rightarrow (H,X)$ on $(\D_{\R^d}[0,\infty), \dM) \times (\D_{\R^d}[0,\infty), \rho)$, where $\rho \in \{\dJ, \dM\}$,  with $(H^n,U^n)\Rightarrow_{f.d.d.} (H,X)$ on a dense set of times, then $X$ is a semimartingale in the natural filtration generated by $(H,X)$ and it holds
	$$ \Big(X^n, \, \int_0^\bullet \, H^n_{s-} \; \diff X^n_s\Big)  \; \; \; \Longrightarrow \; \; \;  \Big(X, \, \int_0^\bullet \, H_{s-} \; \diff X_s\Big) \qquad \text{ on } \quad \; (\D_{\R^{2d}}[0,\infty), \rho).$$
\end{theorem}

\section{Convergence of stochastic integrals for moving averages and correlated CTRWs}
\label{subsec:Final_results_for_CTRWs}
Throughout this section, we shall rely on the decomposition $\psi^{-1}X^n=U^n+V^n$ given in \eqref{eq:defi_U^n_in_CTRW_decomp}--\eqref{eq:defi_V^n_in_CTRW_decomp}, working with the filtration $\bar{\mathcal F}^n$ from Lemma \ref{lem:decomp_U^n_part_has_GD}. Starting from there, our objective is to apply the generalised framework developed in the previous section.

As one would hope, we have that moving averages and correlated CTRWs constitute natural classes of integrators which enjoy the $\operatorname{GD\,mod\,CA}$ property. 
Throughout the section, we assume the correlated CTRWs \eqref{defi:CTRW} are constructed from $(J_k)_{k\geq 1} \indep (\theta_k)_{k\in\mathbb{Z}}$. With a little additional work, Theorems \ref{thm:result_int_conv_CTRWs_lipschitz_integrands} and \ref{thm:result_int_conv_CTRWs_with_independence_cond} can also be shown to hold for $(J_k)_{k\geq 1} \indep (\zeta_k)_{k\geq 1}$, but we prefer the simpler presentation for the natural assumption $(J_k)_{k\geq 1} \indep (\theta_k)_{k\in\mathbb{Z}}$.

\begin{prop}[CTRWs\,\,are $\operatorname{GD \, mod \, CA}$] \label{prop:correlate_CTRWs_are_GDmodCA}
	Moving averages \eqref{eq:MovAv} and correlated CTRWs \eqref{defi:CTRW}, with $1\le \alpha\le 2$ and \eqref{eq:technical_condition}, are $\operatorname{GD\, mod \, CA}(\gamma, \beta)$ for all $\gamma>(\beta-\beta/\alpha)$.
\end{prop}

In fact, it is straightforward to see that Lemma \ref{lem:decomp_U^n_part_has_GD}, and therefore also Proposition \ref{prop:correlate_CTRWs_are_GDmodCA}, equally hold if we replace \eqref{eq:technical_condition} by $\sum_{i=1}^\infty \sum_{j=i}^\infty c_j < \infty$ in the case of $1< \alpha \le 2$. Moreover, we recall that $\beta=1$ for moving averages. From Theorem \ref{thm:Lipschitz_GDmodCA} and Propositions \ref{prop:correlate_CTRWs_are_GDmodCA}, we immediately obtain the following result on weak integral convergence of Lipschitz integrands with respect to moving averages and CTRWs with $1\le \alpha \le 2$.

\begin{theorem}[Weak integral convergence for CTRWs I] \label{thm:result_int_conv_CTRWs_lipschitz_integrands} 
	Let $X^n$ be a moving average as in \eqref{eq:MovAv} or a correlated CTRW as in \eqref{defi:CTRW}, with $1\le \alpha\le 2$ and \eqref{eq:technical_condition}.
	Furthermore, let $H^n$ be processes adapted to the filtration \eqref{eq:filtration_CTRWs} and let $\gamma \in (0, \beta/\alpha)$ such that for every $T>0$ there exists $C_T>0$ with 
	$$ |H^n_s- H^n_t| \; \le \; C_T \, n^{\gamma} \, |t-s|\quad \text{for all}\quad 0\le s,t \le T, \, n\ge 1.$$
	If $(H^n,X^n) \Rightarrow (H, X)$ on $(\D_{\R^d}[0,\infty), \dM) \times (\D_{\R^d}[0,\infty), \dM)$ with $X=(\sum_{j=0}^\infty c_j) Z$ or $X=(\sum_{j=0}^\infty c_j)  Z_{D^{-1}}$, and the $(H^n, X^n)$ and $(H^n,U^n)$ each satisfy \eqref{eq:oscillcond}, then $X$ is a semimartingale in the natural filtration generated by $(H,X)$ and it holds that
	$$ \left( X^n, \, \int_0^\bullet  \, H^n_{s-} \; \diff X^n_s\right) \; \; \Longrightarrow \; \;  \left( X, \, \int_0^\bullet \, H_{s-} \; \diff X_s\right) \qquad \text{ on } \quad \; (\D_{\R^{2d}}[0,\infty), \dM).$$
\end{theorem}

Similarly, if moving averages or correlated CTRWs have centered innovations and $1<\alpha\le 2$, they also serve as a natural vast group admitting $\operatorname{GD\,mod\,CI}$.

\begin{prop}[CTRWs with $1< \alpha\le 2$ and centered innovations are $\operatorname{GD\,mod\,CI}$] \ \label{prop:correlated_CTRW_are_GDmodCI} 
	Let $X^n$ be moving averages as in \eqref{eq:MovAv} or correlated CTRWs as in \eqref{defi:CTRW} with $1< \alpha\le 2$, satisfying \eqref{eq:technical_condition} and $\E[\theta_0]=0$. Then, the sequence $\psi^{-1}X^n$ is $\operatorname{GD\,mod\,CI}$ with 
	\begin{itemize}
		\item if $X^n$ is a CTRW: $\sigma^{n,i}_k:= \sigma^{n}_k:= \sum_{\ell=1}^{k} J_\ell/n=L_{k}/n$ and
		$$ V^{n,i}_t \; := \; - \, \frac{(\sum_{j=0}^\infty c_j)^{-1}}{n^{\frac \beta \alpha}} \Big(\, \sum_{\ell=i}^{\infty} c_{\ell}\, \Big) \, \theta_{N_{nt}-i+1}\, \ind_{\{N_{nt}\ge i\}},\quad \text{for all}\;\; i,k\ge 1; $$
		\item if $X^n$ is a moving average: $\sigma^{n,i}_k:=\sigma^n_k:= k/n$ and 
		$$ V^{n,i}_t \; := \; - \, \frac{(\sum_{j=0}^\infty c_j)^{-1}}{n^{\frac 1 \alpha}} \Big(\, \sum_{\ell=i}^{\infty} c_{\ell}\, \Big) \, \theta_{\lfloor nt \rfloor-i+1}\, \ind_{\{\lfloor nt \rfloor \ge i\}} ,\quad\text{for all}\;\;i,k\ge 1. $$
	\end{itemize}
\end{prop}

Just as we have used Theorem \ref{thm:Lipschitz_GDmodCA} and Propositions \ref{prop:correlate_CTRWs_are_GDmodCA} to deduce Theorem \ref{thm:result_int_conv_CTRWs_lipschitz_integrands}, we can similarly obtain a tailored result for moving averages and CTRWs with $1< \alpha \le 2$ on account of Theorem \ref{thm:Integral_converg_under_GDmodCI} and Proposition \ref{prop:correlated_CTRW_are_GDmodCI} (respectively Corollary \ref{cor:CTRWs_alpha=1_are_GDmodCI}). To simplify the notation for the next theorem, it will be understood that $L_k$ is the quantity defined before \eqref{defi:CTRW} if we consider CTRWs, while $L_k=k$ if we consider moving averages.

\begin{theorem}[Weak integral convergence for CTRWs II] \label{thm:result_int_conv_CTRWs_with_independence_cond}
	Let $X^n$ be a moving average as defined in \eqref{eq:MovAv} or a correlated CTRW as in \eqref{defi:CTRW}, with $1< \alpha\le 2$ and \eqref{eq:technical_condition}, which are adapted to the filtrations \eqref{eq:filtration_CTRWs}. Moreover, let $\E[\theta_0]=0$ and the $H^n$ be càdlàg processes adapted to the filtration \eqref{eq:filtration_CTRWs} such that for all $n,k\ge 1$,
	\begin{align}
		\sigma\big( H^n_{t \wedge \frac{L_k}{n}} \, : \, t\ge 0 \big)\, \vee \, \mathcal P_k \; \; \indep \; \; \mathcal I_k
		\label{eq:5.17}
	\end{align}
	where $\mathcal P_k = \sigma(J_\ell : 1\le\ell\le k) \vee  \sigma(\theta_j : 1\le j<k-\mathcal J)$ and
	$\mathcal I_k = \sigma(\theta_{k-\ell} : 0\le\ell\le(k-1)\wedge\mathcal J)$ with $\mathcal{J}=\operatorname{inf}\{j\ge 1 : c_i=0 \;\text{ for }\; i>j\}$. If $(H^n,X^n) \Rightarrow (H, X)$ on $(\D_{\R^d}[0,\infty), \dM) \times (\D_{\R^d}[0,\infty), \dM)$ with $X=(\sum_{j=0}^\infty c_j) Z$ or $X=(\sum_{j=0}^\infty c_j)  Z_{D^{-1}}$, then $X$ is a semimartingale in the natural filtration generated by $(H,X)$ and it holds that
	$$ \left( X^n, \, \int_0^\bullet  \, H^n_{s-} \; \diff X^n_s\right) \; \; \Longrightarrow \; \;  \left( X, \, \int_0^\bullet \, H_{s-} \; \diff X_s\right) \qquad \text{ on } \quad \; (\D_{\R^{2d}}[0,\infty), \dM).$$
\end{theorem}

In fact, for $\alpha=1$ it is not clear whether moving averages and correlated CTRWs are $\operatorname{GD\,mod\,CI}$. However, it turns out that, if their innovations are symmetric around zero, they satisfy a similar set of conditions---akin to $\operatorname{GD\,mod\,CI}$---still allowing for an application of Theorem \ref{thm:Integral_converg_under_GDmodCI}. For the sake of a concise presentation, we defer details on these conditions to Remark \ref{rem:modified_GDmodCI} in the section dedicated to proofs (Section \ref{sec:proofs}), and here we only state the result.

\begin{cor}\label{cor:CTRWs_alpha=1_are_GDmodCI}
	Let $X^n$ be a moving average as in \eqref{eq:MovAv} or a correlated CTRW as in \eqref{defi:CTRW}, with $\alpha=1$ and satisfying \eqref{eq:technical_condition}. If the law of $\theta_0$ is symmetric around zero, then Theorem \ref{thm:result_int_conv_CTRWs_with_independence_cond} also applies in this case.
\end{cor}

\begin{remark}[Weaker assumptions on the integrands] \label{rem:weaker_integrand_conditions}
	While we have chosen to center our treatment around weak convergence $(H^n,X^n) \Rightarrow (H,X)$ on the $M_1$ product space $(\D_{\R^d}[0,\infty), \dM) \times (\D_{\R^d}[0,\infty), \dM)$, we note that this condition can be relaxed to a form of finite-dimensional distributional convergence for the integrands, together with a certain regularity requirement regarding their boundedness and their maximal number of large increments (properties which are implied by tightness on the $M_1$ Skorokhod space). For brevity, we choose not to cover these conditions in any more detail and instead refer the interested reader to \cite[Prop.~2.12]{andreasfabrice_theorypaper} for more information. Nonetheless, it should be pointed out here that the Theorems \ref{thm:result_int_conv_uncorrelated_CTRWs}, \ref{thm:result_int_conv_CTRWs_lipschitz_integrands}, \ref{thm:result_int_conv_CTRWs_with_independence_cond}, and Corollary \ref{cor:CTRWs_alpha=1_are_GDmodCI} (as well as the more general results in Theorems \ref{thm:Lipschitz_GDmodCA} and \ref{thm:Integral_converg_under_GDmodCI}) generalise to this setting.
\end{remark}

\section{Convergence of SDE and SDDE models of anomalous diffusion}\label{sect:SDEs_SDDEs}

In Section \ref{subsect:anomalous_diffusion_stoch_int}, we discussed the connection between fractional Fokker--Planck--Kolmogorov equations and SDEs driven by time-changed L\'evy processes. As a concrete example, consider a spherically symmetric L\'evy process $Z$ (with characteristic exponent $\Psi(\xi)=-\kappa|\xi|^\alpha$) time-changed by the inverse of an independent standard $\beta$-stable subordinator $D$. Under suitable assumptions, an SDE of the form
\[
\diff X_t = \sigma(X_{t-})\diff Z_{D^{-1}_t},\quad X_0=x_0,
\]
will then have transition densities $p(t,x)$ governed by
\[
D_t^\beta p(t,x) = - \kappa (-\Delta)^{\frac{\alpha}{2}}\bigl(|\sigma(x)|^\alpha p(t,x) \bigr),\quad p(0,\cdot)=\delta_{x_0},
\]
where $-(-\Delta)^{\frac{\alpha}{2}}$ is a fractional Laplacian and $\kappa$ is the diffusivity constant. To connect this continuum formulation with CTRW driven models, it is natural to exploit the results on weak convergence of stochastic integrals established earlier in the paper, akin to the classical results on SDEs of \cite{slominski_stability} and \cite[Sect.~5]{kurtzprotter}. Beyond establishing a rigorous theoretical link to CTRW formulations, this also provides a tractable numerical scheme for the simulation of SDEs driven by time-changed L\'evy processes and, consequently, the associated fractional Fokker--Planck--Kolmogorov equations through a Monte Carlo procedure.

In Section \ref{sec:SDE} we focus on uncorrelated CTRWs for which there is $J_1$ convergence. Next, Section \ref{sect:SDDE} studies certain SDEs with a delay driven by strictly $M_1$ convergent CTRWs. The latter serves to illustrate that, for such equations, it is still possible to have a functional limit theorem in the $M_1$ topology. We achieve this by exploiting the framework of Section \ref{subsec:Control_through_independence}. To keep the analysis as concise as possible, we shall focus on moving averages rather than more general correlated CTRWs. In turn, the limiting equations are driven by a standalone L\'evy process rather than
a time-changed one.

\subsection{Functional limit theorems for SDEs driven by CTRWs}
\label{sec:SDE}

A series of papers \cite{jin_kobayashi, jum_kobayashi, Magdiarz_SDE1} have investigated weak and strong approximation schemes for SDEs of the general form
\begin{align} \label{eq:first_SDE}
	\begin{cases}
		\diff X_t = \mu(D^{-1}_t,X_{t})_- \diff D^{-1}_t  + \sigma(D^{-1}_t,X_{t})_- \diff B_{D^{-1}_t} \\
		X_0 = x,
	\end{cases}
\end{align}
for a Brownian motion $B$ and subordinator $D$. Motivated by stable limit theorems, we focus on stable subordinators for concreteness, but the main thing is simply that the subordinator is strictly increasing. A recent work \cite{jin_kobayashi_2} extends the results of \cite{jum_kobayashi} to allow for a traditional drift term, but only of the specific form $b(D_t^{-1})\diff t$. In general, $\diff t$ terms are problematic because they break the useful duality with non-time-changed SDEs explored in \cite{Kobayashi}.

As discussed in \cite[Ch.~6.6]{hahn}, the inverse subordinator and time-changed driver in \eqref{eq:first_SDE} are not Markovian and do not enjoy independent or stationary increments, so one cannot rely directly on the usual tools for classical SDEs such as the Euler method. On the other hand, CTRW approximations present themselves as a natural alternative (see e.g.~\cite[Ch.~5]{meerschaert} for the simulation of CTRWs and their use in simulating time-changed L\'evy processes).

Section 4 of the aforementioned work \cite{Kobayashi} studies theoretical properties, such as existence and uniqueness, for the more general class of SDEs
\begin{align} \label{eq:SDE} \tag{S}
	\begin{cases}
		\diff X_t = b(t,D^{-1}_t,X_{t})_- \diff t \, + \, \mu(t,D_t^{-1},X_{t})_- \diff D^{-1}_t  \, + \, \sigma(t,D_t^{-1}, X_{t})_- \diff Z_{D^{-1}_t} \\
		X_0 = x,
	\end{cases}
\end{align}
for suitable $Z$ and $D$. As we have done throughout, here we take $Z$ to be an $\alpha$-stable Lévy process and $D^{-1}$ to be the inverse of a $\beta$-stable subordinator, so that our analysis is aligned with the stable limit theory recalled in Section \ref{sect:CTRW_and_GD}. We are then interested in connecting \eqref{eq:SDE} with the approximating SDEs
\begin{align} \label{eq:SDE_approx} \tag{$\operatorname{S}_n$}
	\begin{cases}
		\diff X^n_t = b(t,D^n_t,X^n_{t})_- \diff t \, + \, \mu(t,D^n_t,X^n_{t})_- \diff D^n_t  \, + \, \sigma(t,D^n_t,X^n_{t})_- \diff Z^n_t \\
		X^n_0 = x,
	\end{cases}
\end{align}
where $D^n:= n^{-\beta} N_{n\bullet}$ and where the $Z^n$ denote uncorrelated uncoupled CTRWs defined as in \eqref{defi:CTRW} with $c_j=0$ for all $j\ge 1$. Of course, one could also consider weakly convergent initial conditions independent of the other stochastic inputs.

In terms of structural assumptions, we take the functions $b,\mu,\sigma:\R_+\times \R^2 \to \R$ to be continuous as well as satisfying, for all $T,R>0$, a strictly sublinear growth condition
\begin{align} \label{eq:sublinear_growth_condition_coeff_SDE}
	\sup{0\le t \le T} \, \sup{|\tilde y|\le R} \; (|b(t,\tilde y,y)| \vee |\mu(t,\tilde y,y)|\vee |\sigma(t,\tilde y,y)|) \; \le \; K \,|y|^{p}+ C,
\end{align}
where the constants $K,C>0$ and the exponent $p\in (0,1)$ may depend on $T,R$. As usual, $a\vee b:= \operatorname{max}\{a,b\}$. Imposing strict sublinearity as a growth condition in the spatial variable is important for us in deducing tightness of the solutions $X^n$ on Skorokhod space. One may also wish to allow for linear growth (i.e., $p=1$ in \eqref{eq:sublinear_growth_condition_coeff_SDE}), which one might suspect should be possible by way of Gr{\"o}nwall type arguments, but it is not immediately clear how to implement this as part of our approach. One could consider applying a general stochastic version of Gr{\"o}nwall's lemma, such as \cite[Thm.~1.2]{geiss}, but the lack in predictability and integrability of the integrators for classical upper bounds does not permit this in the present setting.

With the above assumptions, we obtain the following functional limit theorem for SDEs driven by CTRWs. The proof is given in Section \ref{sect:proofs_SDE}. 

\begin{theorem}[Convergence of SDEs driven by CTRWs] \label{thm:SDE_approx_result}
	Every subsequence of solutions $X^n$ to \eqref{eq:SDE_approx} has a further subsequence converging weakly on $(\D_{\R}[0, \infty), \, \dJ)$ to a solution $X$ to the SDE \eqref{eq:SDE}. In particular, if \eqref{eq:SDE} admits at most one solution, then it holds that $X^n \Rightarrow X$ on $(\D_{\R}[0, \infty), \, \dJ)$.
\end{theorem}

\subsection{Stochastic delay differential equations driven by moving averages} \label{sect:SDDE}
For a given delay parameter $r>0$, consider the stochastic delay differential equation (SDDE)
\begin{equation*}\label{eq:SDDE} \tag{SD}
	\begin{cases}
		\diff X_t = b(t,X_{t-r})_- \diff t + \sigma(t,X_{t-r})_- \diff Z_t \\
		X_s = \eta(s),\quad s\in[-r,0]
	\end{cases} 
\end{equation*}
whose driver, $Z$, is a stable L\'evy process. Such a model can capture phenomena driven by anomalous diffusion in the space variable, where the rates of change for the drift and diffusion depend on the state of the system some units of time into the past.

When $Z$ is a Brownian motion, discrete approximation schemes for delayed systems such as \eqref{eq:SDDE} have been examined in \cite{SDDE}. We also note that \cite{Hottovy} studies convergence of stochastic integrals for SDDEs, again in the Brownian setting, but there the focus is on characterising the small delay limit. Here we are interested instead in the stability of \eqref{eq:SDDE} with respect to the driver $Z$ when this arises as the weak scaling limit of moving averages in the $M_1$ topology. To this end, we shall rely on Theorem \ref{thm:result_int_conv_CTRWs_with_independence_cond}. One could also consider the corresponding systems driven by time-changed Lévy processes, and hence study the weak convergence in terms of general CTRWs. However, in order to keep the treatment succinct, we prefer to illustrate the procedure with the simpler moving averages as drivers.

Consider the SDDE from \eqref{eq:SDDE} with initial condition $\eta \in \D_{\R}[-r,0]$ and continuous $b,\sigma: \R_+ \times \R \to \R$ satisfying for any $T>0$ the boundedness condition
\begin{align}
	\sup{0\, \le \, t \, \le \, T} \; \sup{x \in \R} \; |b(t,x)| \vee |\sigma(t,x)| \; < \; \infty.  \label{eq:boundedness_cond_SDDEs}
\end{align}
Suppose that \eqref{eq:SDDE} has a unique solution $X$ and that $(X^n)_{n\ge 1}$ is a sequence of solutions to 
\begin{align} \begin{cases} \diff X^n_t \; &= \; b(t,X^n_{t-r})_-\, \diff t \; + \;c^{-1} \, \sigma(t, X^n_{t-r})_- \, \diff Z^n_t \\
		X^n_s \; &= \; \eta(s), \quad s \in [-r, 0] \end{cases} \label{eq:SDDE_approx} \tag{$\text{SD}_n$}\end{align}
where $c:=\sum_{k=0}^\mathcal{J} c_k$ and the $Z^n$ are moving averages as defined in \eqref{eq:MovAv}, with $1<\alpha\le 2$, $\E[\theta_0]=0$, and $c_j=0$ for all $j>\mathcal{J}$ and some $\mathcal{J}\ge 1$. Then, we know from Section \ref{sect:CTRW_and_GD}, that $c^{-1} Z^n \Rightarrow Z$ on the $M_1$ Skorokhod space. Notice that the $Z^n$ are pure jump processes with finitely many jumps on compact time intervals and therefore solutions to \eqref{eq:SDDE_approx} do not only exist but are also explicit. We are interested in whether the solutions to \eqref{eq:SDDE_approx} converge weakly to the unique solution of \eqref{eq:SDDE}, i.e. if $X^n \Rightarrow X$ on $(\D_{\R}[0,\infty), \dM)$.

\begin{theorem}[Convergence of SDDEs driven by moving averages] \label{thm:5.18}
	Every subsequence of solutions $X^n$ to \eqref{eq:SDDE_approx} has a further subsequence converging weakly on $(\D_{\R}[0, \infty), \, \dM)$ to a solution $X$ to the SDE \eqref{eq:SDDE}. In particular, if \eqref{eq:SDDE} admits at most one solution, then it holds that $X^n \Rightarrow X$ on $(\D_{\R}[0, \infty), \, \dM)$.
\end{theorem}
One can easily extend the previous result to SDDEs of the type
\begin{align*} \begin{cases} \diff X_t \; &= \; b(t,X_{t-r})_-\, \diff t \; + \; [\sigma(t, X_{t-r})_- \, + \, \tilde{\sigma}(t, X_{[t-r, t]})_-] \, \diff Z_t \\
		X_u \; &= \; \eta(u), \quad u \in [-r, 0] \end{cases} \quad \tag{$\operatorname{\tilde{SD}}$} \label{eq:ext_SDDE}\end{align*}
where $Z$, $b$, $\sigma$ and $\eta$ are defined as before in \eqref{eq:SDDE} and $X_{[t-r, t]}$ denotes the path segment $(X_s)_{t-r \, \le \, s \, \le \, t}$. Here $\tilde{\sigma}$ is of the form
$$ \tilde\sigma(t, \, X_{[t-r,t]}) \; = \; \int_{t-r}^t \, \Phi(t,s, X_s) \; \diff s$$
where $\Phi: [0,\infty) \times [-r,\infty) \times \R \to \R$ is such that for every $T>0$
$$ \sup{t \, \in \, [0,T]} \; \sup{s \, \in \, [t-r, t]} \; \sup{x \, \in \, \R} \; |\Phi(t,s,x)| \; < \; \infty.$$
We let $\Phi$ be continuous in the third component (i.e., $x\mapsto \Phi(t,s,x)$ is continuous for fixed $t,s$) and Lipschitz on compacts in the first component: for all $T>0$ there is $L_T>0$ such that 
$$ |\Phi(t,s,x) - \Phi(\tilde{t},s,x)| \; \le \; L_T \; |t- \tilde{t}|$$
for any $0\le t\le \tilde{t} \le T$, $x \in \R$ and $ \tilde t -r \le s \le t $. Obviously, the map $x\mapsto \tilde\sigma(\bullet,x_{[\bullet-r,\bullet]})$ is continuous from $ (\D_{\R}[0,\infty),\dM)$ into $\C_{\R}[0,\infty)$ with the topology of uniform convergence on compacts. Indeed, relative compactness follows from the Arzel\`a--Ascoli theorem, and pointwise convergence can easily be shown by dominated convergence, the continuity of $\Phi$ in the third component and the fact that $x_n \to x$ in $(\D_{\R}[0,\infty),\dM)$ implies in particular $x_n(s)\to x(s)$ for all but countably many $s\in [0,\infty)$. An example of such a $\tilde{\sigma}$ could be the convolution with a Lipschitz kernel $\rho:[-r,0]\to \R$, that is $\Phi(t,s,X_s):= \rho(s-t)\, f(X_s)$ for $t-r\le s \le t$, where $f$ is bounded and continuous.

Denote by ($\operatorname{\tilde{SD}_n}$) the approximation SDDEs to \eqref{eq:ext_SDDE} in analogy to how \eqref{eq:SDDE_approx} relates to \eqref{eq:SDDE}. Then we arrive at the following result.

\begin{cor}[Another class of SDDEs driven by moving averages]\label{cor:ext_SDDE_conv} 
	Theorem \ref{thm:5.18} holds true with \eqref{eq:SDDE_approx} replaced by ($\operatorname{\tilde{SD}_n}$), and \eqref{eq:SDDE} by \eqref{eq:ext_SDDE}.
\end{cor}

\section{Proofs}
\label{sec:proofs}

\subsection{Proofs pertaining to Section \ref{subsec:convergence_uncorrelated_CTRWs}}

\begin{proof}[Proof of Proposition \ref{prop:uncorr_CTRW_martingales}]
	Let $n\ge 1$ and $0\le s \le t$. First, we establish the integrability of $N_{nt}$. Note that
	$$ 0 \le  N_{nt} \le  \tilde{N}_b(nt)  :=  \max{} \Bigl\{ m \ge 0 \; : \; \sum_{k=1}^m  b  \ind_{\{J_k > b\}} \, \le \, nt \Bigr\}\; = \; \sum_{m=1}^\infty \, \ind_{\left\{ \sum_{k=1}^m \, \ind_{\{J_k > b\}} \le  \frac{nt}{b}\right\}}$$
	for any $b> 0$.  Choose $b> nt$ and note that then 
	\begin{align*}
		\E[\tilde{N}_b(nt)] \le  \sum_{m=1}^\infty \Pro \biggl( \sum_{k=1}^m  \ind_{\{J_k > b\}}  <  1 \biggr) = \sum_{m=1}^\infty  \Pro \Bigl( \intersection{k=1}{m}  \{J_k \le b\} \Bigr) = \sum_{m=1}^\infty  \Pro(J_1 \le b)^m  < \infty
	\end{align*}
	since $\Pro(J_1\le b)<1$ (see e.g.~\cite{feller}, \cite[Thm.~4.5.2]{whitt}). Hence, $0  \le \E[N_{nt}]  \le \E[\tilde{N}_b(nt)]  < \infty$.
	Therefore, we obtain 
	$$ \E [|M^n_t|] \le  \E \biggl[\,  \sum_{k=1}^{N_{nt}}  |\zeta^n_k| \ind_{\{|\zeta^n_k|\le a\}} \, \biggr] +  \E \biggl[\,  \sum_{k=1}^{N_{nt}}  \E[|\zeta^n_k| \ind_{\{|\zeta^n_k|\le a\}}| J_k] \, \biggr] \le  2a \E[N_{nt}]  < \infty.$$
	Clearly, $M^n$ is adapted to the filtration $(\mathcal{F}^n_t)_{t\ge 0}$. Set $C_{n,k}:= \E[\zeta^n_k\ind_{\{|\zeta^n_k|\le a\}}|J_k]$ and define 
	$$ \mathcal{H}_k \; := \; \sigma(\zeta_j : 1\le j<k) \vee \sigma(J_j: j\ge 1) \vee \mathcal{G}^n_\infty, \qquad k\ge 1,$$
	where $\mathcal{G}^n_\infty=\sigma( \cup_{t\geq 0}\mathcal{G}^n_t)$. Note that for any $A\in \mathcal{F}^n_s$ we have that $\E[(M^n_t-M^n_s) \ind_A]$ equals
	\begin{align*}
		\sum_{k=1}^\infty \;&\E \big[ \,\ind_{A\cap \{L_k>ns\}} \ind_{\{L_k\le nt\}} \big( \zeta^n_k \ind_{\{|\zeta^n_k|\le a\}} -C_{n,k}\big) \big]\\
		&=\;  \sum_{k=1}^\infty\mathbb E \Big[ \mathbf 1_{A\cap\{L_k>ns\}} \mathbf 1_{\{L_k\le nt\}} \, \mathbb E\big[ \zeta_k^n\mathbf 1_{\{|\zeta_k^n|\le a\}} - C_{n,k} \,\big|\,\mathcal H_k \big]\Big] \; = \; 0.
	\end{align*}
	Indeed, it is straightforward to show $A\cap \{L_k>ns\} \in \mathcal{H}_k$. Heuristically, on the event that by time $s$ the $k$-th jump has not yet occured, all the information about the jumps available comes from $\zeta_1,...,\zeta_{k-1}$. The bound on the jumps of the $M^n$ follows from $|\Delta N_{nt}|\le 1$ and the boundedness of the $\zeta^n_k \ind_{\{|\zeta^n_k| \le a\}}$.
\end{proof}

Towards a proof of Proposition \ref{prop:aux_result_CTRW_has_GD}, we need some auxiliary results. Recall that $h(x)=x\ind_{|x|\le a} + \operatorname{sgn}(x)a\ind_{|x|> a}$ where we fixed $a\ge 1$ before \eqref{eq:uncorrelated_CTRW_decomp}. As in the sequel, we will make use of characteristics of semimartingales, we refer to \cite[Sec.~II.2--II.4]{shiryaev} for a detailed introduction to this concept. If one accepts Theorem \ref{thm:5.6} and the fact that the function $\psi_{b,c,\nu}$ in \eqref{eq:char_func_infdiv} uniquely defines the triple $(b,c,\nu)$, known as the characteristics of $\mu$, the proofs can be comprehended without any further knowledge of the concept of characteristics.

\begin{theorem}[A partial version of {\cite[Thm.~VII.2.9]{shiryaev}}] \label{thm:5.6} Let $(\mu_n)_{n\ge 0}$ be infinitely divisible probability measures on $\R$ with characteristics $(b_n, c_n, \nu_n)$---i.e., the characteristic function of $\mu_n$ is of the form $\varphi_{\mu_n}(u)= \exp(\psi_{b_n, c_n,\nu_n}(u))$---with 
	\begin{align} \psi_{b_n, c_n,\nu_n}(u) \; = \; iub_n \, - \, \frac12c_n u^2 \, + \, \int_{\R \setminus \{0\}} \, e^{iux} -1 - iuh(x)\; \nu_n(\diff x) \label{eq:char_func_infdiv}\end{align}
	where $b_n \in \R$, $c_n\ge 0$ and $\nu_n$ a measure on $\R$ such that $\nu_n(\{0\})=0$ and $\int x^2\wedge 1 \, \nu_n(\diff x) < \infty$. If $\mu_n \Rightarrow \mu_0$, then $b_n \to b_0$ and $\langle g ,\nu_n\rangle \to \langle g ,\nu_0 \rangle$ for all continuous, bounded functions $g$ satisfying $g(x)= \oland(x^2)$ for $x\to 0$.
\end{theorem}

The next lemma formalises the fact that the random variables $Z^n:=\zeta^n_1- \E[h(\zeta^n_1)]$ possess very beneficial asymptotic properties. 
\begin{lem}\ \label{lem:5.7} Let $Z^n= \zeta^n_1- \E[h(\zeta^n_1)]$ and denote $\varphi_{Z^n}(u)=\E[ e^{iuZ^n}]$ the characteristic function of $Z^n$. Then, it holds that
	\begin{enumerate}[(a)]
		\item $\operatorname{sup}_{n\ge 1} n^\beta \Pro(|Z^n|> c)< \infty$ for all $c>0$;
		\item $n^\beta \E[h(Z^n)] \to 0$ as $n\to \infty$.
		\item for all $\gamma>0$,  $ \operatorname{sup}_{n\geq 1} \operatorname{sup}_{|u|<\gamma} \, n^\beta | \varphi_{Z^n}(u) -1| < \infty$
	\end{enumerate}
\end{lem}
\begin{proof}
	Recall that $\Pro(|\zeta^n_1|>x) = \Pro(|\theta_1| > xn^{\beta/\alpha}) = \Oland(n^{-\beta}x^{-\alpha})$ for $x\to +\infty$. Now, for every $0<\varepsilon< a$, we have $\E[h(|\zeta^n_1|)] \le \varepsilon + a \Pro(|\zeta^n_1|>\varepsilon) \to \varepsilon$ as $n\to \infty$ and since $\varepsilon>0$ was chosen arbitrary, we obtain $\E[h(|\zeta^n_1|)] \to 0$ as $n\to \infty$. Therefore, we deduce the first claim since $ n^\beta \Pro(|Z^n|> c) \; \le \; n^\beta \Pro(|\zeta^n_1|> c/2) \; + \; n^\beta\Pro(|\E[h(\zeta^n_1)]|>c/2)$ for all $c>0$ and the second summand equals zero for large enough $n$. Hence, we even obtain $ \limsup_{n\to \infty} n^\beta \Pro(|Z^n|>c)  \le C(2/c)^\alpha$, for some $C\geq 0$.
	
	For (b), let $\varepsilon>0$ and note that $h$ satisfies $|h(x)-h(y)| \le |x-y|$ for all $x,y \in \R$. Choose $n$ large enough such that $\E[h(|\zeta^n_1|)]\le (a/2) \wedge \varepsilon$. We have $\E[h(Z^n)]=\E[h(Z^n)-h(Z^n+ \E[h(\zeta^n_1)]) +\E[ h(\zeta^n_1)]]$ and the quantity $h(Z^n)-h(Z^n+ \E[h(\zeta^n_1)]) +\E[ h(\zeta^n_1)]$ is equal to zero if $|Z^n|\le a/2$. Therefore, by the Lipschitz continuity of $h$,
	$$ | n^\beta \E[h(Z^n)] | \; \le \; n^\beta \E[(\varepsilon + \E[ |h(\zeta^n_1|)] ) \ind_{\{ |Z^n|>a/2\}}] \; \le \; 2\varepsilon n^\beta \Pro(|Z^n|>a/2)$$
	and the result follows from (a) above and the fact that $\varepsilon>0$ was chosen arbitrary.
	
	Finally, we are going to prove (c). Let $\gamma>0$. For $|x|\le1$, we have $h(x)=x$, so the fundamental theorem of calculus and $|e^{iy}-1|\le|y|$ give $\bigl|e^{iux}-1-iuh(x)\bigr|=\bigl|iux\int_0^1(e^{itux}-1)\,\diff t\bigr|\le u^2x^2/2$. Using also that $h$ is bounded, this gives $|e^{iux}-1-iuh(x)|\le C_\gamma(|x|^2\wedge1)$ for all $x\in\R$, $|u|\le\gamma$, for some $C_\gamma>0$. Consequently,
	$$\operatorname{sup}_{|u|<\gamma}n^\beta|\varphi_{Z^n}(u)-1|
	\;\le\;\gamma n^\beta|\E[h(Z^n)]|+C_\gamma n^\beta\E[|Z^n|^2\wedge1],$$
	where the first term on the right-hand side tends to zero by (b). For the second term, the tail estimate for $\alpha<2$ in the proof of (a), and the fact that $\E[|\theta_1|^2]<\infty$ for $\alpha=2$, give $\operatorname{sup}_{n\ge1}n^\beta\E[|\zeta^n_1|^2\wedge1]<\infty$. Since $|\E[h(\zeta^n_1)]|^2\le\E[|h(\zeta^n_1)|^2]\le a^2\E[|\zeta^n_1|^2\wedge1]$, the same bound holds with $Z^n$ in place of $\zeta^n_1$, so we obtain (c).
\end{proof}

Let $Y^n:= \sum_{k=1}^{\lfloor n^\beta \rfloor} \zeta^n_k$ and recall $\mathcal{L}(Y^n) \to \mu$, where $\mu$ is the law of an $\alpha$-stable Lévy process at time $t=1$. Denote the characteristics of $\mu$ by $(b,c,\nu)$.

\begin{proof}[Proof of Proposition \ref{prop:aux_result_CTRW_has_GD}] It suffices to consider $d=1$. The proof follows very closely that of \cite[Lem.~VII.2.43]{shiryaev}, but has been simplified and adapted to our particular setting.
	Let $u \in \R$. Note that we can rewrite the characteristic function $\varphi_{Y^n}$ by
	$$ \varphi_{Y^n}(u) \; = \; \E[ \exp(iuY^n)] \; = \; \E[\exp(iu \, \zeta^n_1)]^{n^\beta} \; = \; \E[\exp(iu \, Z^n)]^{n^\beta}\;  \exp\{iu \, n^\beta \E[h(\zeta^n_1)]\},$$
	where we abuse notation and write $n^\beta$ in place of $\lfloor n^\beta \rfloor$. Next, we add a convenient factor, rewriting $\varphi_{Y^n}$ as
	\begin{align}
		\varphi_{Y^n}(u) \; = \; \Lambda_n(u) \, \exp\left\{iu \, n^\beta \E[h(\zeta^n_1)] + n^\beta \int e^{iuZ^n} -1 - iu\, h(Z^n) \, \diff \Pro\right\} \label{eq:Levy_Khintchine_Yn}
	\end{align}
	where $\Lambda_n$ is defined as 
	\begin{align*}
		\Lambda_n(u) \; = \; \E[\exp(iu \, Z^n)]^{n^\beta}  \exp\left\{- n^\beta \E \left[e^{iuZ^n} -1 - iu\, h(Z^n) \right] \right\}.
	\end{align*}
	Now, if the first factor of \eqref{eq:Levy_Khintchine_Yn} tends to $1$, then, since we know that $\varphi_{Y^n} \to \varphi_\mu$ pointwise, we can conclude that the second factor then has to converge to $\varphi_\mu$. As displayed in \eqref{eq:char_func_infdiv}, the second factor represents the characteristic function of an infinitely divisible law with characteristics $(b_n,c_n,\nu_n)=(n^{\beta}\E[h(\zeta^n_1)], 0,n^\beta(\Pro\circ(Z^n)^{-1})\vert_{\R\setminus\{0\}})$, where $\nu_n=n^\beta(\Pro\circ(Z^n)^{-1})\vert_{\R\setminus\{0\}}$ satisfies the conditions in \eqref{eq:char_func_infdiv}. Finally, Theorem \ref{thm:5.6} then yields the desired convergence $b_n \to b$. According to Lemma \ref{lem:5.7} we do know about the asymptotic behaviour of $n^\beta |\varphi_{Z^n}-1|$. Thus, we continue 
	\begin{align*}
		\Lambda_n(u) \; &= \; \E[\exp(iu \, Z^n)]^{n^\beta}  \exp\left\{- n^\beta \E \left[e^{iuZ^n} -1 - iu\, h(Z^n) \right] \right\} \\
		&= \; \left( \E[\exp(iu \, Z^n) -1] +1 \right)^{n^\beta}\,  e^{- n^\beta \E[\exp(iu \, Z^n)-1]}   \exp\left\{iu \, n^\beta \E \left[h(Z^n) \right] \right\} \\
		&= \; (\lambda_n +1)^{n^\beta}  \, e^{-n^\beta \lambda_n} \,\exp\left\{ iu \, n^\beta \E \left[h(Z^n) \right] \right\} 
	\end{align*}
	and $\lambda_n:= \E[\exp(iu \, Z^n) -1]$. By Lemma \ref{lem:5.7}(b), $\exp\{ iu \, n^\beta \E [h(Z^n) ] \} \to 1$ as $n\to \infty$. Hence, it only remains to show $(\lambda_n +1)^{n^\beta} \, e^{-n^{\beta} \lambda_n} \to 1$: by Lemma \ref{lem:5.7}(c), $\operatorname{sup}_{n\ge1}n^\beta|\lambda_n|<\infty$. Without loss of generality, let $n\ge 1$ large enough such that $|\lambda_n|\le 1/2$. Note that the principal branch $\log$ of the complex logarithm has Taylor expansion $\log(1+x)= \sum_{k=1}^\infty \frac{(-1)^{k+1}}{k} x^k$ for all $|x|<1$ and it is elementary to show that this gives ground to the inequality $|\log(1+x)-x| \le 2|x|^2$ for all $|x|\le 1/2$. Hence, we obtain
	\begin{align*}
		|n^\beta (\log(1+\lambda_n)-\lambda_n)| \; \le \; 2n^\beta  |\lambda_n|^2 \; \conv{n\to \infty} \; \; 0.
	\end{align*}
	Finally, observe that therefore
	$$ (\lambda_n +1)^{n^\beta}  \, e^{-n^\beta \lambda_n} \; = \; \exp \{ n^\beta \log(1+\lambda_n) - n^\beta \lambda_n\} \; \to \; 1.$$
	Thus, we obtain that $\Lambda_n(u) \to 1$ and eventually as $\varphi_{Y^n}(u) \to \varphi_\mu(u)$, this gives us
	\begin{align*}
		\exp \{ \psi_{n^\beta \, \E[h(\zeta^n_1)] \, , \,  0\, , \,  \nu_n}(u)\}  &=  \exp\Bigl\{iu \, n^\beta \E[h(\zeta^n_1)] + n^\beta \!\!\int \!e^{iuZ^n}\! -1 - iu\, h(Z^n) \, \diff \Pro\Bigr\} \\
		&\conv{n\to \infty}  \; \; \varphi_{\mu}(u) =  \exp \{ \psi_{b, c, \nu}(u)\}
	\end{align*}
	which, by Theorem \ref{thm:5.6}, in particular implies $n^\beta\E[h(\zeta^n_1)] \to b$.
\end{proof}

\begin{proof}[Proof of Theorem \ref{thm:CTRW_has_GD}]
	In the case $0<\alpha<1$ \eqref{eq:Mn_An_condition} is given by Lemma \ref{lem:uncoupled_CTRW_alpha<1_tight_total_variation}. Let $1\le \alpha\le 2$. Consider the decomposition \eqref{eq:uncorrelated_CTRW_decomp}, where $M^n$ is a martingale with uniformly bounded jumps almost surely, by Proposition \ref{prop:uncorr_CTRW_martingales}. In order to show that the $X^n$ admit \eqref{eq:Mn_An_condition}, it therefore suffices to show that the process $A^{n}:= \sum_{k=1}^{N_{n\bullet}}\,  \E[\zeta^n_k \ind_{\{|\zeta^n_k|\le a\}}|J_k] + \sum_{k=1}^{N_{n\bullet}} \zeta^n_k \ind_{\{|\zeta^n_k|>a\}} $ is of tight total variation on compact intervals. To this end, let $\varepsilon,t>0$ and choose $K_\varepsilon>0$ such that $ \Pro(|n^{-\beta}N_{n\bullet}|^*_t > K_\varepsilon)\le \varepsilon$, which is a well-known consequence of the tightness on Skorokhod space, here \eqref{eq:time_convergence_to_subordinator}. Thus, since $s \mapsto N_{ns}$ is non-negative and increasing, $\Pro(N_{nt}>n^\beta K_\varepsilon)\le \varepsilon$. Note that $A^n$ is a pure jump process with jumps happening at the jump times of $N_{n\bullet}$, and the total variation of $A^n$ can therefore be bounded by
	$$ \operatorname{TV}_{[0,t]}(A^n) \; \leq \; \sum_{k=1}^{N_{nt}}\, |\E[\zeta^n_k \ind_{\{|\zeta^n_k|\le a\}} | J_k]| \, + \,\sum_{k=1}^{N_{nt}} |\zeta^n_k| \ind_{\{|\zeta^n_k|>a\}} .$$
	On $\{N_{nt}\le n^\beta K_\varepsilon\}$, this can be further estimated by
	$$ \operatorname{TV}_{[0,t]}(A^n) \; \le \; \sum_{k=1}^{\lfloor K_\varepsilon n^\beta\rfloor} \big|\E[\zeta^n_k \ind_{\{|\zeta^n_k|\le a\}}|J_k]\big| \, + \, 2\Gamma_a^t(X^n) |X^n|^*_t,$$
	where $\Gamma_a^t(X^n)$ denotes the maximal number of jumps of $X^n$ on $[0,t]$ larger than $a$. Without loss of generality, assume that $K_\varepsilon$ was chosen such that $\Pro( n^{-\beta} N_{nt} \vee |X^n|^*_t \vee \Gamma^t_a > K_\varepsilon)\le \varepsilon$. Indeed, on $\{N_{nt}\le Kn^\beta\}$, both $|X^n|^*_t$ and $\Gamma_a^t(X^n)$ are bounded by the corresponding quantities on $[0,K]$ for the $J_1$-tight partial sums $s\mapsto\sum_{k=1}^{\lfloor n^\beta s\rfloor}\zeta_k^n$ (for the latter see e.g. \cite[Cor.~A.9]{andreasfabrice_theorypaper}). In the case of uncorrelated uncoupled CTRWs or moving averages, $\E[\zeta^n_1 \ind_{\{|\zeta^n_1|\le a\}}|J_1]=\E[\zeta^n_1 \ind_{\{|\zeta^n_1|\le a\}}]$ and from Proposition \ref{prop:aux_result_CTRW_has_GD} we deduce the existence of a constant $C>0$ such that $n^\beta |\E[\zeta^n_1 \ind_{\{|\zeta^n_1|\le a\}}]|< C$ for all $n\ge 1$. Finally, this implies that, for all $n\ge 1$, we have
	$$ 1-\varepsilon \; \le \; \Pro( n^{-\beta} N_{nt} \vee |X^n|^*_t \vee \Gamma^t_a \le K_\varepsilon) \; \le \; \Pro(\operatorname{TV}_{[0,t]}(A^n) \; \le \; K_\varepsilon(C+2K_\varepsilon)).$$ 
	For coupled CTRWs with $1<\alpha\le2$, \eqref{eq:condition_centering_condition_coupled_case} gives
	$$\operatorname{sup}_{n\ge1} n^\beta \E\Big[\big|\E[\zeta^n_1\ind_{\{|\zeta^n_1|\le a\}}|J_1]\big|\Big]
	\;\le\; \operatorname{sup}_{n\ge1} n^\beta\E\big[|\zeta^n_1|\ind_{\{|\zeta^n_1|>a\}}\big]\;<\;\infty,$$
	by the tail probability $\Pro(|\theta_1|>x)=\Oland(x^{-\alpha})$. For $\alpha=1$, it follows directly from \eqref{eq:condition_centering_condition_coupled_case} that the supremum on the left-hand side is finite, since $\zeta^n_1=n^{-\beta}\zeta_1$. By Markov's inequality, this also implies that the $A^n$ are of tight total variation.
\end{proof}

\subsection{Proofs pertaining to Section \ref{subsec:Correlated_CTRWs_do_not_blend_in}}

\begin{lem} \label{lem:uncoupled_CTRW_alpha<1_tight_total_variation}
	Let $0<\alpha<1$. Then, zero-order moving averages, uncorrelated uncoupled CTRWs and (uncorrelated) coupled CTRWs are of tight total variation on compacts.
\end{lem}
\begin{proof}[Proof of Lemma \ref{lem:uncoupled_CTRW_alpha<1_tight_total_variation}] We are only going to prove the lemma for CTRWs. The proof for moving averages can be conducted in a similar way. Let $X^n$ be defined as in \eqref{defi:CTRW} with $c_0=1$ and $c_j=0$ for all $j\ge 1$. As a pure jump process, the total variation of $X^{n}$ is
	$$ \operatorname{TV}_{[0,t]}(X^{n}) \; = \; \sum_{k=1}^{N_{nt}} \;  |\zeta_k /n^{\frac \beta \alpha}|\, .$$
	Fix $t>0$ and let $\varepsilon>0$. Given the tightness from \eqref{eq:time_convergence_to_subordinator}, there exists $K_{t,\varepsilon}>0$ such that
	$$ \Pro \Big( \sup{s \in [0,t]} \, n^{-\beta} N_{ns} \; > \; K_{t,\varepsilon} \Big)\;= \;  \Pro \Big(  N_{nt} \; > \; n ^\beta K_{t,\varepsilon}  \Big)  \; \le \; \frac{\varepsilon}{3}.$$
	Hence,
	\begin{align}
		\Pro \Big(  &\operatorname{TV}_{[0,t]}(X^{n}) \; > \; C \Big) \; \le \; \frac\varepsilon 3 \; + \; \Pro \Big( n^{-\frac \beta \alpha} \sum_{k=1}^{ \lfloor n^\beta K_{t,\varepsilon}  \rfloor } \; | \zeta_k  | \; > \; C\Big) \label{eq:2.10} \\
		&\le \; \frac\varepsilon 3 \; + \; \Pro \Big( \sum_{k=1}^{ \lfloor n^\beta K_{t,\varepsilon}  \rfloor } \; |  \zeta_k | \; > \; C\, n^{\frac \beta \alpha} \;\;, \; \; \max{k=1,...,\, \lfloor n^\beta K_{t,\varepsilon} \rfloor} |\zeta_k | \, \le C\, n^{\frac \beta \alpha} \Big) \notag \\
		& \qquad \;\;  + \; \Pro \Big(  \max{k=1,...,\, \lfloor n^\beta K_{t,\varepsilon} \rfloor} | \zeta_k| \; > \; C\, n^{\frac \beta \alpha} \Big) \label{eq:2.11}
	\end{align}
	for any $C>0$. Firstly, there is
	\begin{align*}
		\Pro \Big(\max{k=1,...,\, \lfloor n^\beta K_{t,\varepsilon} \rfloor} | \zeta_k| \; > \; C\, n^{\frac \beta \alpha} \Big) \; &\le  \; K_{t,\varepsilon}  \; n^\beta  \; \, \Pro \Big( \, |\zeta_1| \; > \; C \; n^{\frac \beta \alpha} \Big) \\
		&= \; \frac{K_{t,\varepsilon}}{C^\alpha} \; \Big[ \, C^\alpha \; n^\beta \; \Pro \Big( \, |\zeta_1| \; > \; C \; n^{\frac \beta \alpha} \Big) \, \Big]. 
	\end{align*}
	Due to the alternative characterisation of the domain of attraction of an $\alpha$-stable law (see \cite[(1.1)]{berger} and \cite[Thm.~4.5.2]{whitt}), the tail probabilities of $\zeta_1$ behave with regularity $\Pro(|\zeta_1|>x) \sim c_\alpha x^{-\alpha}$.  Hence, we can choose $n_0 \ge 1$ such that for all $n\ge n_0$ it holds
	$$C^\alpha \, n^\beta \; \Pro \Big( \, |\zeta_1| \; > \; C \; n^{\frac \beta \alpha} \Big) \; \; \le \; \; 2c_\alpha$$
	implying that for all $n\ge n_0$,
	\begin{align}
		\Pro \Big(  \max{k=1,...,\, \lfloor n^\beta K_{t,\varepsilon} \rfloor} | \zeta_k| \; > \; C\, n^{\frac \beta \alpha} \Big)\;  \le  \; \frac{2c_\alpha K_{t,\varepsilon}}{C^\alpha} \label{eq:(B)}\, .
	\end{align}
	Concerning the second summand of \eqref{eq:2.11}, we will make use of the Fuk--Nagaev inequality (see \cite[Theorem~5.1(i)]{berger} and take into account the remark in \cite[p.~12]{berger} that the inequality only requires the bound $\Pro(|\zeta_1|>x)\le cx^{-\alpha}$) with $X_k=|\zeta_k|$ and $x=y=C n^{\frac \beta \alpha}$ to obtain 
	\begin{align}
		\Pro \Big( \sum_{k=1}^{ \lfloor n^\beta K_{t,\varepsilon}  \rfloor } \; |  \zeta_k | \; > \; C\, n^{\frac \beta \alpha} \;\;, \; \; \max{k=1,...,\, \lfloor n^\beta K_{t,\varepsilon} \rfloor} |\zeta_k | \, \le C\, n^{\frac \beta \alpha} \Big) \; \; &\le \; \;  e\lambda^{-1}\; \lfloor n^{\beta} \, K_{t,\varepsilon} \rfloor \; (C \, n^{\frac{\beta}{\alpha}})^{-\alpha} \notag \\
		&\le \; e\lambda^{-1} \; \frac{K_{t,\varepsilon}}{C^{\alpha}} \label{eq:(A)}
	\end{align}
	for some constant $\lambda>0$ independent of all other quantities. Therefore, choosing $n\ge n_0$ and $C:=C_{1,t,\varepsilon} \ge (\frac{3}{\varepsilon}\operatorname{max}\{2c_\alpha K_{t,\varepsilon}, e\lambda^{-1}K_{t,\varepsilon}\})^{1/\alpha}$, we are able to bound \eqref{eq:(A)} and \eqref{eq:(B)} each by $\varepsilon/3$ and therefore \eqref{eq:2.11} by $\varepsilon$. Since all $X^n$ are of local finite variation, the family $\{\operatorname{TV}_{[0,t]}(X^{n}) \,:\, n< n_0\}$ is obviously tight and we deduce the existence of $C_{2,t,\varepsilon}>0$ such that $\Pro(  \operatorname{TV}_{[0,t]}(X^{n}) \; > \; C_{2,t,\varepsilon} )\le \varepsilon$ for all $n< n_0$. Finally, this yields $\operatorname{sup}_{n\ge 1} \Pro(\operatorname{TV}_{[0,t]}(X^{n}) > C_{t,\varepsilon})\le \varepsilon$ for $C_{t,\varepsilon}:= \operatorname{max}\{C_{1,t,\varepsilon},C_{2,t,\varepsilon}\}$, and since $t,\varepsilon>0$ were arbitrary the tightness of the total variation of the $X^n$.
\end{proof}

\begin{proof}[Proof of Proposition \ref{prop:uncoupled_correlated_CTRW_tight_total_variation}]
	Consider the decompositions \eqref{eq:defi_U^n_in_CTRW_decomp} and \eqref{eq:defi_V^n_in_CTRW_decomp}. Note that 
	$$\operatorname{TV}_{[0,t]}(U^n) \;  \le \; \operatorname{TV}_{[0,t]}(U^{n,1}) \, + \, \psi^{-1}n^{-\beta/\alpha} \sum_{k=0}^{\infty} \Big( \sum_{\ell=1}^{N_{nt}} c_{k+\ell}\Big)\, |\theta_{-k}|$$
	and $\operatorname{TV}_{[0,t]}(V^n) \le 2\operatorname{TV}_{[0,t]}(U^{n,1})$, where $U^{n,1}_t= n^{-\beta/\alpha} \sum_{k=1}^{N_{nt}} \theta_k$ denotes the uncorrelated CTRW. Since $\operatorname{TV}_{[0,t]}(U^{n,1})$ is tight by the proof of Lemma \ref{lem:uncoupled_CTRW_alpha<1_tight_total_variation}, it only remains to show that $n^{-\beta/\alpha} \sum_{k=0}^{\infty} ( \sum_{\ell=1}^{N_{nt}} c_{k+\ell})|\theta_{-k}|$ is tight on compact time intervals. The case of only finitely many $c_j\neq 0$ is straightforward. Indeed, let $J\ge 1$ be such that $c_j=0$ for all $j> J$, then $n^{-\beta/\alpha} \sum_{k=0}^{\infty} ( \sum_{\ell=1}^{N_{nt}} c_{k+\ell})|\theta_{-k}|\le n^{-\beta/\alpha} \sum_{k=0}^{J} ( \sum_{\ell=1}^{J} c_{\ell})|\theta_{-k}| \to 0$ almost surely as $n\to \infty$. Assume now that there are infinitely many $c_j\neq 0$. Let $t\ge 0$, $\varepsilon>0$ as well as $0< \rho<\alpha<1$ such that $\sum_{j=0}^\infty c_j^\rho< \infty$, and set $b_{k,m}=\sum_{\ell=1}^{m}c_{k+\ell}$ with $m=\lfloor K_\varepsilon n^\beta\rfloor$. The tail bound $\Pro(|\theta_0|>x)\le Cx^{-\alpha}$ implies
	$\E[1\wedge u|\theta_0|]\le C u^\alpha$ for $u\ge0$, which we can use with $u=b_{k,m}/(R\,n^{\beta/\alpha})$ to see that
	\begin{align*}
		\Pro  \biggl( n^{-\beta/\alpha} \, \sum_{k=0}^{\infty} \, \biggl( &\sum_{\ell=1}^{N_{nt}} c_{k+\ell}\biggr)\, |\theta_{-k}| \, > \, R \biggr) \;\le \; \frac \varepsilon 2   +  \Pro \biggl( n^{-\beta/\alpha} \, \sum_{k=0}^{\infty} \, b_{k,m}\, |\theta_{-k}| \, > \, R \biggr)   \\
		&\leq\;\frac{\varepsilon}{2}
		+ \sum_{k=0}^\infty \mathbb E \biggl[
		1\wedge\frac{b_{k,m}|\theta_{-k}|}{R\,n^{\beta/\alpha}}
		\biggr] \;\leq \;\frac{\varepsilon}{2}
		+\frac{C}{R^\alpha n^\beta}\sum_{k=0}^{\infty} b_{k,m}^{\alpha},
	\end{align*}
	where, as in \eqref{eq:2.10}, $K_\varepsilon>0$ is chosen so that
	$\operatorname{sup}_{n\ge1}\Pro(n^{-\beta}N_{nt}>K_\varepsilon)\le\varepsilon/2$. Since $\alpha<1$, we get $\sum_{k=0}^\infty b_{k,m}^\alpha\le\sum_{j=1}^{\infty}(j\wedge m)c_j^\alpha$. We also have $\sum_j c_j^\alpha<\infty$, so $m^{-1}\sum_{k=0}^\infty b_{k,m}^\alpha$ vanishes as $m\rightarrow\infty$, while \(m/n^\beta\to K_\varepsilon\). Thus, the last term in the above display tends to zero for every $R>0$, from which we can deduce the desired tightness.
\end{proof}

\begin{proof}[Proof of Theorem \ref{thm:result_int_conv_correlated_CTRWs_alpha<1}]
	Follows directly from Proposition \ref{prop:uncoupled_correlated_CTRW_tight_total_variation} and Theorem \ref{thm:3.19}.
\end{proof}

\subsection{Proofs pertaining to Lemma \ref{lem:decomp_U^n_part_has_GD} and Section \ref{subsec:Direct_control_of_variation}}

\begin{proof}[Proof of Lemma \ref{lem:decomp_U^n_part_has_GD}]
	As outlined at the beginning of Section \ref{sec:generalised_framework}, it suffices to show that the second summand $U^{n,2}$ of $U^n$, defined in \eqref{eq:defi_U^n_in_CTRW_decomp}, is of tight total variation on compact time intervals. If $1<\alpha \le 2$, choose $\rho=1$ and, in the case $\alpha=1$, let $0<\rho<1$ such that \eqref{eq:technical_condition} is satisfied. A simple application of Markov's inequality, monotone convergence and the identical distribution of the $\theta_k$ yield
	$$\Pro \Big( \operatorname{TV}_{[0,t]}(U^{n,2}) \, > \, R\Big)  \le  \Pro \biggl( \frac{\psi^{-1}}{n^{\frac \beta \alpha}} \sum_{k=0}^\infty \Big(\sum_{\ell =1}^{ N_{nt} } c_{k+ \ell} \Big) |\theta_{-k}| \, > \, R\biggr)  \le  \frac{\tilde c \, \psi^{-\rho}}{R^\rho \, n^{ \frac{\rho \beta}{\alpha}}} \E[|\theta_0|^\rho] \to  0$$
	as $R \to \infty$ or $n\to \infty$, where $\tilde{c}:=\sum_{i=1}^\infty  ( \sum_{j=i}^\infty c_j )^\rho$, since the $\rho$-th moment of $\theta_0$ exists as its law is in the domain of attraction of an $\alpha$-stable distribution with $\rho < \alpha$.
\end{proof}

\begin{proof}[Proof of Proposition \ref{prop:direct_control_of _TV}]
	Applying \cite[Prop.~2.5]{andreasfabrice_theorypaper} to $(H^n,U^n)$, we have that $X$ is a semimartingale in the natural filtration generated by $(H,X)$. We rely on an approximation as described in Section \ref{sec:quantities_and_approach_needed_for_integral_convergence_proof}, but given that $H^n$ is pure jump we construct discretisations $H^{n\,\vert\,\varepsilon}$ that change only at the jumps of $H^n$. For each $\varepsilon>0$, component-wise, we can partition space as $\{x_i:i\in\mathbb{Z}\}$ with mesh $\varepsilon$ so that, almost surely, we avoid $H_t$, $H_{t-}$ at discontinuity times of $H$ and $X$, and avoid infima and suprema of $H$ over all intervals with rational endpoints. We set $H^{n\,\vert\,\varepsilon}_0=x_k$ for $H^n_0\in[x_k,x_{k+1})$ and take $H^{n\,\vert\,\varepsilon}=x_i$ when $H^n=x_i$. While $H^n\in(x_i,x_{i+1})$, we set $H^{n\,\vert\,\varepsilon}=x_i$ if it entered from below and $H^{n\,\vert\,\varepsilon}=x_{i+1}$ if it entered from above. Let $H^{\,\vert\,\varepsilon}$ denote the analogous discretisation of $H$. By construction, we have $H^{n\,\vert\,\varepsilon}\Rightarrow_{\mathrm{f.d.d.}}H^{\,\vert\,\varepsilon}$ outside $\operatorname{Disc}_{\mathbb{P}}(H^{\,\vert\,\varepsilon},X)$, jointly with $X^n\Rightarrow X$, and the pairs $(H^{n\,\vert\,\varepsilon},X^n)$ satisfy \eqref{eq:oscillcond}. Moreover, the number of jumps of $H^{n\,\vert\,\varepsilon}$ on any compact time interval is tight by the $M_1$ tightness of $H^n$. Thus,	$(X^n,\int_0^\bullet H^{n\,\vert\,\varepsilon}_{s-}\,\diff X^n_s)
	\Rightarrow
	(X,\int_0^\bullet H^{\,\vert\,\varepsilon}_{s-}\,\diff X_s)$ in the $M_1$ topology, by the simple-integral convergence argument in the proof of \cite[Prop.~2.12]{andreasfabrice_theorypaper}.	Since $|H-H^{\,\vert\,\varepsilon}|_T^*\leq\varepsilon$, it remains to establish the bound \eqref{eq:quantity2_integral_convergence_term_where_GD_needed}, with $H^{n\,\vert\,\varepsilon}$ in place of $H^{n\,\vert\,m,\varepsilon}$. Following the short outline in Section \ref{sec:quantities_and_approach_needed_for_integral_convergence_proof}, we know that the bound holds for the $U^n$ as integrators. Hence, it is enough to establish this bound for the $V^n$ as integrators. Note that for every $\varepsilon>0$, it holds almost surely that
	\begin{align*}
		\left|\int_0^\bullet
		(H^n_{s-}-H^{n\,\vert\,\varepsilon}_{s-})\,\diff V^n_s
		\right|_T^*
		&\le \sum_{k=0}^\infty
		|H^n_{s_k}-H^{n\,\vert\,\varepsilon}_{s_k}|\,
		|V^n_{s_{k+1}}-V^n_{s_k}|
		+2\varepsilon|V^n|_T^* \\
		&\le \varepsilon\sum_{k=0}^\infty
		|V^n_{s_{k+1}}-V^n_{s_k}|
		+2\varepsilon|V^n|_T^*.
	\end{align*}
	Thus, we have
	\[
	\Pro^n\biggl(
	\left|\int_0^\bullet
	(H^n_{s-}-H^{n\,\vert\,\varepsilon}_{s-})\,\diff V^n_s
	\right|_T^*\ge\lambda
	\biggr)
	\le
	\operatorname*{sup}_{n\ge1}\Pro^n\biggl(
	\sum_{k=0}^\infty
	|V^n_{s_{k+1}}-V^n_{s_k}|
	+2|V^n|_T^*
	\ge\frac{\lambda}{\varepsilon}
	\biggr)
	\]
	which tends to $0$ as $\varepsilon\to0$, by \eqref{eq:direct_control_TV} and the tightness of $(|V^n|_T^*)_{n\ge1}$.
\end{proof}

\begin{proof}[Proof of Theorem \ref{thm:Lipschitz_GDmodCA}] 
	Denote $U^n$ and $V^n$ the decomposition processes of $X^n$ according to Definition \ref{def:GD_mod_CA}. Since the $U^n$ have \eqref{eq:Mn_An_condition} and $U^n \Rightarrow X$ in the $M_1$ topology, \cite[Prop.~2.5]{andreasfabrice_theorypaper} gives that $X$ is a semimartingale in its own filtration. Moreover, if $\tilde{\gamma}<\gamma$, then $H_t=H_0$ for $t\ge0$, and integration by parts gives
	\[
	\left|\int_0^\bullet (H^n_s-H^n_0)\,\diff X^n_s\right|^*_T
	\le 2C_TT\,n^{\tilde\gamma-\gamma}|X^n|^*_T
	\xrightarrow{\Pro^n}0,
	\]
	so the claim follows from the joint weak convergence of $(H^n,X^n)$.
	
	From here on, let $\tilde\gamma\ge\gamma$. In order to derive the bound \eqref{eq:quantity2_integral_convergence_term_where_GD_needed}, note that it is enough to show that the bound holds for the $V^n$ as integrators, just as in the proof of Proposition  \ref{prop:direct_control_of _TV} above. Fix $T>0$ which is almost surely a continuity time of $(H,X)$, and let $\pi^n=\{s^n_k: k=0,1,...\}$ be the respective partitions of $[0,T]$ from Definition \ref{def:GD_mod_CA}(ii). For simplicity of notation, we will denote $s_k:=s^n_k$. First note that for every $n \in \N$ and any càdlàg process $G$ of finite variation, integration by parts yields 
	\[
	\int_0^t \, G_{s-} \; \diff V^n_s \; = \; V^n_t G_{t-} \; - \; V^n_0 G_0
	\; - \; \sum_{k=0}^\infty V^n_{s_k}\, [G_{(s_{k+1}\wedge t)-}-G_{(s_k\wedge t)-}].
	\]
	for all $t\ge 0$. Now, choosing $G:= H^n - H^{n \, \vert\, m, \varepsilon}$, where we recall that by definition $|H^n-H^{n \, \vert\, m, \varepsilon}|\le \varepsilon$, we obtain
	\begin{align*}
		A_n \, &:= \, \vert \, \int_0^\bullet \,
		(H_{s-}^n - H^{n \, \vert\, m, \varepsilon}_{s-})
		\diff V^n_s\, \vert^*_T \\
		&\le \; 2\varepsilon \, \vert V^n\vert^*_T
		\; + \;
		\sup{0\le t\le T} \sum_{k=0}^\infty \vert V^n_{s_k}\vert\,
		\vert
		(H^n_{s_k\wedge t}
		- H^{n \, \vert\, m, \varepsilon}_{(s_k\wedge t)-})
		-
		(H^n_{s_{k+1}\wedge t}
		- H^{n \, \vert\, m, \varepsilon}_{(s_{k+1}\wedge t)-})
		\vert.
	\end{align*}
	Throughout, the bounds are understood component-wise. Denote by $\rho^{m,\varepsilon}(H^n)$ the random partition used to define $H^{n\, \vert \,m, \varepsilon}$ in \eqref{eq:quantity1_integral_convergence_term_where_GD_needed} and by $\rho^m$ the deterministic subpartition of almost sure continuity points of $(H,X)$ such that its mesh size $|\rho^m|\to 0$ as $m\to \infty$. Since $H^{n \, \vert\, m, \varepsilon}$ by definition is constant between the partition points of $\rho^{m,\varepsilon}(H^n)$, we have
	\begin{align}
		A_n \; \; \le \; \; 2\varepsilon \; |V^n|^*_T \; &+ \!\! \sum_{\substack{k\ge0\\ \not\exists\,p\in\rho^{m,\varepsilon}(H^n):\,s_k\le p<s_{k+1}}} \hspace{-4ex}
		|V^n_{s_k}|\; C_T n^{\tilde\gamma-\gamma}(s_{k+1}-s_k) \notag\\
		&+\!\! \sum_{\substack{k\ge0\\ \exists\,p\in\rho^{m,\varepsilon}(H^n):\,s_k\le p<s_{k+1}}}
		\hspace{-4ex}|V^n_{s_k}|\; \bigl(2C_T n^{\tilde\gamma-\gamma}(s_{k+1}-s_k)+|H^n_{s_k}-H^{n\,\vert\,m,\varepsilon}_{s_k-}|\bigr) \notag\\[2ex]
		\le \; \; 2\varepsilon \; |V^n|^*_T \; &+ \; 2C_T\,n^{-\gamma}\;\sum_{t\in\pi^n}|V^n_t|
		\; + \!\! \sum_{\substack{k\ge0\\ \exists\,p\in\rho^{m,\varepsilon}(H^n):\,s_k\le p<s_{k+1}}}
		\hspace{-4ex}	|V^n_{s_k}|\; |H^n_{s_k}-H^{n\,\vert\,m,\varepsilon}_{s_k-}|. \label{eq:3.2}
	\end{align}
	where we have made use of the Lipschitz continuity of the $H^n$ and  $s_{k+1}-s_{k} \le |\pi^{n}|\le n^{- \tilde \gamma}$. Since $V^n=X^n-U^n$, $(|V^n|^*_T)_{n\ge 1}$ is tight in $\R$ as a result of $(|U^n|^*_T)_{n\ge 1}, (|X^n|^*_T)_{n\ge 1}$ being tight in $\R$ (due to the tightness of $U^n,X^n$ on the Skorokhod space), and by Definition \ref{def:GD_mod_CA}(ii), for the first two terms of (\ref{eq:3.2}) it holds
	\begin{align} \lim\limits_{\varepsilon \to 0} \; \; \limsup\limits_{n\to \infty} \; \Pro^n \Bigl( \varepsilon \; |V^n|^*_T \; + \; 2C_T\, n^{-\gamma}\;  \sum_{t \; \in \; \pi^n}  \; |V^n_{t}| \; > \; \gamma\Bigr)  \; = \; 0 \label{eq:3.3}
	\end{align}
	for all $\gamma>0$. Considering the last term of (\ref{eq:3.2}), we continue bounding it by
	\begin{align}
		&\sum_{\substack{k\ge 0 \, \\ \exists \, p\, \in \, \rho^{m,\varepsilon}(H^n) \; : \; s_k\le p<s_{k+1} }} \hspace{-6ex} |V^n_{s_{k}}| \; \; |H_{s_{k}}^n - H^{n \, \vert\, m, \varepsilon}_{s_k-}| \notag \\
		&\le \;  \sum_{\substack{k\ge 0 \\ \exists \, p\, \in \, \rho^{m,\varepsilon}(H^n)\setminus \rho^m  : \, s_k\le p<s_{k+1} }}  \hspace{-6ex} |V^n_{s_{k}}| \; \; |H_{s_{k}}^n - H^{n \, \vert\, m, \varepsilon}_{s_k-}| \; \; + \; \sum_{\substack{k\ge 0 \\ \exists \, p\, \in \, \rho^{m} \; : \; s_k\le p<s_{k+1} }}  \hspace{-5ex} |V^n_{s_{k}}| \; \; |H_{s_{k}}^n - H^{n \, \vert\, m, \varepsilon}_{s_k-}| \notag \\
		&\le \;  \sum_{\substack{k\ge 0 \\ \exists  p \in  \rho^{m,\varepsilon}(H^n)\setminus \rho^m  : \, s_k\le p<s_{k+1}}} \hspace{-6ex} (|V^n_{s_{k}}-V^n_{\min{}\{t  \in  \rho^m\,: \, t  \ge  s_{k}\}}| + |V^n_{\min{}\{t \, \in \, \rho^m : \, t \, \ge\,  s_{k}\}}|) \, |H_{s_{k}}^n - H^{n \, \vert\, m, \varepsilon}_{s_k-}| \notag \\
		& \qquad+ \sum_{\substack{k\ge 0 \\ \exists \, p\, \in \, \rho^{m} : \, s_k\le p<s_{k+1} }} \hspace{-5ex} \varepsilon |V^n_{s_{k}}|\, . \label{eq:3.4}
	\end{align}
	
	By Definition \ref{def:GD_mod_CA}(ii), for every $t>0$, we have
	\[
	\int_0^t |V^n_s|\,\diff s
	\le n^{-\tilde\gamma}\sum_{s\in\pi^n_t}|V^n_s|
	=n^{\gamma-\tilde\gamma}\Bigl(n^{-\gamma}\sum_{s\in\pi^n_t}|V^n_s|\Bigr)
	\xrightarrow{\Pro^n}0,
	\]
	where $\pi^n_t$ is the partition for $[0,t]$. Since $V^n=X^n-U^n$ and $(X^n,U^n)$ is $M_1$ tight on the product space, every subsequence has a further subsequence where the finite-dimensional distributions of $V^n$ converge to those of a c\`adl\`ag process on the almost sure continuity times of $X$. By the above, any such c\`adl\`ag limit must be zero, so $V^n\Rightarrow_{f.d.d.}0$ on a cocountable set $D\subseteq [0,T]$ with $0,T\in D$. Moreover, it follows from this that $(H^n,U^n)\Rightarrow_{f.d.d.}(H,X)$ on a cocountable set, and we have that $(|U^n|_T^*)_{n\geq 1}$ is tight due to the weak $M_1$ convergence. Since the $U^n$ also have \eqref{eq:Mn_An_condition}, \cite[Prop.~2.5]{andreasfabrice_theorypaper} thus gives that $X$ is a semimartingale in the filtration generated by $(H,X)$.
	
	Recall that the number of partition points of the deterministic partition $\rho^m=\{t^m_1<t^m_2<...\}$ is bounded by some constant $k_m$ only depending on $m$ and $T$. Then, we continue by estimating
	\begin{align}
		\sum_{\substack{k\ge0\\ \exists\,p\in\rho^m:\,s_k\le p<s_{k+1}}}|V^n_{s_k}|
		\; \le \; \sum_{p=1}^{k_m}\;\sup{(t_p^m-n^{-\tilde\gamma})\vee0\le s\le t_p^m}|V^n_s|
		\; \conv{n\to\infty}{\Pro^n}\;0. \label{eq:3.5}
	\end{align}
	for each fixed $m\ge1$, since $V^n_{t_p^m}\to0$ in probability and the weak $M_1$ convergence of $X^n,U^n$ gives vanishing oscillations around these almost sure continuity points. Further, the number of partition points of $\rho^{m,\varepsilon}(H^n) \, \backslash \, \rho^{m}$ is bounded by $N^T_\varepsilon(H^n)$, where $N^T_\varepsilon$ is the maximal number of $\varepsilon$-increments of $H^n$ on $[0,T]$ defined by
	\begin{align} N^T_\varepsilon(x) \; := \;   \sup{} \left\{ \, n \;  :  \; 0\le t_1\le t_2 \le ...\le t_{2n}\le T \; , \;  |x(t_{2i})- x(t_{2i-1})|\ge \varepsilon \right\} \label{eq:maxnumosc}. \end{align}
	On account of the bound \eqref{eq:3.5}, \eqref{eq:3.4} can be further estimated by 
	\begin{align}
		&\sum_{\substack{k\ge0\\ \exists\,p\in\rho^{m,\varepsilon}(H^n):\,s_k\le p<s_{k+1}}}\; |V^n_{s_k}|\; |H^n_{s_k}-H^{n\,\vert\,m,\varepsilon}_{s_k-}| \notag\\
		&\le\; N^T_\varepsilon(H^n)\;\sup{\substack{0\le s\le u\le r\le T\\ |s-r|\le|\rho^m|}}\; |H^n_s-H^n_u|\;|V^n_u-V^n_r|
		\; +\;2\varepsilon\;N^T_\varepsilon(H^n)\;\sum_{p=1}^{k_m}|V^n_{t_p^m}| \notag\\
		&\hspace{2cm}+\;\varepsilon\;\sum_{p=1}^{k_m}\;\sup{(t_p^m-n^{-\tilde\gamma})\vee0\le s\le t_p^m}|V^n_s|\notag\\[2ex]
		&\le\;2N^T_\varepsilon(H^n)\;(|H^n|^*_T+|V^n|^*_T)\;\hat w^T_{|\rho^m|}(H^n,V^n) \notag\\
		&\hspace{2cm}+\;2\varepsilon\;N^T_\varepsilon(H^n)\;\sum_{p=1}^{k_m}|V^n_{t_p^m}|
		\; +\;\varepsilon\;\sum_{p=1}^{k_m}\;\sup{(t_p^m-n^{-\tilde\gamma})\vee0\le s\le t_p^m}|V^n_s|. \label{eq:3.6}
	\end{align}
	where the $t^m_p$ denote the partition points of the partition $\rho^m$ (which without loss of generality can be assumed to be in the dense set $D$). Furthermore, $\hat{w}$ is the consecutive increment function from Definition \ref{def:avco}. The last inequality is based on the fact that $|x y| \, \le \, (|x|\vee |y|)(|x|\wedge |y|)$. Combining \eqref{eq:3.2} with \eqref{eq:3.6} we obtain  
	\begin{align}
		\left|\,\int_0^\bullet\,(H^n_{r-}-H^{n\,\vert\,m,\varepsilon}_{r-})\;\diff V^n_r\,\right|^*_T
		\;&\le\;2\varepsilon\;|V^n|^*_T\; +\;2C_T\,n^{-\gamma}\sum_{t\in\pi^n}|V^n_t| \notag\\
		& +\;2N^T_\varepsilon(H^n)\,(|H^n|^*_T+|V^n|^*_T)\,\hat w^T_{|\rho^m|}(H^n,V^n) \notag\\
		&+\;2\varepsilon N^T_\varepsilon(H^n)\sum_{p=1}^{k_m}|V^n_{t_p^m}| +\,\varepsilon\,\sum_{p=1}^{k_m}\,\sup{(t_p^m-n^{-\tilde\gamma})\vee0\le s\le t_p^m}|V^n_s|. \label{eq:3.7}
	\end{align}
	We know that $|H^n|^*_T$ is tight since the $H^n$ are tight in $M_1$. The same holds for $|V^n|^*_T$ as $V^n=X^n-U^n$ and both $X^n$ and $U^n$ are tight in $M_1$. For fixed $\varepsilon>0$, it is known that $N_\varepsilon^T(H^n)$ is tight (see e.g. \cite[Cor.~A.9]{andreasfabrice_theorypaper}). In addition, for $\tilde{\gamma} = \gamma$ we get $\hat{w}^T_{\delta}(H^n,\, V^n) \leq C_T \delta$, and for $\tilde{\gamma}>\gamma$ the assumptions give \eqref{eq:oscillcond} for the sequence $(H^n,V^n)$, which is
	\begin{align} 
		\lim\limits_{m\to \infty} \; \limsup\limits_{n\to \infty} \; \; \Pro^n \left( \hat{w}^T_{|\rho^m|}(H^n,\, V^n) \; > \; \gamma \right) \; = \; 0 \label{eq:3.8}
	\end{align}
	for all $\gamma>0$ and $T\ge 1$. Furthermore, for fixed $m\ge 1$, due to the convergence of the finite-dimensional distributions of $V^n$ to 0 along $D$, it holds that
	\begin{align}
		\sum_{p=1}^{k_m}  \; |V^n_{t^m_p}| \; \;  \conv{n\to \infty}{\Pro^n} \; \; 0 . \label{eq:3.9}
	\end{align}
	With regards to \eqref{eq:3.7}, we now employ \eqref{eq:3.3}, \eqref{eq:3.5}, \eqref{eq:3.8}, and \eqref{eq:3.9} to deduce
	$$\lim\limits_{\varepsilon \to 0} \; \; \limsup\limits_{m\to \infty}\;\;\limsup\limits_{n\to \infty} \; \; \Pro^n \biggl( \, \biggl| \, \int_0^\bullet \, (H_{r-}^n - H^{n \, \vert \, m,\varepsilon}_{r-}) \; \diff V^n\, \biggr|^*_T \; > \; \gamma \biggr) \; \; = \; \; 0$$
	for all $\gamma>0$, which immediately yields \eqref{eq:quantity3_integral_convergence_term_where_GD_needed}.
\end{proof}

\subsection{Proofs pertaining to Section \ref{subsec:Control_through_independence}}
\begin{proof}[Proof of Theorem \ref{thm:Integral_converg_under_GDmodCI}]
	Applying \cite[Prop.~2.5]{andreasfabrice_theorypaper} to $(H^n,U^n)$, we have that $X$ is a semimartingale in the natural filtration generated by $(H,X)$. As in the proofs of Proposition \ref{prop:direct_control_of _TV} and Theorem \ref{thm:Lipschitz_GDmodCA}, it is enough to show the bound \eqref{eq:quantity2_integral_convergence_term_where_GD_needed}
	for every $T>0$ with integrators $V^n= \sum_{i\ge 1} V^{n,i}$. Indeed, by (ii.i) and the moment bounds in (ii.iii), \((|V^n|_T^*)_{n\ge1}\) is tight, and hence so is \((|\tilde U^n|_T^*)_{n\ge1}\), since \(\tilde U^n=X^n-U^n-V^n\) and both \(X^n\) and \(U^n\) are \(M_1\)-tight. To establish the bound, we exploit the main ideas of the proof of \cite[Lem.~2(b)]{avram} and adapt it to our setting. Fix $T>0$ and denote $\tilde{H}^{n \, \vert\, m,\varepsilon}:=H^n - H^{n \, \vert\, m, \varepsilon}$. For all $\varepsilon>0$ and $n,m\ge 1$,
	\begin{align*} \int_0^t \,\tilde{H}^{n \, \vert\, m,\varepsilon}_{s-} \; \diff V^n_s \; &= \; \sum_{i=1}^\infty \sum_{k=1}^{\Lambda^{n,i}(t)} \; \tilde{H}^{n \, \vert\, m,\varepsilon}_{\sigma^{n,i}_k-}\; (
		V^{n,i}_{\sigma_k^{n,i}}-V^{n,i}_{\sigma_{k-1}^{n,i}}) \\
		& = \; \sum_{i=1}^\infty \sum_{k=1}^{\Lambda^{n,i}(t)} \; \tilde{H}^{n \, \vert\, m,\varepsilon}_{\sigma^{n,i}_k-}\; V^{n,i}_{\sigma^{n,i}_k}  \; - \; \sum_{i=1}^\infty \sum_{k=1}^{\Lambda^{n,i}(t)} \; \tilde{H}^{n \, \vert\, m,\varepsilon}_{\sigma^{n,i}_k-}\; V^{n,i}_{\sigma^{n,i}_{k-1}}.
	\end{align*}
	Let $\delta>0$. By (ii.i) of Definition \ref{defi:GD_mod_CI} there exists $K_\delta>0$ such that
	\begin{align} \E^n \biggl[  \left| \, \int_0^\bullet \, \tilde{H}^{n \, \vert\, m,\varepsilon}_{s-} \diff V^{n}_s  \right|^*_T \wedge  1 \biggr] \; &\le \; \delta \; + \; \sum_{i=1}^\infty \, \E^n \biggl[\sup{j=1,...,f(n)K_\delta}  \, \biggl| \, \sum_{k=1}^{j} \, \tilde{H}^{n \, \vert\, m,\varepsilon}_{\sigma^{n,i}_k-}\; V^{n,i}_{\sigma^{n,i}_k} \, \biggr| \,\wedge \, 1 \, \biggr]\notag  \\
		& \qquad + \;  \sum_{i=1}^\infty \, \E^n \biggl[\sup{j=1,...,f(n)K_\delta}  \, \biggl| \, \sum_{k=1}^{j} \, \tilde{H}^{n \, \vert\, m,\varepsilon}_{\sigma^{n,i}_k-}\;  V^{n,i}_{\sigma^{n,i}_{k-1}} \, \biggr|\,\wedge\, 1 \,   \biggr] .\label{eq:4.27}
	\end{align}
	where we have made use of dominated convergence to swap the infinite sum. We begin by bounding each of the summands in the first infinite sum of \eqref{eq:4.27} and the bound for the terms in the second sum can then be obtained in full analogy. To this end, we define
	\begin{align*}
		Z^{n,i,\, \le}_k \; &:= \; \tilde{H}^{n \, \vert\, m,\varepsilon}_{\sigma^{n,i}_k-} \biggl( V^{n,i,\, \le }_{\sigma^{n,i}_k} \, - \, \E^n\Big[V^{n,i,\, \le}_{\sigma^{n,i}_k} \; | \;\mathcal{V}^i_{n,k-1}\Big] \biggr) \\
		Z^{n,i,\, >}_k \; &:= \;  \tilde{H}^{n \, \vert\, m,\varepsilon}_{\sigma^{n,i}_k-} \biggl( V^{n,i,\, > }_{\sigma^{n,i}_k} \, + \, \E^n\Big[V^{n,i,\, \le }_{\sigma^{n,i}_k} \; | \;\mathcal{V}^i_{n,k-1}\Big] \biggr)
	\end{align*}
	where we have denoted 
	\begin{align} V^{n,i, \, \le }_{\sigma^{n,i}_k}\, := \, V^{n,i }_{\sigma^{n,i}_k} \ind_{\{|V^{n,i }_{\sigma^{n,i}_k}| \, \le \, 1\}} \quad \text{ and } \quad V^{n,i,\, > }_{\sigma^{n,i}_k}:= V^{n,i }_{\sigma^{n,i}_k} \ind_{\{|V^{n,i }_{\sigma^{n,i}_k}|\, > \, 1\}} \label{eq:defi_V_<=_>}. \end{align}
	Let us remark that $Z^{n,i,\, \le}_k+Z^{n,i,\, >}_k= \tilde{H}^{n \, \vert\, m,\varepsilon}_{\sigma^{n,i}_k-} \; V^{n,i}_{\sigma^{n,i}_k}$. The discrete-time processes $j\mapsto \sum_{k=1}^{j}\, Z^{n,i,\, \le}_k$ and $j\mapsto \sum_{k=1}^{j}\, Z^{n,i,\, >}_k$ are both discrete martingales with respect to the filtration $\{ \sigma(\mathcal{V}^i_{n,j} ,  \mathcal{H}^i_{n,j}) \}_{j\ge 1}$. To see this, it suffices to note that for $k> j$ we have
	\begin{align*}
		\E^n \bigl[ Z_k^{n,i,\, \le} \; | \; \mathcal{V}^i_{n,j} ,  \mathcal{H}^i_{n,j} \,   \bigr] =  \E^n \Bigl[  \E^n \bigl[ \,Z_k^{n,i,\, \le} \; | \; \mathcal{V}^i_{n,k-1} ,  \mathcal{H}^i_{n,k}\, \bigr] \; | \; \mathcal{V}^i_{n,j} ,  \mathcal{H}^i_{n,j} \, \Bigr].
	\end{align*}
	Based on the measurability of $\tilde{H}^{n \, \vert\, m,\varepsilon}_{\sigma^{n,i}_k-}$ with respect to $\mathcal{H}^i_{n,k}$ as well as assumption \eqref{eq:indep_cond_int_conv_GDmodCI}, which implies the independence of $V^{n,i}_{\sigma^{n,i}_k}$ and $\mathcal{H}^i_{n,k}\vee \mathcal{V}^i_{n,k-1}$, we obtain
	\begin{align*}
		&\E^n \bigl[ Z_k^{n,i,\, \le} \, |  \mathcal{V}^i_{n,k-1} ,  \mathcal{H}^i_{n,k}   \bigr] \, = \, \E^n \biggl[ \tilde{H}^{n \, \vert\, m,\varepsilon}_{\sigma^{n,i}_k-} \Big( V^{n,i, \le }_{\sigma^{n,i}_k}  -  \E^n\Big[V^{n,i,\le}_{\sigma^{n,i}_k} \,| \,\mathcal{V}^i_{n,k-1}\Big]  \Big) \, | \, \mathcal{V}^i_{n,k-1} ,  \mathcal{H}^i_{n,k}      \biggr]  \\
		& \qquad  = \,  \tilde{H}^{n \, \vert\, m,\varepsilon}_{\sigma^{n,i}_k-} \; \Big( \E^n \Big[ \, V^{n,i, \le }_{\sigma^{n,i}_k} \, | \, \mathcal{V}^i_{n,k-1}\, , \, \mathcal{H}^i_{n,k}\, \Big] \; - \;\E^n \big[ \,V^{n,i, \le }_{\sigma^{n,i}_k} \, | \, \mathcal{V}^i_{n,k-1} \,  \Big] \Big)\; = \; 0
	\end{align*}
	We proceed similarly for the process $j\mapsto \sum_{k=1}^{j}\, Z^{n,i,\, >}_k$ by using $V^{n,i, \, >}_{\sigma^{n,i}_k}=V^{n,i}_{\sigma^{n,i}_k}- V^{n,i, \, \le}_{\sigma^{n,i}_k}$ and (ii.ii) from Definition \ref{defi:GD_mod_CI}. Since both  $j\mapsto \sum_{k=1}^{j}\, Z^{n,i,\, \le}_k$ and $j\mapsto \sum_{k=1}^{j}\, Z^{n,i,\, >}_k$ are discrete martingales, we can apply Doob's maximal inequality (after Jensen's inequality) to obtain 
	\begin{align}
		&\E^n \biggl[\sup{j=1,...,f(n)K_\delta} \, \biggl|
		\, \sum_{k=1}^{j} \,
		\tilde{H}^{n \, \vert\, m,\varepsilon}_{\sigma^{n,i}_k-}
		\; V^{n,i}_{\sigma^{n,i}_k} \, \biggr|
		\wedge 1  \biggr] \notag \\
		&\le \;\E^n \biggl[\sup{j=1,...,f(n)K_\delta} \, \biggl|
		\, \sum_{k=1}^{j} \, Z^{n,i,\, \le}_k \, \biggr|^\mu
		\! \wedge 1 \biggr]^{\frac{1}{\mu}}
		\!+ \,\E^n \biggl[\sup{j=1,...,f(n)K_\delta} \, \biggl|
		\, \sum_{k=1}^{j} \, Z^{n,i,\, >}_k \, \biggr|^\lambda
		\! \wedge 1 \biggr]^{\frac{1}{\lambda}}
		\notag \\
		&\le \; \frac{\mu}{\mu-1} \; \E^n \Biggl[ \, \Biggl|
		\, \sum_{k=1}^{f(n)K_\delta} \; Z^{n,i,\, \le}_k
		\, \Biggr|^\mu \, \Biggr]^{\frac{1}{\mu}}\!
		+\, \frac{\lambda}{\lambda-1} \E^n \Biggl[ \, \Biggl|
		\, \sum_{k=1}^{f(n)K_\delta} \; Z^{n,i,\, >}_k
		\, \Biggr|^\lambda \, \Biggr]^{\frac{1}{\lambda}}. \label{eq:eq_for_remark_on_alt_set_of_cond_GDmodCI}
	\end{align}
	By Bahr-Esseen's martingale-differences inequality \cite[Theorem 2]{bahresseen}, this is further estimated by 
	\begin{align} \le \; C_\mu \; \biggl(\sum_{k=1}^{f(n)K_\delta} \E^n \Bigl[ \, \Bigl| \,Z^{n,i,\, \le}_k \, \Bigr|^\mu \, \Bigr]\biggr)^\frac{1}{\mu} \; + \; C_\lambda \; \biggl(\sum_{k=1}^{f(n) K_\delta} \; \E^n \Bigl[ \, \Bigl| \,  Z^{n,i,\, >}_k \, \Bigr|^\lambda\, \Bigr]\biggr)^\frac{1}{\lambda}. \label{eq:4.29}
	\end{align}
	with $C_\mu,C_\lambda>0$ only depending on $\mu,\lambda$ respectively. Finally, recall that $|\tilde{H}^{n \, \vert\, m,\varepsilon}|\le \varepsilon$, and, by definition of the $Z^{n,i,\, \le}_k$, a simple application of Jensen's inequality yields
	\[
	\E^n \Bigl[ \, \Bigl| \,Z^{n,i,\, \le}_k \, \Bigr|^\mu \, \Bigr] \; \le \; (2\varepsilon)^{\mu}  \;  \E^n\Bigl[\, |V^{n,i}_{\sigma^{n,i}_k}|^\mu\, \ind_{\{|V^{n,i}_{_{\sigma^{n,i}_k}}|\, \le \, 1\}}\,\Bigr] \quad \text{for all}\;\;n,k,i\ge 1.
	\]
	Analogously we obtain the estimate $\E^n [  | Z^{n,\, i,\, >}_k  |^\lambda ] \, \le \, (2 \varepsilon)^\lambda \, \E^n[|V^{n,i}_{\sigma^{n,i}_k}|^\lambda\, \ind_{\{|V^{n,i}_{_{\sigma^{n,i}_k}}|\, > \, 1\}}]$. Hence, based on (\ref{eq:4.29}), this yields
	\begin{align}
		&\limsup_{n\to \infty}\; \sum_{i=1}^\infty \, \E^n \biggl[\sup{j=1,...,f(n)K_\delta}  \, \biggl| \, \sum_{k=1}^{j} \, \tilde{H}^{n \, \vert\, m,\varepsilon}_{\sigma^{n,i}_k-}\; V^{n,i}_{\sigma^{n,i}_k} \, \biggr| \,\wedge \, 1 \, \biggr]  \notag \\ 
		&\le \; \; 2\varepsilon \; C_{\mu} \; \limsup_{n\to \infty} \; \sum_{i=1}^\infty \, \Bigl(\sum_{k=1}^{f(n)K_\delta} \E^n\Bigl[\, |V^{n,i}_{\sigma^{n,i}_k}|^\mu\, \ind_{\{|V^{n,i}_{_{\sigma^{n,i}_k}}|\, \le \, 1\}}\,\Bigr] \Bigr)^\frac{1}{\mu}  \notag \\  
		&\qquad \; + \; 2\varepsilon \; C_{\lambda} \; \limsup_{n\to \infty} \; \sum_{i=1}^\infty \, \Bigl(  \sum_{k=1}^{f(n)K_\delta} \E^n\Bigl[\, |V^{n,i}_{\sigma^{n,i}_k}|^\lambda\, \ind_{\{|V^{n,i}_{_{\sigma^{n,i}_k}}|\, > \, 1\}}\,\Bigr] \Bigr)^\frac{1}{\lambda} . \label{eq:4.30}
	\end{align}
	Then, the limit superior parts of \eqref{eq:4.30} are finite by (ii.iii) of Definition \ref{defi:GD_mod_CI} and \eqref{eq:4.30} tends to $0$ as $\varepsilon \to 0$. Analogously to the above procedure, one can achieve the same bound for the second infinite sum of \eqref{eq:4.27}, where we have martingality for $\{\sigma( \mathcal{V}^i_{n,j-1} , \mathcal{H}^{i}_{n,j})\}_{j\ge1}$ instead of $\{ \sigma(\mathcal{V}^i_{n,j} ,  \mathcal{H}^i_{n,j}) \}_{j\ge 1}$. Hence, we deduce \eqref{eq:quantity2_integral_convergence_term_where_GD_needed}.
\end{proof}

\subsection{Proofs pertaining to Section \ref{subsec:Final_results_for_CTRWs}}

\begin{proof}[Proof of Proposition \ref{prop:correlate_CTRWs_are_GDmodCA}] By Lemma \ref{lem:decomp_U^n_part_has_GD} with the filtration $\bar{\mathcal{F}}^n_t$, and the convergence results for uncorrelated or correlated uncoupled CTRWs in \eqref{eq:CTRW_$M_1$_conv}, it remains to show (ii) of Definition \ref{def:GD_mod_CA}. To this end, fix $T>0$, $\gamma>\beta-\beta/\alpha$ and suppose \eqref{eq:technical_condition} to hold. Let $\rho=1$ if $1<\alpha\le 2$. If $\alpha=1$, take $0<\rho<1$ sufficiently close to $1$ such that \eqref{eq:technical_condition} holds for this $\rho$ with $\rho(\gamma+\beta/\alpha)>\beta$. First note that, for every $\sigma(J_k:k\geq 1)$-measurable random variable $s: \Omega \to [0,T]$, by $(J_k)_{k\geq 1}\indep (\theta_k)_{k\in \mathbb{Z}}$, monotone convergence, and identical distribution of the $\theta_k$, we obtain 
	\begin{align}
		\E \left[ \,  |V^n_s|^\rho \, \right] \; & \le \; \frac{\tilde c\, \psi^{-\rho}}{n^{\frac{\rho \beta} \alpha}}\, \E[ \, |\theta_{0}|^\rho \, ]  \label{eq:5.14}
	\end{align}
	where $\tilde c:= \sum_{i=1}^\infty (\sum_{j=i}^\infty c_j)^\rho<\infty$.
	Recall in particular that the $\rho$-th moment of $\theta_0$ exists as its law is in the domain of attraction of an $\alpha$-stable random variable where $1\le \alpha\le 2$. For every $n\ge 1$, set $n_\beta:=\lceil Tn^\beta\rceil$ and define (random) partitions $\pi^n$ by
	$$ \pi^n \; := \; \underbrace{\Bigl\{ \, \frac{k}{n_\beta}T \; :\; k=0,1,...,n_\beta\, \Bigr\}}_{=: \pi^{n,1}} \; \, \cup \, \; \underbrace{\bigl\{ \, 0< s \le T \; : \; \Delta N_{ns} \neq 0 \, \bigr\}}_{=: \pi^{n,2}}$$
	where the second set $\pi^{n,2}$ contains the jump times of $V^n$ on $[0,T]$. Moreover, the mesh size of the partitions satisfy $|\pi^n| \le n^{-\beta}$. In addition, since $t\mapsto N_{nt}$ is a counting process,we have
	$$ \operatorname{card}(\pi^{n,2}) \; = \; \operatorname{card} \left( \left\{ \, 0< s \le T \; : \; \Delta N_{ns} \neq 0 \, \right\}\right) \; = \; N_{nT}.$$
	Now, let $\varepsilon>0$. Then, by tightness of $n^{-\beta} N_{n\bullet}$, there exists $K_\varepsilon>0$ such that $\Pro( \, N_{nT}  > K_\varepsilon n^\beta) \le \varepsilon$.
	Hence, we have
	\begin{align*}
		\Pro &\Big( \, n^{- \gamma}\,\sum_{s \, \in \, \pi^n} |V^n_s| \; > \; \lambda\, \Big) \; \le \; \Pro \Big( \, \sum_{s \, \in \, \pi^{n,1}} |V^n_s| \; > \; \frac{\lambda n^{ \gamma}}{2}\, \Big) \; \\
		& +\; \Pro \Big( \, \sum_{s \, \in \, \pi^{n,2}} |V^n_s| \; > \; \frac{\lambda n^{ \gamma}}{2} \; , \; N_{nT} \, \le \, K_\varepsilon n^\beta  \, \Big)  + \; \Pro \left( \, N_{nT} \, > \, K_\varepsilon n^\beta \, \right) \\
		&\le \; \Pro \biggl( \, \sum_{k=0}^{n_\beta} |V^n_{\frac{k}{n_\beta}T}| \; > \; \frac{\lambda n^{ \gamma}}{ 2}\, \biggr)\; + \; \Pro \biggl( \, \sum_{k=1}^{K_\varepsilon n^{\beta}} |V^n_{(L_k/n) \wedge T}| \; > \; \frac{\lambda n^{\gamma}}{2} \, \biggr) \; + \; \varepsilon
	\end{align*}
	where the $L_k$ are defined as in \eqref{defi:CTRW}. We recall that $\rho\le 1$. Then, by Markov's inequality and \eqref{eq:5.14}, this implies
	\begin{align*}
		\Pro \Bigl( \, n^{-\gamma}\,\sum_{s \, \in \, \pi^n} |V^n_s| \; > \; \lambda\, \Bigr) \; &\le \; \Big( \frac{2}{\lambda n^{\gamma}}\Big)^\rho \biggl( \,\sum_{k=1}^{n_\beta} \, \E[\, |V^n_{\frac{k}{n_\beta}T}|^\rho\,] \; + \;  \sum_{k=1}^{K_\varepsilon n^{\beta}} \E[\, |V^n_{(L_k/n)\wedge T}|^\rho\, ]\, \biggr) \; + \; \varepsilon \\
		&\le \; \frac{2^\rho \tilde c \, \psi^{-\rho} (K_\varepsilon+T+1)n^\beta}{\lambda^\rho n^{\rho (\gamma+\frac\beta \alpha)}} \; \E[\,|\theta_0|^\rho\,] \; + \; \varepsilon 
	\end{align*}
	for all $n\ge 1$. Since $\gamma>\beta - \beta/\alpha$ by assumption and $\rho$ is such that $\rho(\gamma + \beta/\alpha)> \beta$, this tends to $\varepsilon$ as $n\to \infty$. As $\varepsilon>0$ was arbitrary, this yields (ii) of Definition \ref{def:GD_mod_CA} for $\operatorname{GD\,mod\,CA}(\gamma, \beta)$.
\end{proof}

\begin{proof}[Proof of Theorem \ref{thm:result_int_conv_CTRWs_lipschitz_integrands}]
	The result follows directly from Proposition \ref{prop:correlate_CTRWs_are_GDmodCA} (to deduce $\operatorname{GD\, mod \, CA}(\beta - \gamma, \beta)$) and Theorem \ref{thm:Lipschitz_GDmodCA} (applied to $\psi^{-1}X^n=U^n+V^n$).
\end{proof}

\begin{proof}[Proof of Proposition \ref{prop:correlated_CTRW_are_GDmodCI}]
	We will only prove the claim for correlated CTRWs, the proof for moving averages can be conducted in full analogy. Set $\tilde{U}^n\equiv 0$ and note that $\psi^{-1}X^n=U^n+\tilde{U}^n+ \sum_{i=1}^\infty V^{n,i}$, where $U^n$ is defined as in \eqref{eq:defi_U^n_in_CTRW_decomp}. Due to Lemma \ref{lem:decomp_U^n_part_has_GD} for the filtration $\bar{\mathcal{F}}^n_t$ and the convergence results for uncorrelated and correlated uncoupled CTRWs in \eqref{eq:CTRW_$M_1$_conv}, only (ii) of Definition \ref{defi:GD_mod_CI} is still to be shown. 
	With stopping times $\sigma^{n,i}_k= L_k/n$, $k\ge 1$, we obtain by definition of $N_{nt}$,
	\begin{align} 
		V^{n,i}_{\sigma^{n,i}_k} \; = \;  - \, \frac{1}{\psi n^{\frac \beta \alpha}} \Big(\, \sum_{\ell=i}^{\infty} c_{\ell}\, \Big) \, \theta_{k-i+1}\mathbf{1}_{\{k\geq i\}}, \label{eq:definition_V_correlated_CTRWs_are_GDmodCI}
	\end{align}
	where we recall the definition $\psi=\sum_{\ell=0}^\infty c_\ell$. Hence (ii.ii) of Definition \ref{defi:GD_mod_CI} follows from the independence of the $\theta_k$ and the fact that they are centered as well. Since $\Lambda^{n,i}(t)= N_{nt}$, by definition of the $\sigma^{n,i}_k$ and given that $|N_{n\bullet}/n^\beta|^*_T$ is tight in $\R$ for every $T>0$, we obtain (ii.i) of Definition \ref{defi:GD_mod_CI} with $f(n)=n^\beta$. Thus, it only remains to verify (ii.iii). To this end, denote $\tilde c_i:= {\psi^{-1}}\sum_{\ell=i}^\infty c_\ell\le 1$ and $\tilde c:= \sum_{i=1}^\infty \tilde c_i$. If $\alpha=2$, then $\E[|\theta_0|^2]<\infty$ by \eqref{eq:norm_dom_of_attraction}, and taking $\lambda=\mu=2$ both quantities in (ii.iii) are bounded by the finite value
	$ (K\E[|\theta_0|^2])^{1/2}\sum_{i=1}^\infty\tilde c_i$
	due to \eqref{eq:technical_condition}. If, instead, $1<\alpha<2$, we can take $0<\rho<1$ satisfying \eqref{eq:technical_condition} and let $\gamma>0$ be sufficiently small such that $\alpha-\gamma>1$, $\alpha+\gamma\le2$, and $\alpha/(\alpha+\gamma)\ge\rho$. Then, using Jensen's inequality, monotone convergence and the identical distribution of the $\theta_k$, we obtain
	\begin{align}
		\E \left[ \, |V^{n,i,\, >}_{\sigma^{n,i}_k}|^{\, \alpha -\gamma} \, \right] \; \le \; \tilde c_i^{\alpha-\gamma} \, n^{-\frac{\beta(\alpha-\gamma)}{\alpha}} \, \E \left[ \, |\theta_0 \, |^{\, \alpha -\gamma} \; \ind_{\{|\theta_0|\, > \, n^{\frac \beta \alpha} \}} \, \right]  \label{eq:4.32}
	\end{align}
	for all $i,k,n\ge 1$, where $V^{n,i,\, >}_{\sigma^{n,i}_k}$ is defined in \eqref{eq:defi_V_<=_>}. Further,
	\begin{align}
		\E \Bigl[ \, |\theta_{0} \, |^{\, \alpha -\gamma}\; &\ind_{\{|\theta_{0}|\, > \, n^{\frac \beta \alpha} \}} \, \Bigr] \; = \; \int_{(n^{\frac \beta \alpha}, \, \infty)} \, x^{\alpha-\gamma} \; \diff \Phi(x) \; + \; \int_{(-\infty,\, - n^{\frac \beta \alpha})} \, (-x)^{\alpha-\gamma} \; \diff \Phi(x) \notag\\
		&= \; -\int_{(n^{\frac \beta \alpha}, \, \infty)} \, x^{\alpha-\gamma} \; \diff (1-\Phi)(x) \; - \; \int_{(-\infty,\, - n^{\frac \beta \alpha})} \, (-x)^{\alpha-\gamma} \; \diff (1-\Phi)(x) \label{eq:4.31}
	\end{align}
	where we have denoted the cdf of $\theta_0$ by $\Phi$. Integration by parts now yields
	\begin{align*}
		\int_{n^{\frac \beta \alpha}}^\infty \, x^{\alpha-\gamma} \; &\diff (1-\Phi)(x)  =  \left[ \, x^{\alpha-\gamma} (1-\Phi)(x)  \right]_{ n^{\frac \beta \alpha}}^{\infty} \; - \; (\alpha-\gamma)  \int_{ n^{\frac \beta \alpha}}^{\infty} x^{\alpha-\gamma-1} \, (1-\Phi)(x) \diff x \\
		&\ge \; -C n^{-\frac{\beta\gamma}{\alpha}} \; - \; C(\alpha-\gamma) \, \int_{n^{\frac\beta\alpha}}^\infty x^{-\gamma-1}\diff x \\
		&= \; -C n^{-\frac{\beta\gamma}{\alpha}} \; - \; \frac{C(\alpha-\gamma)}{\gamma} n^{-\frac{\beta\gamma}{\alpha}} \; = \; -\frac{C\alpha}{\gamma} n^{-\frac{\beta\gamma}{\alpha}}
	\end{align*}
	for $n$ large enough, where we have used that $(1-\Phi)(x)\le \Pro(|\theta_0|>x)\le Cx^{-\alpha}$ for some $C>0$, for all sufficiently large $x$. Applying the same argument to $-\theta_0$, we achieve the same lower bound for the second integral of \eqref{eq:4.31}. Hence, this implies
	$$ \E \left[ \, |\theta_{0} \, |^{\, \alpha -\gamma}\; \ind_{\{|\theta_{0}|\, > \, n^{\frac \beta \alpha} \}} \, \right] \; \le \; \frac{2C\alpha}{\gamma} \, n^{-\frac{\beta \gamma}{\alpha}}$$
	for $n$ large enough, and combining this with (\ref{eq:4.32}), we deduce
	$$\sum_{i=1}^\infty  \Bigl( \sum_{k=1}^{K \, n^\beta}  \E \Big[ \, |V^{n,i,\, >}_{\sigma^{n,i}_k}|^{\, \alpha-\gamma} \, \Big]\Bigr)^{\frac{1}{\alpha-\gamma}} \le  \tilde c \Bigl( \frac{2C\alpha }{\gamma}\,  (K n^\beta+1) \,  n^{-\beta}\Bigr)^{\frac{1}{\alpha-\gamma}} \le  \tilde c \Bigl( \frac{2C(K+1)\alpha }{\gamma}\,  \Bigr)^{\frac{1}{\alpha-\gamma}}$$
	for $n$ large enough, where the final term is finite, thus yielding the first bound in (ii.iii) in Definition \ref{defi:GD_mod_CI} with $\lambda =\alpha-\gamma$, since $\tilde c< \infty$ by \eqref{eq:technical_condition}. For the second quantity in (ii.iii), we obtain the required bound  with $\mu=\alpha+\gamma$ since 
	$$ \E \Big[|V^{n,i}_{\sigma^{n,i}_k}|^\mu \ind_{\{|V^{n,i}_{\sigma^{n,i}_k}|\leq 1\}}\Big] \; \le \; \tilde c_i^{\mu} \, n^{-\frac{\beta \mu}{\alpha}} \, \E \Big[ \, |\theta_0 \, |^{\, \mu} \; \ind_{\{|\theta_0|\, \le \, n^{ \beta /\alpha} \tilde c_i^{-1} \}} \, \Big] \; \le \; C_\mu \tilde c_i^\alpha n^{-\beta}  $$ 
	for some constant $C_\mu >0$ and $n$ large enough, due to the fact that $\E[|\theta_0|^\mu\ind_{\{|\theta_0|\le R\}}]\le C_\mu R^{\mu-\alpha}$ for $R$ sufficiently large. Indeed, using the tail probability $\Pro(|\theta_0|>x)\le Cx^{-\alpha}$ and $\alpha/\mu\ge\rho$, we have $\sum_{i=1}^\infty\tilde c_i^{\alpha/\mu}<\infty$ by \eqref{eq:technical_condition}, which gives the second bound in (ii.iii).
\end{proof}

\begin{proof}[Proof of Theorem \ref{thm:result_int_conv_CTRWs_with_independence_cond}]
	By setting $\psi=\sum_{j=0}^\infty c_j$ and $\bar X^n=X^n-\psi U^{n,2}$, we can remove the contribution of innovations with nonpositive indices. Since the estimate in the proof of Lemma \ref{lem:decomp_U^n_part_has_GD} gives $\operatorname{TV}_{[0,\infty)}(U^{n,2})\to0$ in probability, and $|H^n|^*_T$ is tight, working with $\bar X^n$ instead of $X^n$ changes nothing for the joint weak convergence of $(H^n,X^n)$ and the limit of the stochastic integrals.
	
	Consider the $\operatorname{GD\,mod\,CI}$ decomposition
	$\psi^{-1}\bar X^n=U^{n,1}+\sum_i V^{n,i}$,
	with the $\sigma^{n,i}_k$ and $V^{n,i}$ given in
	Proposition \ref{prop:correlated_CTRW_are_GDmodCI}, noting that
	$U^{n,1}$ has \eqref{eq:Mn_An_condition} for the filtration
	$\bar{\mathcal F}^n$ in Lemma \ref{lem:decomp_U^n_part_has_GD}. Moreover, note that the estimate \eqref{eq:5.14} gives $V^n_t\to0$ in probability for each $t\ge0$, so we can deduce that $(H^n,U^{n,1})\Rightarrow_{f.d.d.}(H,\psi^{-1}X)$ outside $\operatorname{Disc}_{\mathbb{P}}(H,X)$. Next, we argue that \eqref{eq:5.17} implies the condition \eqref{eq:indep_cond_int_conv_GDmodCI} of Theorem \ref{thm:Integral_converg_under_GDmodCI}. Indeed, whenever nonzero, $V^{n,i}_{\sigma^{n,i}_k}$ and
	$V^{n,i}_{\sigma^{n,i}_{k-1}}$ depend respectively on
	$\theta_{k-i+1}$ and $\theta_{k-i}$, whereas
	$\mathcal V^i_{n,k-2}$ involves only
	$\theta_1,\ldots,\theta_{k-i-1}$. The former innovations are $\mathcal I_k$-measurable, while the latter are either $\mathcal P_k$-measurable or part of the remaining independent innovations generating $\mathcal I_k$.
	Since $\sigma^{n,i}_1,\ldots,\sigma^{n,i}_k$ are
	$\mathcal P_k$-measurable, condition \eqref{eq:5.17} and mutual
	independence of the innovations thus give that $\mathcal H^i_{n,k}\vee\mathcal V^i_{n,k-2}$ is independent of $(V^{n,i}_{\sigma^{n,i}_k},
	V^{n,i}_{\sigma^{n,i}_{k-1}})$, which is \eqref{eq:indep_cond_int_conv_GDmodCI}. Consequently, as described in Remark \ref{rem:alternative_conditions_AVCI}, to apply Theorem \ref{thm:Integral_converg_under_GDmodCI} it only remains to verify conditions (a) and (b) of \cite[Thm.~4.1]{andreasfabrice_theorypaper} for $\bar X^n$.
	
	Let $\sigma^n_0=0$, $\sigma^n_k=L_k/n$ for CTRWs and $\sigma^n_k=k/n$ for moving averages, and set
	\begin{align*}
		Y^{n,k}_\delta
		&:=\sup{0\le s\le\delta}\,|\bar X^n_{\sigma^n_{k+1}+s}-\bar X^n_{\sigma^n_k}|\\
		&\le\sup{0\le s\le\delta}\,|X^n_{\sigma^n_{k+1}+s}-X^n_{\sigma^n_k}|+\psi\operatorname{TV}_{[0,\infty)}(U^{n,2}).
	\end{align*}
	Since $X^n_{\sigma^{n}_{k+1}\, +\, \bullet} \, - \, X^n_{\sigma^{n}_k} \; \sim \; X^n_{\tilde \sigma^{n}\, +\,\bullet} \, - \, X^n_0$ for all $n\ge1$, $k\ge0$, with $\tilde\sigma^n=\sigma^n_1$, we obtain
	\begin{align*}
		\Pro^n\Big(\, \sup{0 <  s \le \delta} \, |X^n_{\sigma^{n}_{k+1}\, +\, s}  -  X^n_{\sigma^{n}_k}|  > \lambda \Big) \; &= \; \Pro^n\Big(\, \sup{0 < s  \le \delta} \, |X^n_{\tilde \sigma^{n} + s}  -  X^n_{0}|  >  \lambda \Big)\\
		&\le \; \Pro^n\Big(\, \sup{0<  s  \le  2\delta} \, |X^n_{ s}  -  X^n_{0}|  > \lambda \Big) + \Pro^n\big(\, \tilde \sigma^{n} >  \delta \big),
	\end{align*}
	where $\tilde\sigma^n\to0$ in probability. The second summand then converges to zero as $n\to \infty$, and so does the first summand uniformly in $n$ as $\delta \to 0$ by $J_1$ or $M_1$ tightness of the $X^n$. Since $\operatorname{TV}_{[0,\infty)}(U^{n,2})\to0$ in probability and the above estimate is uniform in $k$, we obtain
	\[
	\lim_{\delta\downarrow0}\limsup_{n\to\infty}\sup{k\ge0}
	\Pro^n(Y^{n,k}_\delta>\lambda)=0,\qquad \lambda>0,
	\]
	which verifies (b). Turning to condition (a), for any $H^n$-stopping time $\nu^n$,
	the event $\{\nu^n\in[\sigma^n_k,\sigma^n_{k+1})\}$ lies in
	$\sigma(H^n_{t\wedge (L_{k+1}/n)}:t\ge0)
	\vee\mathcal P_{k+1}$. By adaptedness of $H^n$ and independence of $(J_\ell)_{\ell\ge1}$ and
	$(\theta_j)_{j\in\mathbb Z}$, we have that
	$(\theta_j,J_j)_{j>k+1}$ are jointly independent of this $\sigma$-algebra
	together with $\mathcal I_{k+1}$. By
	\eqref{eq:5.17}, we thus deduce that
	$\{\nu^n\in[\sigma^n_k,\sigma^n_{k+1})\}$ is independent of
	$\mathcal I_{k+1}\vee\sigma(\theta_j,J_j:j>k+1)$.
	Since $Y^{n,k}_\delta$ is measurable with respect to the latter, it follows that it is independent of
	$\{\nu^n\in[\sigma^n_k,\sigma^n_{k+1})\}$, which verifies (a).
\end{proof}

\begin{remark} \label{rem:modified_GDmodCI}
	A close inspection of the proof of Theorem \ref{thm:Integral_converg_under_GDmodCI} reveals that if instead of (ii.ii) in Definition \ref{defi:GD_mod_CI}, we assume  
	\begin{align}
		\E^n \left[ V^{n,i,\, \le}_{\sigma^{n,i}_k} \; | \; \mathcal{V}^i_{n,k-1}\, \right] \; = \; 0 \label{eq:remark_alternative_GDmodCI_eq1}
	\end{align}
	for every $n,k,i\ge 1$, then, in particular, we gain some degree of freedom towards (ii.iii), where it suffices to replace the first bound by the direct control 
	\begin{align} 
		\lim_{M \to \infty} \; \limsup_{n \to  \infty} \;  \; \Pro^n \biggl( \;  \sum_{k=1}^{K f(n)}\sum_{i=1}^\infty \;|V^{n,i, \, >}_{\sigma^{n,i}_{k}}|  \; \ge  \; M \, \biggr) \; =\; 0 \label{eq:remark_alternative_GDmodCI_eq2}
	\end{align}
	for all $K>0$. Under this alternative set of assumptions, Theorem \ref{thm:Integral_converg_under_GDmodCI} is clearly still valid. Indeed, in this case, proceed as in the estimate leading to \eqref{eq:4.27} without exchanging the expectation and the infinite sum until the small and large parts have been separated as in \eqref{eq:eq_for_remark_on_alt_set_of_cond_GDmodCI}. Then, only
	"pull out" the infinite sum for the first part and proceed for this term as in the proof of Theorem	\ref{thm:Integral_converg_under_GDmodCI}. For the second expectation, we then obtain the bound
	\begin{align}
		\E^n \biggl[\, \biggl(
		\sum_{i=1}^\infty
		\sup{j=1,...,f(n)K_\delta}
		\, \bigl| \, \sum_{k=1}^{j} Z^{n,i,\, >}_k\bigr|
		\biggr) \,\wedge \, 1 \, \biggr] \; &\le \;  \E^n \biggl[\, \bigl( \varepsilon \sum_{k=1}^{K_\delta f(n)} \sum_{i=1}^\infty \, |V^{n,i,\, >}_{\sigma^{n,i}_k}| \bigr) \,\wedge \, 1 \, \biggr] \notag \\
		&\le \;  \eta \; + \; \Pro^n \biggl( \sum_{k=1}^{K_\delta f(n)} \sum_{i=1}^\infty \, |V^{n,i,\, >}_{\sigma^{n,i}_k}|  \,\ge \, \frac{\eta}{\varepsilon} \biggr) \notag 
	\end{align}
	for all $\eta>0$. Due to \eqref{eq:remark_alternative_GDmodCI_eq2}, by first taking the limit superior in $n$ and then letting $\varepsilon\to 0$, the probability vanishes. Since $\eta$ is arbitrary, we deduce the desired convergence. 
\end{remark}

\begin{proof}[Proof of Corollary \ref{cor:CTRWs_alpha=1_are_GDmodCI}]
	We only need to verify the conditions \eqref{eq:remark_alternative_GDmodCI_eq1} and \eqref{eq:remark_alternative_GDmodCI_eq2} from Remark \ref{rem:modified_GDmodCI}, as the remaining properties of Definition \ref{defi:GD_mod_CI} follow in the exact same way as outlined in the proof of Proposition \ref{prop:correlated_CTRW_are_GDmodCI} (we recall, with $f(n)=n^\beta$). Due to the symmetry of the law of the $\theta_k$, \eqref{eq:remark_alternative_GDmodCI_eq1} is obtained immediately with the choice of $V^{n,i}$ and $\sigma^{n,i}_k$ made in \eqref{eq:definition_V_correlated_CTRWs_are_GDmodCI}. Moreover, since $t\mapsto \sum_{k=1}^{\lfloor n^\beta t\rfloor} n^{-\beta/\alpha} \theta_k$ is a subsequence of the zero-order moving averages \eqref{eq:MovAv} converging weakly in the $J_1$ Skorokhod space, to every $\delta>0$ there exists $\Gamma_\delta >0$ such that 
	$$ \sup{n\ge 1} \;\;  \Pro \Biggl( \sum_{k=1}^{Kn^\beta} \ind_{\{ |\theta_k| \, > \, n^{ \beta / \alpha} \}} \, > \, \Gamma_\delta \; \text{ or } \; n^{-\beta/\alpha} \operatorname{max}_{k=1,...,Kn^\beta} |\theta_k| \, > \, \Gamma_\delta \Biggr) \; \le \; \delta $$
	as the maximal number of large oscillations as well as the absolute size of jumps is tight (see e.g. \cite[Thm.~A.8 \& Cor.~A.9]{andreasfabrice_theorypaper}). However, this yields
	\begin{align*}
		\Pro \Biggl( \;  \sum_{k=1}^{K n^\beta}\sum_{i=1}^\infty \;|V^{n,i, \, >}_{\sigma^{n,i}_{k}}|  \; \ge  \; M \, \Biggr) \; &\le \; \Pro \Biggl( \;  \tilde c \, \sum_{k=1}^{K n^\beta} \;n^{-\beta/\alpha} |\theta_k| \ind_{\{ |\theta_k| \, > \, n^{ \beta / \alpha} \}} \; \ge  \; M \, \Biggr) \\
		&\le \; \delta \; + \; \Pro \bigl( \;  \tilde c \, \Gamma_\delta^2 \; \ge  \; M \, \bigr) \; = \; \delta
	\end{align*}
	for $M$ large enough, where $\tilde c= \psi^{-1} \sum_{i=1}^\infty \sum_{\ell=i}^\infty c_\ell <\infty$ due to \eqref{eq:technical_condition}. As $\delta$ was arbitrary, we deduce \eqref{eq:remark_alternative_GDmodCI_eq2}.
\end{proof}

\subsection{Proofs pertaining to Section \ref{sec:SDE}}\label{sect:proofs_SDE}
\begin{lem} \label{lem:tightness_with_sublinear_growth_cond}
	The solutions $X^n$ to \eqref{eq:SDE_approx} are stochastically bounded uniformly in $n$, that is $\lim_{\eta\to \infty} \operatorname{sup}_{n\ge 1} \Pro(|X^n|^*_T > \eta)= 0$ for all $T\ge 0$.
\end{lem}
\begin{proof}
	To alleviate notation, we only provide a proof for the case $C\equiv 0$ and $x=0$. On behalf of the decomposition \eqref{eq:uncorrelated_CTRW_decomp}, we can write $Z^n=M^n+A^n$, where the $M^n$ are local martingales with $|\Delta M^n|\le 1$, $M_0\equiv 0$ without loss of generality, and the $A^n$ are of tight total variation. Thus, with $\tilde{\sigma}:=|b|\vee |\mu|\vee |\sigma|$, the elementary inequality $(a+c)^2\le 2(a^2+c^2)$, and $\tilde{A}^n_t:= t + D^n_t+\operatorname{TV}_{[0,t]}(A^n)$, we obtain from the form of the solution \eqref{eq:SDE_approx} that
	\begin{align}
		(X^n_t)^2 \; \le \; 2\Big( \int_0^t \, \sigma(s,D^n_s,X^n_s)_{-} \, \diff M^n_s \Big)^2 \; + \; 2\tilde{A}^n_t\int_0^t \, \tilde \sigma(s,D^n_s,X^n_s)_{-}^2 \, \diff \tilde{A}^n_s, \label{eq:proof_lem_tightness_with_sublinear_growth_cond_1}
	\end{align}
	for $t\ge 0$, where we note $\tilde{A}^n$ is increasing. Now, let $T,\varepsilon,\eta>0$ and define stopping times $\tau_n:= \operatorname{inf}\{t>0: |X^n_t|>\eta \}\wedge T$ and $\rho_n:= \operatorname{inf}\{t>0: |D^n_t|>R \}$ where $R>0$ is chosen such that $\operatorname{sup}_{n\ge 1}\Pro(\rho_n\le T)\le \varepsilon$. The latter is possible since, by \eqref{eq:time_convergence_to_subordinator}, the $D^n$ converge on Skorokhod space and hence are stochastically bounded uniformly in $n$. Then, we may continue by 
	\begin{align}
		\Pro(|X^n|^*_T \, > \, \eta)
		&\le \; \Pro(|X^n_{\tau_n\wedge \rho_n}|^2 \, \ge \, \eta^2)
		+ \varepsilon \; \le \;
		\Pro\Big(\Big(\int_0^{\tau_n\wedge \rho_n}
		\, \!\!\!\sigma(s,D^n_s,X^n_s)_{-} \, \diff M^n_s
		\Big)^2 \, \ge \, \frac{\eta^2}{4} \Big) \notag \\
		&\quad{}+ \; \Pro\Big(K^2 \tilde{A}^n_{\tau_n\wedge\rho_n}
		\int_0^{\tau_n\wedge \rho_n} \, |X^n_{s-}|^{2p}
		\, \diff \tilde{A}^n_s \, \ge \, \frac{\eta^2}{4} \Big)
		+ \varepsilon \notag \\
		&\le \; \Pro\Big(\Big(\int_0^{\tau_n\wedge \rho_n}
		\, \sigma(s,D^n_s,X^n_s)_{-} \, \diff M^n_s
		\Big)^2 \, \ge \, \frac{\eta^2}{4} \Big)
		+ \; \Pro\Big(\tilde{A}^n_T \, \ge \,
		\frac{\eta^{1-p}}{2K} \Big) + \varepsilon,
		\label{eq:proof_lem_tightness_with_sublinear_growth_cond_2}
	\end{align}
	where $p$ comes from the strict sublinear growth bound \eqref{eq:sublinear_growth_condition_coeff_SDE}. Beyond \eqref{eq:sublinear_growth_condition_coeff_SDE}, we have used \eqref{eq:proof_lem_tightness_with_sublinear_growth_cond_1} as well as $|X^n_s|\le \eta$ for all $s<\tau_n$. We now apply Lenglart's inequality \cite[Lem. I.3.30b]{shiryaev} to the first term of \eqref{eq:proof_lem_tightness_with_sublinear_growth_cond_2} with the quadratic variation of $\int \sigma(s,D^n_s,X^n_s)_{-} \diff M^n_s$ as the $L$-domination process (justified since the integral with respect to $M^{n}$ is a local martingale, see e.g. \cite[Thm.~III.29]{protter}). This yields
	\begin{align}
		\Pro(|X^n|^*_T  \,>\,  \eta) \le  \frac{4(\gamma+K^2\eta^{2p})}{\eta^2} +  \Pro\Big([M^n]_T \,>\,  \frac {\gamma}{\eta^{2p}K^2} \Big)  + \Pro\Big(\tilde{A}^n_T  \, \ge\,  \frac {\eta^{1-p}}{2 K} \Big)  +  \varepsilon \label{eq:proof_lem_tightness_with_sublinear_growth_cond_3}
	\end{align}
	for all $\gamma>0$, where we have again made use of the sublinear growth \eqref{eq:sublinear_growth_condition_coeff_SDE} and the fact that $|\Delta M^n|\le 1$. Clearly, the $\tilde{A}^n$ are of tight total variation on $[0,T]$ and the tightness of the $[M^n]_T$ follows from another application of Lenglart's inequality, with $(|M^n|^*)^2$ as a valid $L$-domination process due to the Burkholder-Davis-Gundy inequality, the fact that $|\Delta M^n|\le 1$ and the tightness of $Z^n$ in the Skorokhod space (which implies the tightness of the $|Z^n|^*_T$, and thus, by the tight total variation of the $A^n$, the tightness of the $|M^n|^*_T$). Choosing $\gamma = \eta^\delta$ with $\delta \in (2p,2)$ yields $\gamma/\eta^2 \to 0$ and $\gamma/\eta^{2p} \to \infty$ as $\eta \to \infty$, so we deduce the claim from \eqref{eq:proof_lem_tightness_with_sublinear_growth_cond_3} and the fact that $\varepsilon>0$ was arbitrary.
\end{proof}

\begin{remark} \label{rem:7.6}
	In order to establish convergence $(Y^n, Z^n)\Rightarrow (Y,Z)$ on the Skorokhod product space, one would
	often aim to show that the finite-dimensional distributions of the pairs $(Y^n, Z^n)$ converge to
	those of $(Y,Z)$ and then prove tightness of the $Y^n$ and the $Z^n$ separately on their respective
	Skorokhod space. Indeed, the latter is sufficient for tightness of the pairs $(Y^n, Z^n)$ on the product
	space, due to the characterisation of compact sets in product Hausdorff spaces as precisely those
	closed sets which are contained in the Cartesian product of compact sets of the individual spaces,
	i.e.\ in our case $Q$ is compact in $(\D_{\mathbb{R}^k}[0,\infty),\rho)\times (\D_{\mathbb{R}^m}[0,\infty),\rho)$
	if and only if $Q$ is closed and $
	Q \subseteq K_1 \times K_2$
	with $K_1$ compact in $(\D_{\mathbb{R}^k}[0,\infty),\rho)$ and $K_2$ compact in $(\D_{\mathbb{R}^m}[0,\infty),\rho)$.
	On the other hand, convergence on the strong space $(\D_{\mathbb{R}^{k+m}}[0,\infty),\rho)$ comes with significantly more restrictive requirements (for details see e.g.~the
	compactness characterisations in \cite{whitt}).
\end{remark}

\begin{proof}[Proof of Theorem \ref{thm:SDE_approx_result}]
	Let $X_0=0$. As in Lemma \ref{lem:tightness_with_sublinear_growth_cond}, for simplicity of notation, we assume $C\equiv 0$ in \eqref{eq:sublinear_growth_condition_coeff_SDE}. In order to prove the claim, it suffices to show that 
	\begin{enumerate}[(1)]
		\item the $X^n$ are tight on the space $(\D_{\R}[0,\infty), \dJ)$;\label{it1:proof_SDE}
		\item \label{it2:proof_SDE} as well as there being convergence 
		\begin{align*} \int_{0}^\bullet \sigma(s, D^n_s,X^n_{s})_- \diff Z^n_s \; \Rightarrow \;  \int_{0}^\bullet \sigma(s, D^{-1}_s,Y_{s})_- \diff Z_{D^{-1}_s} 
		\end{align*}
		on $(\D_{\R}[0,\infty), \dJ)$, jointly with $(X^n,D^n,Z^n)$, whenever this triple converges weakly to $(Y,D^{-1},Z_{D^{-1}})$ on $(\D_{\R}[0,\infty), \dJ)^3$ for a càdlàg process $Y$. 
	\end{enumerate}
	Indeed, once this has been shown, by the Prohorov Theorem we know that the $X^n$ are relatively compact on $(\D_{\R}[0,\infty), \dJ)$, and, due to Remark \ref{rem:7.6} as well as the weak convergence $(D^n,Z^n) \Rightarrow (D^{-1},Z_{D^{-1}})$ on $(\D_{\R}[0,\infty), \dJ)^2$ (see \eqref{eq:time_convergence_to_subordinator} \& \eqref{eq:CTRW_$M_1$_conv}), we obtain the relative compactness of the triplets $((X^n,D^n,Z^n))_{n\ge 1}$ on the product space $(\D_{\R}[0,\infty), \dJ)^3$. Then, for every subsequence $((X^{n_k},D^{n_k},Z^{n_k}))_{k\ge 1}$, there exists a further subsequence $((X^{n_{k_\ell}},D^{n_{k_\ell}},Z^{n_{k_\ell}}))_{\ell \ge 1}$ which converges weakly on $(\D_{\R}[0,\infty), \dJ)^3$ to some càdlàg limit $(Y,D^{-1},Z_{D^{-1}})$ (where $Y$, a priori, may depend on the choice of the specific subsequence). To simplify the notation, we denote this subsequence again by $(X^n, D^n, Z^n)$. Due to the continuity of $b$, $\mu$ and $\sigma$, without loss of generality we can then deduce that 
	\begin{enumerate}
		\item[(3)]  \label{it3:proof_SDE} for all $\ell\ge 1$ and $t_1,...,t_\ell \in \Lambda$, where $\Lambda$ is a co-countable subset of $[0,\infty)$, it holds \begin{align*}& \Big(D^n,Z^n, \, [b,\mu,\sigma](t_1,D^n_{t_1},X^n_{t_1}),\,...,\, [b,\mu,\sigma](t_\ell,D^n_{t_\ell},X^n_{t_\ell}) \Big) \\
			& \qquad \quad  \Rightarrow  \;  \Big(D^{-1},Z_{D^{-1}}, \, [b,\mu,\sigma](t_1, D^{-1}_{t_1},Y_{t_1}), \,  ...,\, [b,\mu,\sigma](t_\ell, D^{-1}_{t_\ell}, Y_{t_\ell})\Big)
		\end{align*}
		on $(\D_{\R}[0, \infty), \dJ)^2 \times (\R^{3\ell}, |\cdot|\, )$, where we have used the shortcut array notation $[b,\mu,\sigma](\alpha,\beta,\gamma):=\allowbreak (b(\alpha,\beta,\gamma),\mu(\alpha,\beta,\gamma),\sigma(\alpha,\beta,\gamma))$;
		\item[(4)] \label{it4:proof_SDE} for any $T>0$, the sequence $(|[b,\mu,\sigma](\bullet,D^n_\bullet, X^{n}_\bullet)|^*_T)_{n\ge 1}$ is tight;
		\item[(5)] \label{it5:proof_SDE} for any $\delta>0$, the number of $\delta$-increments of the $([b,\mu,\sigma](\bullet,D^n_\bullet,X^{n}_\bullet))_{n \ge 1}$ over any interval $[0,T]$ is tight (see \eqref{eq:maxnumosc} for a precise definition of $\delta$-increments).
	\end{enumerate}
	Then, according to \cite[Prop.~2.12]{andreasfabrice_theorypaper} and Remark \ref{rem:weaker_integrand_conditions}, we obtain that 
	\begin{align*} \int_0^\bullet \, b(s, D^n_s, X^{n}_{s})_- \, \diff s  \; &\Rightarrow \; \int_0^\bullet \, b(s, D^{-1}_s, Y_{s})_- \, \diff s, \\ \int_0^\bullet \, \mu(s, D^n_s, X^{n}_{s})_- \, \diff D^n_s  \; &\Rightarrow \; \int_0^\bullet \, \mu(s, D^{-1}_s, Y_{s})_- \, \diff D^{-1}_s
	\end{align*}
	as $n \to \infty$, jointly with $(X^n,D^n,Z^n)$ and the
	stochastic integral in \eqref{it2:proof_SDE} on
	$(\D_{\R}[0,\infty),\dJ)^6$. Since both processes $t \mapsto \int_0^t b(s, D^{-1}_s,Y_{s})_- \diff s$ and $t \mapsto \int_0^t \mu(s, D^{-1}_s,Y_{s})_- \diff D^{-1}_s$ are continuous (recall that $D^{-1}$ is a continuous process), it is a well known consequence that 
	\begin{align} &X^n =  \! \int_0^\bullet \! b(s, D^n_s, X^n_{s})_-  \diff s  + \! \int_0^\bullet \! \mu(s, D^n_s, X^n_{s})_- \diff D^n_s + \!  \int_0^\bullet \! \sigma(s, D^n_s,X^n_{s})_-  \diff Z^{n}_s \notag \\
		&\Rightarrow    \int_0^\bullet \! b(s, D^{-1}_s, Y_{s})_-  \diff s + \!\int_0^\bullet \! \mu(s, D^{-1}_s, Y_{s})_-  \diff D^{-1}_s  + \! \int_0^\bullet \! \sigma(s, D^{-1}_s, Y_{s})_- \diff Z_{D^{-1}_s} \label{eq:proof_thm_6.1} \end{align}
	as $n \to \infty$ on $(\D_{\R}[0,\infty), \dJ)$. By continuous mapping, we then deduce that $Y_t$ equals the right-hand side of \eqref{eq:proof_thm_6.1} a.s., for all $t\notin \operatorname{Disc}_{\Pro}(Y,Z_{D^{-1}})$, and hence right-continuity gives that $Y$ solves the SDE \eqref{eq:SDE}. If $Y$ is the unique solution, then we get $X^n \Rightarrow Y$ on $(\D_{\R}[0,\infty), \dJ)$ for the original sequence of solutions $X^n$.
	
	Now, we will show \eqref{it1:proof_SDE} and \eqref{it2:proof_SDE}. To show \eqref{it1:proof_SDE}, we employ Aldous' $J_1$ tightness criterion \cite[Thm.~16.10]{billingsley}, requiring us to verify that, for every $T>0$ and $\eta>0$, $|X^n|^*_{T}$ is tight and
	\begin{align} \lim\limits_{\delta \downarrow 0 } \,\limsup\limits_{n\to \infty} \, \operatorname{sup}_{\tau,h}\Pro(|X^n_{\tau+h}-X^n_{\tau}|>\eta) =  0, \label{eq:Aldous_$J_1$_tightness_criterion}
	\end{align}
	where the inner supremum runs over all $\mathbbm{F}^n$--stopping times $\tau$ bounded by $T$ and all $h\in[0,\delta]$. The first part follows immediately from Lemma \ref{lem:tightness_with_sublinear_growth_cond}. Towards \eqref{eq:Aldous_$J_1$_tightness_criterion}, let $\varepsilon>0$ and fix $\eta,T>0$. Choose $R>0$ large enough such that $\operatorname{sup}_{n\ge 1} \Pro(|X^n|^*_{T+1} \vee |D^n|^*_{T+1}>R)\le \varepsilon/4$ and define stopping times $\rho_n:=\operatorname{inf}\{t>0:|X^n_t|\vee |D^n_t|>R\}$. Then, for any $\mathbbm{F}^n$-stopping time $\tau$ bounded by $T$ and $\delta \le 1$, we have
	\begin{align}
		\Pro(|X^n_{\tau+\delta}-X^n_{\tau}| \, >\, \eta) \; \le \;\frac \varepsilon 4 \; + \; \Pro(|X^n_{(\tau+\delta)\wedge \rho_n}-X^n_{\tau\wedge \rho_n}| \, >\, \eta). \label{eq:eq1_proof_of_thm_5.1}
	\end{align}
	By \eqref{eq:SDE_approx}, the definition of $\rho_n$ and the strict sublinear growth condition \eqref{eq:sublinear_growth_condition_coeff_SDE}, the second term on the right side of \eqref{eq:eq1_proof_of_thm_5.1} can be further estimated by 
	\begin{align}
		\Pro(|X^n_{(\tau+\delta)\wedge \rho_n}-X^n_{\tau\wedge \rho_n}| \, >\, \eta) \; &\le \; \ind_{\delta \, >\, \frac \eta {3KR^p}} \; + \; \Pro \Big(D^n_{(\tau+\delta)\wedge \rho_n}-D^n_{\tau \wedge \rho_n} \, >\, \frac \eta {3KR^p} \Big) \notag  \\
		&\qquad + \; \Pro \Big(\Big|\int_{\tau \wedge \rho_n}^{(\tau+\delta)\wedge \rho_n} \sigma(s,D^n_s,X^n_s)_- \diff Z^n_s \Big| \, >\, \frac \eta {3} \Big) \label{eq:eq2_proof_of_thm_5.1}
	\end{align}
	where we recall that $D^n=n^{-\beta}N_{n\bullet}$. The first term on the right side of \eqref{eq:eq2_proof_of_thm_5.1} disappears uniformly in $n$ as $\delta\to0$, and the second term vanishes in the iterated limit \eqref{eq:Aldous_$J_1$_tightness_criterion}. Indeed, recall that since $D^n\Rightarrow D^{-1}$ on the $J_1$ space and $D^{-1}$ is continuous, it holds in particular that 
	\begin{align}
		\lim\limits_{\delta \downarrow 0}\; \limsup_{n\to \infty}\; \Pro\Big( \; \sup{\substack{0\, \le \,  s\, \le \, t \, \le\,  T+1 \\ t-s\, \le \, \delta}} (D^n_t-D^n_s) \, > \, \lambda\Big) \; = \; 0 \label{eq:uniform_converg_D^n_proof_SDE}
	\end{align}
	for all $\lambda>0$. Thus, it suffices to investigate the convergence of the third term on the right-hand side of \eqref{eq:eq2_proof_of_thm_5.1}. Choose $a>0$ such that $(9aKR^p/\eta)^2\le \varepsilon/4$, set $\tilde{C}_a:=1+\operatorname{sup}_{n\ge 1} n^\beta |\E[\zeta^n_1 \ind_{\{|\zeta^n_1|\leq a/2\}}]|<\infty$ (cf. Proposition \ref{prop:aux_result_CTRW_has_GD}), and let $Z^n=M^{n,a}+A^{n,a}$ be good decompositions \eqref{eq:uncorrelated_CTRW_decomp} so that $|\Delta M^{n,a}|\le a$. Making use of the concrete form of the good decompositions as well as again the definition of $\rho_n$ and the strict sublinear growth condition \eqref{eq:sublinear_growth_condition_coeff_SDE}, we obtain
	\begin{align}
		\Pro\Big(\Big|\int_{\tau \wedge \rho_n}^{(\tau+\delta)\wedge \rho_n} & \hspace{-2ex}\sigma(s,D^n_s,X^n_s)_- \diff Z^n_s \Big| \, >\, \frac \eta 3 \Big) \; \le \;  \Pro\Big(\Big|\int_{\tau \wedge \rho_n}^{(\tau+\delta)\wedge \rho_n} \hspace{-2ex} \sigma(s,D^n_s,X^n_s)_- \diff M^{n,a}_s \Big|^2  > \frac {\eta^2} {9^2} \Big) \notag \\
		& \hspace{-2ex}+  \Pro\Big(|\Delta Z^n|^*_{T+\delta} \!\!\sum_{k=N_{n\tau}+1}^{N_{n(\tau+\delta)}} \hspace{-2ex}\ind_{\{|\zeta^n_k| >  a/2 \}} > \frac \eta {9KR^p}  \Big)  +  \Pro\Big( D^n_{\tau+\delta}-D^n_{\tau} >\frac \eta{9KR^p\tilde{C}_a} \Big).\label{eq:eq3_proof_of_thm_5.1}
	\end{align}
	For the first probability term of \eqref{eq:eq3_proof_of_thm_5.1}, similarly to the proof of Lemma \ref{lem:tightness_with_sublinear_growth_cond} we apply Lenglart's inequality \cite[Lem.~I.3.30b]{shiryaev} in order to obtain 
	\begin{align*}\Pro\Big(\Big|\int_{\tau \wedge \rho_n}^{(\tau+\delta)\wedge \rho_n} &\sigma(s,D^n_s,X^n_s)_- \diff M^{n,a}_s \Big|^2 \, >\, \frac{\eta^2}{9^2} \Big) \; \le \; \frac{9^2(\gamma+(aKR^p)^2)}{\eta^2} \\
		&\qquad + \Pro\Bigl([M^{n,a}]_{\tau+\delta}-[M^{n,a}]_{\tau}\, > \, \frac{\gamma}{K^2R^{2p}}\Bigr)
	\end{align*}
	for all $\gamma>0$. Choose $\gamma>0$ such that $9^2\gamma/\eta^2\le \varepsilon/4$. Hence, we can further bound \eqref{eq:eq3_proof_of_thm_5.1} by
	\begin{align}
		&\Pro\Bigl(\Bigl|
		\int_{\tau\wedge\rho_n}^{(\tau+\delta)\wedge\rho_n}
		\sigma(s,D^n_s,X^n_s)_- \diff Z^n_s
		\Bigr| > \frac{\eta}{3} \Bigr) \notag \\
		&\le \frac{\varepsilon}{2}
		+ \Pro\Bigl([M^{n,a}]_{\tau+\delta} - [M^{n,a}]_\tau
		> \frac{\gamma}{K^2R^{2p}} \Bigr) \notag \\
		&\quad{}+ \Pro\Bigl(|Z^n|^*_{T+\delta}
		\sum_{k=N_{n\tau}+1}^{N_{n(\tau+\delta)}}
		\ind_{\{|\zeta^n_k|>a/2\}} > \frac{\eta}{18KR^p} \Bigr)
		+ \Pro\Bigl(D^n_{\tau+\delta} - D^n_\tau
		> \frac{\eta}{9KR^p\tilde{C}_a} \Bigr) \notag \\
		&\le \frac{\varepsilon}{2}
		+ \Pro\Bigl(\sum_{k=N_{n\tau}+1}^{N_{n(\tau+\delta)}}
		(\zeta^n_k)^2 \ind_{\{|\zeta^n_k|\le a\}}
		> \frac{\gamma}{4K^2R^{2p}} \Bigr) \notag \\
		&\quad{}+ \Pro\Bigl(n^{-\beta}D^n_{T+\delta}
		> \frac{\gamma^2}{4\tilde{C}_a K^2R^{2p}} \Bigr)
		+ \Pro\Bigl(n^{-\beta}D^n_{T+\delta}
		> \frac{\gamma}{18\tilde{C}_a^2 K^2R^{2p}} \Bigr) \notag \\
		&\quad{}+ \Pro\Bigl(|Z^n|^*_{T+\delta}
		\sum_{k=N_{n\tau}+1}^{N_{n(\tau+\delta)}}
		\ind_{\{|\zeta^n_k|>a/2\}} > \frac{\eta}{18KR^p} \Bigr)
		+ \Pro\Bigl(D^n_{\tau+\delta} - D^n_\tau
		> \frac{\eta}{9KR^p\tilde{C}_a} \Bigr),
		\label{eq:third_eq_proof_SDE}
	\end{align}
	where we have made use of the specific form of $M^{n,a}$ given after \eqref{eq:uncorrelated_CTRW_decomp}. Since the $Z^n$ are $J_1$ tight, we can choose $C_\varepsilon>0$ such that $\Pro(|Z^{n}|^*_{T+1} > C_\varepsilon)\le \varepsilon/4$. We will examine each term of \eqref{eq:third_eq_proof_SDE} individually. Clearly, since $D^n_{T+\delta}$ are tight, we directly deduce 
	$$ \limsup\limits_{n\to \infty} \Big[\Pro\Big( n^{-\beta} \, D^n_{T+\delta} \, >\, \frac {\gamma^2} {12 a\tilde{C}_a K^2R^{2p}} \Big) \, + \, \Pro\Big( n^{-\beta} D^n_{T+\delta} \, >\, \frac \gamma {6\tilde{C}_a^2K^2R^{2p}} \Big) \Big] \; = \; 0.$$
	As $\lim_{\delta \downarrow 0} \limsup_{n\to \infty} \Pro (D^n_{\tau+\delta}-D^n_{\tau} \, > \,  \eta/(9KR^p\tilde{C}_a))=0$ due to \eqref{eq:uniform_converg_D^n_proof_SDE}, it only remains to consider the second and fifth terms of \eqref{eq:third_eq_proof_SDE}. For the latter, by independence of $N_{n\bullet}$ and the $\zeta^n_k$ (recall the CTRW is uncoupled) as well as the identical distribution of the $\zeta^n_k$,
	\begin{align*}
		\Pro\Big(|Z^n|^*_{T+\delta} \, \, &\sum_{k=N_{n\tau}+1}^{N_{n(\tau+\delta)}} \hspace{-2ex}\ind_{\{|\zeta^n_k|\, > \, a/2 \}} \, >\, \frac \eta {18KR^p} \Big) \; \le \; \frac{\varepsilon}{4} \,+\, \Pro\Big(\sum_{k=N_{n\tau}+1}^{N_{n(\tau+\delta)}} \hspace{-2ex}\ind_{\{|\zeta^n_k|\, > \, a/2 \}} \, >\, \frac \eta {18KR^pC_\varepsilon} \Big) \\
		&\le \; \frac{\varepsilon}{4} \,+\, \Pro\Big(D^n_{\tau+\delta}-D^n_{\tau} \, > \, \lambda \Big) \, + \, \Pro\Big(\sum_{k=1}^{\lfloor n^\beta \lambda \rfloor} \ind_{\{|\zeta^n_k|\, > \, a/2 \}} \, >\, \frac \eta {18KR^pC_\varepsilon} \Big)\\
		&\le \; \frac{\varepsilon}{4} \,+\, \Pro\Big(D^n_{\tau+\delta}-D^n_{\tau} \, > \, \lambda \Big) \, + \, \Pro\Big(\tilde{N}^{\lambda}_{a/2}(\tilde{Z}^{n}) \, >\, \frac \eta {18KR^pC_\varepsilon} \Big)
	\end{align*}
	for any $0<\lambda<1$, where $\tilde{Z}^{n}:= \sum_{k=1}^{\lfloor n^\beta\bullet \rfloor} \zeta^n_k$ is tight in $J_1$ as a zero-order moving average and $\tilde{N}^T_{a/2}(\tilde Z^n)$ the maximal number of $a/2$-increments of $\tilde Z^n$ on $[0,T]$ (with the quantity defined as in \eqref{eq:maxnumosc}). Indeed, the independence of $N_{n\tau}$ and $(\zeta^n_{N_{n\tau}+k})_{k\ge1}$  can be seen as follows: note that $N_{n\tau}\le \ell$ is equivalent to $\tau < L_{\ell+1}/n$, and it therefore suffices to have that there is independence of sets of the form $\{\tau\le q\} \cap \{q< L_{\ell+1}/n\}$ and the $\sigma$-algebra $\sigma(\zeta^n_{\ell+1},\zeta^n_{\ell+2},...)$ for each $\ell\ge 0$ and $q \in \Q_{\ge 0}$. The former sets, however, belong to $\sigma(J_1,J_2,...)\vee \sigma(\zeta_1,...,\zeta_\ell)$ (throughout the proof, we suppress $\mathcal G^n_\infty$ from the notation), and independence follows due to the fact that the CTRWs are uncoupled. Now, according to the classical tightness criterion based on the $J_1$ modulus of continuity (see e.g. \cite[Def.~A.7]{andreasfabrice_theorypaper}), we deduce that $\lim_{\lambda \downarrow 0} \limsup_{n\to \infty}\Pro(\tilde{N}^{\lambda}_{a/2}(\tilde Z^{n}) > \eta/(18KR^pC_\varepsilon) )=0$. Further, again due to \eqref{eq:uniform_converg_D^n_proof_SDE}, $\lim_{\delta \downarrow 0} \limsup_{n\to \infty} \Pro(D^n_{\tau+\delta}-D^n_{\tau} \, > \, \lambda)=0$, and thus
	$$ \lim_{\delta \downarrow 0} \; \limsup_{n\to \infty}  \;\Pro\Big(|Z^n|^*_{T+\delta} \, \, \sum_{k=N_{n\tau}+1}^{N_{n(\tau+\delta)}} \ind_{\{|\zeta^n_k|\, > \, a/2 \}} \, >\, \frac \eta {18KR^p} \Big) \; \leq \; \frac{\varepsilon}{4}.$$
	Finally, for the second term of \eqref{eq:third_eq_proof_SDE} we can proceed similarly to how we did for the previous (fifth) term, noting that $[\tilde Z^n]=\sum_{k=1}^{\lfloor n^\beta \bullet \rfloor} (\zeta^n_k)^2$ is tight on the $J_1$ space. Indeed, combine Theorem \ref{thm:CTRW_has_GD}, \cite[Cor.~2.8]{jakubowskimeminpages} and the remark before \cite[Cor.~C.2]{andreasfabrice_theorypaper} to deduce the tightness. For $0\le h\le\delta$, the second term with lag $h$ is bounded by
	$\Pro(D^n_{\tau+\delta}-D^n_\tau>\lambda)
	+\Pro([\tilde Z^n]_\lambda>\gamma/(4K^2R^{2p}))$.
	By \eqref{eq:uniform_converg_D^n_proof_SDE} and $J_1$-tightness of
	$[\tilde Z^n]$ with $[\tilde Z^n]_0=0$, this bound vanishes uniformly
	in $\tau$ as $n\to\infty$, then $\delta\downarrow0$, then
	$\lambda\downarrow0$. Thus we have proven \eqref{it1:proof_SDE}.

	To show the convergence in \eqref{it2:proof_SDE}, we would like to invoke \cite[Thm.~4.1]{andreasfabrice_theorypaper}. The good decompositions of the $Z^n$ follow from Theorem \ref{thm:CTRW_has_GD}. It only remains to show that the pairs $(\sigma(\bullet, D^n_\bullet, X^n_\bullet),Z^n)$ satisfy the conditions (a)\&(b) set out in \cite[Thm.~4.1]{andreasfabrice_theorypaper}. Taking $\sigma^n_k:= L_k/n$ the jump times of $Z^n$, condition (b) follows in analogy to the approach set out at the end of the proof of Theorem \ref{thm:result_int_conv_CTRWs_with_independence_cond}. In terms of (a), we note that $H^n:=\sigma(\bullet,D^n_\bullet,X^n_\bullet)$ is adapted to \eqref{eq:filtration_CTRWs}. Thus, we deduce
	$$ (H^n_{t_1},...,H^n_{t_\ell})^{-1}(A_1\times...\times  A_\ell)\cap  \{q_2<\sigma^{n}_{k+1}\} \cap \{\sigma^n_k\le q_1\} \, \in \, \sigma\bigl(\theta_1,...,\theta_k, J_1,...,J_{k+1}\bigr)$$
	for all $A_1,...,A_\ell \in \B(\R)$, $n,k\ge 1$, $q_1,q_2 \in \Q$ and $t_1,...,t_\ell\le q_2$, while $Z^n_{\sigma^n_{k+1}+\, \bullet}-Z^n_{\sigma^n_k}$ is $\sigma(\theta_{k+i},J_{k+1+i}:i\ge 1)$--measurable. Applying \cite[Lem.~4.3]{andreasfabrice_theorypaper} yields condition (a) and hence the desired convergence.
\end{proof}

\subsection{Proofs pertaining to Section \ref{sect:SDDE}}

Let $X^n$ be a sequence of solutions to \eqref{eq:SDDE_approx} for the given integrators $Z^n$. We then define
$$ \Xi^n_t \; := \; \frac1{c n^{1/\alpha}}
\sum_{j=1}^{\lfloor nt\rfloor}H_j^n\zeta_j,\qquad t\ge0,$$
where $H^n_j:=\sigma(j/n,X^n_{(j/n-r)-})$ satisfies $|H^n_j|\le R$, with $R>0$ a uniform bound from \eqref{eq:boundedness_cond_SDDEs}. Replacing the $\zeta_j$ by their definition \eqref{eq:linear_process_innovations}, we can decompose $\Xi^n_t = U^n_t - R^n_t$, where
\begin{equation*}
	U^n_t\; := \; \frac1{c n^{1/\alpha}} \sum_{\ell=1-\mathcal J}^{\lfloor nt\rfloor} A_\ell^n \theta_\ell,\quad R^n_t \; := \; \frac1{c n^{1/\alpha}} \sum_{\lfloor nt \rfloor - \mathcal J<\ell\le \lfloor nt \rfloor}\theta_\ell
	\sum_{i=\lfloor nt \rfloor -\ell+1}^{\mathcal J}c_iH_{\ell+i}^n ,\end{equation*}
with
$A_\ell^n:=
\sum_{0\le i\le \mathcal{J}, \,\ell+i\ge1}c_iH_{\ell+i}^n$. Observe that $U^n$ is given by a discrete Itô integral and a remainder, namely
$$ U^n_t \; = \; \int_{0}^t G^n_{s-} \diff Y^n_s \, + \, W^n_t, \qquad t\ge 0,$$
where $Y^n$ is the zero-order moving average and $G^n_s:= c^{-1} A^n_{\lfloor ns\rfloor+1}$ is a pure jump process with $|G^n|\le R$ for all $n\ge 1$, while $W^n$ converges to zero almost surely. We consider only $n> \mathcal{J} / r$ and stress that each $G^n$ is then adapted to $  \mathcal{F}^n_t:=\sigma\big(\theta_j:\,j\le\lfloor nt\rfloor\big)$ with the $Y^n$ having good decompositions \eqref{eq:Mn_An_condition} for this filtration (see Theorem \ref{thm:CTRW_has_GD}).

\begin{lem} \label{prop:tight_sup_osc_soln_SDDEs}
	Recall the definition of $N_\delta^T$ in \eqref{eq:maxnumosc}.  For every $T>0$ and $\delta>0$, the sequences $(|X^n|^*_T)_{n\ge 1}$ and $(N_\delta^T(X^n))_{n\ge 1}$ are tight.
\end{lem}
\begin{proof}
	Fix $T>0$. Towards the first claim, due to \eqref{eq:boundedness_cond_SDDEs}, it suffices to show that $(|\Xi^n|^*_T)_{n\ge 1}$ is tight. To this end, recall the above decomposition $\Xi^n_t = U^n_t - R^n_t$. Since $Y^n$ has \eqref{eq:Mn_An_condition}, in the form
	\eqref{eq:uncorrelated_CTRW_decomp} (with $N_{nt}=\lfloor nt\rfloor$, $\beta=1$, and $|G^n|\le R$), one can proceed similarly to the proof of Lemma \ref{lem:tightness_with_sublinear_growth_cond} to deduce the tightness of $(|U^n|^*_T)_{n\ge 1}$. Next, for the processes $R^n$, it holds that $|R^n_t| \le \mathcal J R n^{-1/\alpha} \operatorname{max}_{1-\mathcal{J}\le \ell \le \lfloor Tn \rfloor }|\theta_\ell|$. Thus, utilising the tail probabilities $\Pro(|\theta_0|>x) \lesssim x^{-\alpha}$ (see, e.g., \cite[Thm.~4.5.2]{whitt} for $1<\alpha<2$, and \cite[(8.3)]{feller} in the case $\alpha=2$ noting that a slowly varying function is asymptotically negligible against any positive power), we have
	$$\operatorname{sup}_{n\ge 1} \Pro(|R^n|^*_T > x) \; \lesssim \; x^{-\alpha} \; \conv{x\to \infty} \; 0,$$
	which yields the desired conclusion.
	
	Turning to the second claim, it is straightforward to see from the definition \eqref{eq:maxnumosc} that $N_{\delta}^T(x+y) \le N_{\delta/2}^T(x)+N_{\delta/2}^T(y)$. Therefore, it suffices to show the tightness of $N_{\delta}^T$ for the three processes $\int B^n_{s-} \diff s$, $U^n$, and $R^n$, where $B^n:= b(\bullet, X^n_{\bullet-r})$. Fix $T,\delta>0$. By \eqref{eq:boundedness_cond_SDDEs}, this is immediate for the first process. In order to establish the tightness of $(N_\delta^T(U^n))_{n\ge 1}$, we can proceed as in the proof of Theorem \ref{thm:SDE_approx_result}, from \eqref{eq:Aldous_$J_1$_tightness_criterion} onwards: this establishes $J_1$ tightness of the $U^n$ which directly implies the desired $N_\delta^T$-tightness (see, e.g., \cite[Cor.~A.9]{andreasfabrice_theorypaper}). It remains to consider $R^n$. Since $|R^n_u| \le \mathcal J R n^{-1/\alpha} \operatorname{max}_{\lfloor nu \rfloor- \mathcal J < \ell \le \lfloor nu \rfloor }|\theta_\ell|$, we deduce that $|R^n_u|\geq  \delta/2$ implies
	\begin{align} |\theta_\ell| \; \ge \; \delta n^{1/\alpha} / (2\mathcal J R).\label{eq:proof_lemma_7.8}\end{align}
	for some $ \lfloor nu \rfloor- \mathcal J < \ell\leq \lfloor nu \rfloor $.
	Hence, if there are $m$ distinct $k\in \{ 0,\ldots, \lfloor nT \rfloor\}$ with $|R^n_{k/n}| \geq \delta /2$, then there are at least $m / \mathcal{J}$ distinct $\ell \in \{ 1-\mathcal{J},\ldots,\lfloor nT \rfloor\}$ satisfying \eqref{eq:proof_lemma_7.8}. Now recall that $R^n_u=R^n_{\lfloor nu \rfloor / n}$, so we can let $ t_i=k_i/n $ for $i=1,\ldots,2N$ be times that achieve $N=N_\delta^T(R^n)$ in the definition \eqref{eq:maxnumosc}. For each $j=1,\ldots,N$, we have $|R^n_{t_{2j}}-R^n_{t_{2j-1}}|\geq \delta$, so $t_{2j-1}<t_{2j}$ with $|R^n_{t_{2j}}|\geq \delta/2$ or $|R^n_{t_{2j-1}}|\geq \delta/2$. Consequently, there must be at least $N/2$ distinct times $t_i$ with $|R^n_{t_i}|\geq \delta/2$, and hence the above gives at least $N / (2 \mathcal{J})$ distinct indices $\ell \in \{ 1-\mathcal{J},\ldots,\lfloor nT \rfloor\}$ satisfying \eqref{eq:proof_lemma_7.8}. So, we conclude
	$$ N_\delta^T(R^n) > x \; \implies \; \# \big\{ \ell : 1- \mathcal{J}\le \ell \le \lfloor nT\rfloor , \, |\theta_\ell| \ge \delta n^{1/\alpha}/(2\mathcal JR) \big\} > x/(2\mathcal{J}).$$
	Setting
	$$ K_n \; := \; \# \big\{ \ell : 1- \mathcal{J} \le \ell \le \lfloor nT\rfloor , \, |\theta_\ell| \ge \varepsilon n^{1/\alpha} \big\}= \sum_{\ell=1-\mathcal{J}}^{\lfloor nT\rfloor} \ind_{\{|\theta_\ell|\geq \varepsilon n^{1/\alpha}\}}$$
	with $\varepsilon:=\delta /(2\mathcal JR)$, it only remains to show tightness of $(K_n)_{n\geq 1}$. Noting that
	\begin{align*}
		\mathbb{E}[K_n]
		\; = \; 
		\sum_{\ell=1-\mathcal{J}}^{\lfloor nT\rfloor}
		\Pro \big(
		|\theta_\ell|\geq \varepsilon n^{1/\alpha}
		\big) \; = \;
		(\lfloor nT\rfloor + \mathcal{J})\,
		\Pro \big(
		|\theta_0|\geq \varepsilon n^{1/\alpha}
		\big)
	\end{align*}
	with $\Pro (
	|\theta_0|\geq \varepsilon n^{1/\alpha}
	)  \lesssim n^{-1} \varepsilon^{-\alpha} \lesssim n^{-1} \delta^{-\alpha}$, this follows from Markov's inequality.
\end{proof}

\begin{proof}[Proof of Theorem~\ref{thm:5.18}]
	Let $Z$ denote the weak limit of $c^{-1}Z^n$ and let $X^n$ be an arbitrary subsequence of solutions still indexed by $n$. By Lemma \ref{prop:tight_sup_osc_soln_SDDEs}, we have tightness of $|X^n|^*_T$ and  $N_\delta^T(X^n)$ for any $\delta >0$ and $T>0$. Thus, we can find a sequence $(n_k)_{k \ge 1}$ and a cocountable set $D \; \subseteq \; [0,\infty)$ so that, for every $\ell \ge 1$ and all $t_1, \ldots, t_\ell \in D$, we have
	\begin{align*}
		\big( c^{-1} Z^{n_k}, \, X^{n_k}_{t_1}, \ldots, X^{n_k}_{t_\ell} \big) \; \Rightarrow \; \big( \tilde Z, \, \tilde X_{t_1}, \ldots, \tilde X_{t_\ell} \big)
	\end{align*}
	on $( \D_\R[0,\infty), \dM ) \times ( \R^\ell, |\cdot| )$, where $(\tilde{X},\tilde{Z})$ are two given c\`adl\`ag processes on a common probability space, with $\tilde Z$ having the same law as $Z$. By removing a countable set, we can take $D\cap \operatorname{Disc}_{\Pro}(\tilde X)=\emptyset$, and we then let $\tilde{D}:=\{t+r:t\in D\}\cup \{ t\in[0,r): \Delta \eta_{t-r}=0\}$. Since $(|X^n|^*_T)_{n\ge 1}$ is tight and $b,\sigma$ are continuous, we can work on events where the coefficients $b$ and $\sigma$ are uniformly continuous to conclude that $(N_\delta^T(b(\bullet, X^n_{\bullet-r})))_{n\ge 1}$, $(N_\delta^T(\sigma(\bullet, X^n_{\bullet-r})))_{n\ge 1}$ are tight, for all $T,\delta >0$, due to the tightness of $(N_\delta^T(X^n))_{n\ge 1}$. Furthermore, by continuous mapping, without loss of generality we also deduce  
	\begin{align*} (c^{-1} Z^{n_k}, B^{n_k}_{t_1},...,B^{n_k}_{t_\ell}, \Sigma^{n_k}_{t_1},...,\Sigma^{n_k}_{t_\ell}) \Rightarrow (\tilde Z, \tilde B_{t_1},...,\tilde B_{t_\ell}, \tilde \Sigma_{t_1},...,\tilde \Sigma_{t_\ell}) \end{align*}
	for every $\ell \ge 1$ and all $t_1, \ldots, t_\ell \in \tilde{D}$, where $B^n= b(\bullet, X^n_{\bullet-r})$ and $\Sigma^n=\sigma(\bullet, X^n_{\bullet-r})$, and $\tilde B$ and $\tilde \Sigma$ are defined analogously with $\tilde X$ in place of $X^n$. The sequence $(\Sigma^{n_k}, c^{-1}Z^{n_k})_{k\ge 1}$ satisfies the assumptions of Remark \ref{rem:weaker_integrand_conditions} associated to Theorem \ref{thm:result_int_conv_CTRWs_with_independence_cond}. Indeed, for $n\geq 1$ such that $\mathcal{J}/n < r$, it suffices to note the following: for each $k\ge 1$, the process $t\mapsto X^n_{(t\wedge k/n)-r}$ is $\sigma(\theta_j : 1-\mathcal{J} \le j \le k-\lceil nr\rceil)$-measurable, in particular it is
	$\sigma(\theta_{k-\ell} : \ell \ge \mathcal{J}+1)$-measurable since $\mathcal{J}+1 \le \lceil nr\rceil$, and therefore
	$H^n=\Sigma^n$ satisfies \eqref{eq:5.17}.  Set $Y^{1,n_k}:=\int B^{n_k}_{s-}\diff s$ and
	$Y^{2,n_k}:=c^{-1}\int \Sigma^{n_k}_{s-}\diff Z^{n_k}_s$,
	and define $\tilde Y^1:=\int \! \tilde B_{s-}\diff s$ and
	$\tilde Y^2:=\int \tilde\Sigma_{s-}\diff\tilde Z_s$.
	As in \cite[Prop.~2.12]{andreasfabrice_theorypaper}, using
	Remark \ref{rem:weaker_integrand_conditions} and retaining
	the finite evaluations of $X^{n_k}$ in
	the arguments, we obtain
	\begin{align*}
		\big(X^{n_k}_{t_1},\ldots,X^{n_k}_{t_\ell},c^{-1}Z^{n_k},Y^{1,n_k},Y^{2,n_k}\big)
		\Rightarrow
		\big(\tilde X_{t_1},\ldots,\tilde X_{t_\ell},\tilde Z,\tilde Y^1,\tilde Y^2\big)
	\end{align*}
	on $(\R^\ell,|\cdot|)\times(\D_{\R}[0,\infty),\dM)^3$,
	for every $\ell\ge1$ and all $t_1,\ldots,t_\ell\in D$.  As the paths $t\mapsto \int_0^t \tilde B_{s-} \diff s$ are continuous, we can thus conclude that
	\begin{equation} X^{n_k}  =  \eta(0)  + \int_0^\bullet \!B^{n_k}_{s-} \diff s  +  c^{-1} \int_0^\bullet \Sigma^{n_k}_{s-} \diff Z^{n_k}_s \; \Rightarrow \; \eta(0)  +  \int_0^\bullet \!\tilde{B}_{s-} \diff s +  \int_0^\bullet \tilde \Sigma_{s-} \diff \tilde Z_s . \label{eq:proof_thm_6.2}
	\end{equation}
	on $(\D_{\R}[0,\infty),\dM)$, which holds jointly with $c^{-1}Z^{n_k}$ and $X^{n_k}_t$ for $t\in D$.
	Since $X^{n_k}_t=\eta(0)+Y^{1,n_k}_t+Y^{2,n_k}_t$, it then follows by continuous mapping that $\tilde X_t=\eta(0)+\tilde Y^1_t+\tilde Y^2_t$ a.s., for all $t\in D\setminus\operatorname{Disc}_{\Pro}(\tilde Y^2)$. By right-continuity, and the definitions of $\tilde B$ and
	$\tilde\Sigma$, we thus have that $\tilde X$ solves the limiting SDDE.
\end{proof}

\begin{proof}[Proof of Corollary \ref{cor:ext_SDDE_conv}]  Define the processes 
	$$ \Lambda^n_t \; := \; c^{-1}  \int_0^t \, \Gamma^n_{s-} \; \diff Z^n_s,\qquad \Gamma^n_t:=\tilde\sigma (t, X^n_{[t-r, t]}),$$
	for an arbitrary subsequence of solutions $X^n$. Note that $\Gamma^n$ is uniformly bounded and Lipschitz continuous (a.s.) on any $[0,T]$, with a Lipschitz constant $K_T$ that is uniform in $n$. In particular, Arzel\`a--Ascoli gives tightness of $\Gamma^n$. Rewriting $c \Lambda^n_t $ as
	\[
	c\Lambda^n_t \;=\;
	\sum_{k=1}^{\lfloor nt\rfloor}\Gamma^n_{\frac{k}{n}-}\bigl(Z^n_{\frac{k}{n}}-Z^n_{\frac{k-1}{n}}\bigr)
	\;=\;\Gamma^n_{\frac{\lfloor nt\rfloor}{n}}Z^n_{\frac{\lfloor nt\rfloor}{n}}
	+\sum_{k=1}^{\lfloor nt\rfloor}\bigl(\Gamma^n_{\frac{k-1}{n}}-\Gamma^n_{\frac{k}{n}}\bigr)Z^n_{\frac{k-1}{n}},
	\]
	it is straightforward to obtain the tightness of  $(|\Lambda^n|^*_T)_{n\ge 1}$ and $(N^T_\delta(\Lambda^n))_{n\ge 1}$ for any $T,\delta>0$, due to the tightness of $Z^n$ and $\Gamma^n$ in the first term and the bound $TK_T|Z^n|^*_T$ on the total variation of the second term. Treating the drift and the other integral as in the proof of Lemma \ref{prop:tight_sup_osc_soln_SDDEs}, we obtain the tightness of $(|X^n|^*_T)_{n\ge 1}$ and $(N^T_\delta(X^n))_{n\ge 1}$. Following the proof of Theorem \ref{thm:5.18}, passing to a further subsequence, we have
	\begin{align*}
		\bigl(c^{-1}Z^n,\Gamma^n,X^n_{t_1},\ldots,X^n_{t_\ell}\bigr)
		\;\Rightarrow\;
		\bigl(\tilde Z,\tilde\Gamma,\tilde X_{t_1},\ldots,\tilde X_{t_\ell}\bigr)
	\end{align*}
	on \(( \D_\R[0,\infty), \dM )^2\times(\mathbb R^\ell,|\cdot|)\), for all $\ell\ge1$ and $t_1,\ldots,t_\ell$ in a cocountable set $D$ with $D\cap \operatorname{Disc}_{\Pro}(\tilde X)=\emptyset$, where $\tilde{X}$ is c\`adl\`ag, $\tilde \Gamma$ is continuous, and $\tilde Z$ has the same law as $Z$.
	
	We define $\tilde D$ from $D$ as in the proof of Theorem \ref{thm:5.18}. Now, $X^n$ and $\Gamma^n$ satisfy conditions (F1)--(F3) of \cite[Prop.~2.12]{andreasfabrice_theorypaper} and so does $\sigma(\bullet, X^n_{\bullet-r})$, with $c^{-1}Z^n$ as the integrator. Moreover, as shown in the proof of Theorem \ref{thm:5.18}, $(\sigma(\bullet, X^n_{\bullet-r}),c^{-1}Z^n)$ meet the conditions of Remark \ref{rem:weaker_integrand_conditions} with respect to Theorem \ref{thm:result_int_conv_CTRWs_with_independence_cond} while $(\Gamma^n,c^{-1}Z^n)$ satisfy the assumptions of Remark \ref{rem:weaker_integrand_conditions} with respect to Theorem \ref{thm:result_int_conv_CTRWs_lipschitz_integrands} (where the required \eqref{eq:oscillcond} conditions are immediate from the continuity of $\tilde\Gamma$ and criterion \ref{it:prop_suff_cond_AVCO_cont} of Proposition \ref{prop:3.3}). This ensures the control described in Section \ref{sec:quantities_and_approach_needed_for_integral_convergence_proof}. More precisely, take common nested deterministic partitions of vanishing mesh with positive endpoints in $\tilde D\setminus\operatorname{Disc}_{\mathbb P}(\sigma(\bullet,\tilde X_{\bullet-r}),\tilde Z)$. On each fixed partition the two left-endpoint sums form a single simple integral, which converges by \cite[Prop.~2.15]{andreasfabrice_theorypaper}. Applying \cite[Prop.~6.3]{andreasfabrice_theorypaper} and \eqref{eq:quantity3_integral_convergence_term_where_GD_needed} to each integrand and using the triangle inequality, we let the mesh vanish to obtain weak $M_1$ convergence of the stochastic integral in ($\operatorname{\tilde{SD}_n}$) to \(\int_0^\bullet (\sigma(s,\tilde X_{s-r})_-+\tilde\Gamma_s)\,\mathrm d\tilde Z_s\). From there, we can conclude as in the proof of Theorem \ref{thm:5.18}. Retaining $\Gamma^n$ and the evaluations of $X^n$ on $D$ gives $(X^n,\Gamma^n)\Rightarrow(\tilde X,\tilde\Gamma)$, where $\tilde X$ satisfies \eqref{eq:ext_SDDE}, with the driver $\tilde Z$ and with $\tilde\Gamma_t$ in place of $\tilde\sigma(t,\tilde X_{[t-r,t]})$. By continuity of $x\mapsto\tilde\sigma(\bullet,x_{[\bullet-r,\bullet]})$ from $( \D_\R[0,\infty), \dM )$ to $\C_\R[0,\infty)$ with the topology of uniform convergence on compacts, the joint weak limit satisfies $\tilde\Gamma=\tilde\sigma(\bullet,\tilde X_{[\bullet-r,\bullet]})$ almost surely. Thus $\tilde X$ solves \eqref{eq:ext_SDDE}.
\end{proof}

\appendix 

	\section{Limit theory for stochastic integrals}\label{appn}

	This brief appendix recalls two general results from \cite{andreasfabrice_theorypaper} concerning the functional weak convergence of stochastic integrals in the $J_1$ or $M_1$ topologies.

	\begin{theorem}[Weak continuity of stochastic integrals {\cite[Thm.~2.6]{andreasfabrice_theorypaper}}] \label{thm:3.19}
		Consider a sequence of semimartingales $(X^n)_{n\ge 1}$ with good decompositions \eqref{eq:Mn_An_condition} on filtered probability spaces $(\Omega^n, \mathcal{F}^n, \F^n, \Pro^n)$. Let $(H^n)_{n\ge 1}$ be any given sequence of adapted càdlàg processes for the same filtered probability spaces such that (i) there is joint weak convergence
		\begin{equation*}
			(H^n,X^n) \; \Rightarrow \; (H,X)  \quad \text{ on } \quad (\D_{\R^d}[0,\infty)\, , \, \tilde{\rho}\, )\times (\D_{\R^d}[0,\infty)\, , \,\rho)
		\end{equation*}
		with $\rho,\tilde{\rho} \in \{ \dM, \dJ \}$, for some càdlàg limits $H$ and $X$, and (ii) the pairs $(H^n, X^n)$ satisfy \eqref{eq:oscillcond}. Then, $X$ is a semimartingale in the filtration generated by the pair $(H,X)$ and
		\begin{equation}\label{eq:concerted_stoch_int_conv}
			\left( \, X^n, \, \int_0^\bullet \, H^n_{s-} \, \diff X_s^n\, \right) \;  \; \Rightarrow \; \left( \, X, \, \int_0^\bullet \, H_{s-} \, \diff X_s\, \right) \quad \text{ on } \quad (\D_{\R^{2d}}[0,\infty), \, \rho). 
		\end{equation}
	\end{theorem}
	
	We stress that, in \eqref{eq:concerted_stoch_int_conv}, the integrals are understood to be defined component-wise. This also allows one to handle matrix valued processes and `dot product' integrals, as per \cite[Rem.~2.10]{andreasfabrice_theorypaper}. Moreover, we note that, while the metric $\dM$ is understood to refer to the \emph{strong} version of the $M_1$ topology in the language of \cite[Sect.~12.3]{whitt}, the theorem also holds for the  \emph{weak} version of $M_1$ topology (see again \cite[Sect.~12.3]{whitt}), as explained in \cite[Rem.~2.9]{andreasfabrice_theorypaper}.
	
	The \eqref{eq:oscillcond} condition is rather abstract, so it is useful to have simple sufficient criteria.
	
	\begin{prop}[{\cite[Prop.~2.8]{andreasfabrice_theorypaper}}]\label{prop:3.3}
		In the setting of Theorem \ref{thm:3.19}, the condition \eqref{eq:oscillcond} is satisfied if one of the following two criteria holds:
		\begin{enumerate}
			\item The pairs $(H^n,X^n)$ converge weakly to $(H,X)$ on $(\D_{\mathbb{R}^{2d}}[0,\infty), \, \dJ)$. \label{it:prop_suff_cond_AVCO}
			\item \label{it:prop_suff_cond_AVCO_cont}The limiting processes $H$ and $X$ satisfy $\operatorname{Disc}(H)  \cap  \operatorname{Disc}(X)=\emptyset$ almost surely.
		\end{enumerate}
	\end{prop}

\noindent\textbf{Acknowledgements.} The work of FW was funded by the EPSRC grant [EP/S023925/1]. We are grateful to an anonymous referee and the AE for their careful reading and many detailed suggestions, which have greatly improved the structure and presentation of the paper.

\setlength{\bibitemsep}{2pt}
\printbibliography

\end{document}